\documentclass[pdflatex]{sn-jnl}
\usepackage{latexsym}
\usepackage{xfrac}
\usepackage{psfrag}
\usepackage{amsmath,amssymb,cite}
\usepackage{forest}
\forestset{
    whitenode/.style={draw,             circle, minimum size=0.5ex, inner sep=0pt},
    blacknode/.style={draw, fill=black, circle, minimum size=0.5ex, inner sep=0pt},
    colornode/.style={draw, fill=#1,    circle, minimum size=0.5ex, inner sep=0pt},
    colornode/.default={red}
}
\newcommand{\blankforrootedtree}{\rule{0pt}{0pt}}
\NewDocumentCommand\rootedtree{o}{\begin{forest}
    for tree={grow'=90, thick, edge=thick, l sep=0.5ex, l=0pt, s sep=0.5ex},
    delay={
      where content={}{
        for children={no edge, before drawing tree={for tree={y-=5pt}}}
      }
      {
        where content={o}{content={\blankforrootedtree}, whitenode}{
          where content={.}{content={\blankforrootedtree}, blacknode}{}
        }
      }
    }
    [#1]
\end{forest}}

\usepackage{latexsym}
\usepackage{graphicx}
\usepackage{subfigure}
\usepackage{lscape}
\usepackage{afterpage}
\usepackage{listings}
\usepackage[color=green!40]{todonotes}
\usepackage[normalem]{ulem}
\usepackage{dutchcal}
\usepackage{todonotes}

\usepackage{hyperref}
\hypersetup{
    colorlinks=true,
    linkcolor=blue,
    filecolor=magenta,
    urlcolor=cyan,
    citecolor=red,
}
\DeclareMathOperator*{\argmin}{argmin}
\DeclareMathOperator*{\sgn}{sgn}

\newcommand{\ds}{\displaystyle}
\newcommand{\nexto}{\kern -0.54em}
\newcommand{\dR}{{\rm {I\ \nexto R}}}

\newcommand{\dZ}{{\cal Z \kern -0.7em Z}}
\newcommand{\dC}{{\rm\hbox{C \kern-0.8em\raise0.2ex\hbox{\vrule
height5.4pt width0.7pt}}}}
\newcommand{\dQ}{{\rm\hbox{Q \kern-0.85em\raise0.25ex\hbox{\vrule
height5.4pt width0.7pt}}}}
\newcommand{\proofbox}{\hspace{\fill}{$\Box$}}
\newtheorem{lemma}{Lemma}
\newtheorem{theorem}{Theorem}

\newtheorem{proposition}{Proposition}

\newtheorem{remark}{Remark}

\renewenvironment{proof}{Proof.}{$\natural$}

\renewcommand{\theequation}{\arabic{equation}}
\renewcommand{\thepage}{{\bf\arabic{page}}}

\begin{document}


\author*[1]{\fnm{Robert M.} \sur{Corless}}\email{rcorless@uwo.ca}

\author[2]{\fnm{C. Yal{\c c}{\i}n} \sur{Kaya}}\email{ yalcin.kaya@adelaide.edu.au}
\equalcont{These authors contributed equally to this work.}
\affil*[1]{\orgdiv{Department of Computer Science}, \orgname{The University of Western Ontario}, \orgaddress{\city{London}, \postcode{N6A 5B9}, \state{Ontario}, \country{Canada}}}

\affil[2]{\orgdiv{ School of Mathematical Sciences}, \orgname{ Adelaide University}, \orgaddress{ \city{ Mawson Lakes}, \postcode{ 5095}, \state{ South Australia}, \country{ Australia}}}

\title{Minimizing Residuals in ODE Integration Using
Optimal Control}

\date{\today}

\maketitle

\thispagestyle{empty}

\begin{abstract}

Given the set of discrete solution points or nodes, called the skeleton, generated by an ODE solver, we study the  problem of fitting a curve passing through the nodes in the skeleton minimizing a norm of the residual vector of the ODE.  We reformulate this interpolation problem as a multi-stage optimal control problem and, for the minimization of two different norms, we apply the associated maximum principle to obtain the necessary conditions of optimality.  We solve the problem analytically for the Dahlquist test problem and a variant of the leaky bucket problem, in terms of the given skeleton.  We also consider the Van
der Pol equation, for which we obtain interpolating curves with minimal residual norms by numerically solving a direct discretization of the problem through optimization software.  With the skeletons obtained by various ODE solvers of {\sc Matlab}, we make comparisons between the residuals obtained by our approach and those obtained by the {\sc Matlab} function deval.
\end{abstract}




\section{Introduction}
\subsection{Background}
The numerical solution of differential equations, both initial-value problems and boundary-value problems, is by now a very old subject.  In particular, the subject underwent significant development in the 1960's and 1970's as computers gained in speed and power.
Much of that development has been captured in the classic texts by Hairer, N{\o}rsett and Wanner~\cite{HaiNorWan1993}, Hairer and Wanner~\cite{HaiWan1996}, Butcher~\cite{Butcher1987}, and for boundary-value problems, Ascher, Mattheij and Russell~\cite{AscMatRus1995}; not forgetting the earlier work by Stetter~\cite{Stetter1973}; or the myriads of other texts.

Yet some pieces of that development were not captured in those texts, or not wholly, or if they were, they may lie undisturbed as the texts (such as Stetter's) fall out of use.  An example of that is the relationship of {\em defect} (what we will call the {\em residual} in this paper, following Corless and Fillion~\cite{CorFil2013}) to the ``local error per unit step," which is mentioned in Stetter's book but nowhere else (outside of, say, Wayne Enright's lecture notes~\cite{Enright}) that we know of.

Another item of interest is that some design decisions for the methods we use today were taken in a time when computer memory was much more constrained than it is now, and moreover were taken at a time when the relationship between forward error (called ``global error" in the numerical ODE community) and backward error (e.g., the residual) was not as widely understood as perhaps it ought to have been.

Another factor is that modern codes use heuristics for variable stepsize and variable order, and the heuristics and controls for stepsize introduce nonlinearities and feedback into the solution.  There is a tension involved: the error control uses theoretical results that are asymptotically correct as the stepsize~$h$ goes to zero, but for efficiency the code is designed to use stepsizes~$h$ that are as large as possible, consistent with the tolerances.  It is possible, especially for loose tolerances, that the solution is not as accurate as the theoretical model underlying the code predicts that it will be.  Even linear ODE can have numerical solutions whose behaviour is difficult to predict, or even to explain afterwards.  One important motive for this paper is to provide another set of tools for ``retrospective diagnostics'' of what the code actually solved.

A final motive for this paper and perhaps the most salient one is a question of validation, reproducibility, of what we call ``computational epistemology.''  Just because a numerical method and its implementation has been tested on problems from test sets\footnote{Larry Shampine argues in~\cite{Shampine1986} that a ``thoughtful consideration of truncation error coefficients provides more detail and more reliable conclusions than does the study of a battery of numerical tests$\ldots$'' and, while we agree, we point out that most users of modern software simply do not have the expertise to do this consideration themselves. We wish to provide simpler tools, for general use.}, say, A, B, C and D, it does not necessarily follow that the methods will work on problems from classes E, F, or G, or, say, Z.  Yet people rely on the results of widely---and freely---available solvers, sometimes blithely, and sometimes blindly.  We are aware of published papers containing incorrect results obtained with such na{\"\i}ve reliance---some examples are discussed in \cite{ShaRei1997}.  Given the ubiquity of ODE solvers nowadays, what can a skeptical user do to address the validity of the computed numerical solution?
\subsection{A motivating example: The Lorenz system}
In the seminal and magisterial 1963 paper ``Deterministic Nonperiodic Flow'' E.~N.~Lorenz showed that numerical methods applied to the system
\begin{align}
    \frac{dx}{dt} &= \sigma(y-x) \nonumber\\
    \frac{dy}{dt} &= x(\rho-z) - y \nonumber\\
    \frac{dz}{dt} &= xy - \beta z \label{eq:Lorenz}
\end{align}
with an initial condition such as $x(0)=z(0)=0$ and $y(0)=1$ demonstrated bounded nonperiodic solutions for certain parameter values~\cite{Lorenz1963}.  His numerical work was careful and rigorous, using a Lyapunov function approach to demonstrate that his chosen numerical method gave bounded solutions for all small enough stepsizes.

The method he chose is known by various names: he called it the double-approximation procedure, but it is also known as Heun's method, the explicit trapezoidal rule, the improved Euler method, or the modified Euler method.  It is one of a class of 2nd order Runge--Kutta methods, and more than a hundred years old at this time of writing; when Lorenz used it, it was more than fifty years old.

But as Lorenz effectively established and is now very well-known, the Lorenz equations have \textsl{chaotic} solutions~\cite{Sparrow:1982,tucker2002rigorous}. Standard theorems which rely on Lipschitz constants to guarantee that the output of numerical methods for ODE will converge to the correct solution on a given interval $[0,T]$ as the timestep goes to zero, fail in the chaotic context because the solutions are exponentially sensitive to changes in the initial conditions and therefore also sensitive to truncation errors (as pointed out by~\cite{Young1966} for the Lorenz system) and even to the much smaller rounding errors induced by computation in floating-point arithmetic.  Errors grow like $\exp(L t)$ where $L$ is the Lipschitz constant for the problem.

Nevertheless, Lorenz was right, and his numerical simulations were entirely reliable.  One explanation for that reliability was given in~\cite{Corless1994} and more recently in~\cite{corless2025numerical}, and is that a good numerical method gives you the exact solution of a nearby problem. One then has to worry about the effects of changes to the problem, but then one \textsl{always} has to worry about that, and tools for assessing the \textsl{sensitivity} of solutions of ODE are well-known. In this case the \textsl{trajectories} are very sensitive, but the attracting set itself is not. In fact, the attracting set is remarkably stable to persistent disturbances of the Lorenz system.

If we write the Lorenz system in vector form as $\dot y = F_f(y)$ then the trapezoidal rule gives the solution of
\begin{equation}
    \dot{z} = F_f(z) + r(t)
\end{equation}
for some \textsl{residual} $r(t)$; a more careful analysis shows that it gives the solution of the \textsl{modified equation}
\begin{equation}
     \dot{z} = F_{f}\mathopen{}\left( \rootedtree[.] \right)\mathclose{} + \frac{1}{3} h^{2} F_{f}\mathopen{}\left( \rootedtree[.[.[.]]] \right)\mathclose{} + \frac{1}{12} h^{2} F_{f}\mathopen{}\left( \rootedtree[.[.][.]] \right)\mathclose{} + \delta(t) \label{eq:modifiedLorenz}
\end{equation}
for some other residual $\delta(t)$ which will, asymptotically as the time-step $h \to 0+$, be $O(h^3)$ in size.

The (perhaps unfamiliar) notation using rooted trees indicates certain partial derivatives of the vector function $F_f$, which we do not need to understand in detail for the purposes of this paper, although we will use it again later. For an explanation of this notation, see e.g.~\cite{butcher2021b}, or~\cite[pp.~144--150]{HaiNorWan1993}.

This present paper focuses on the residuals $r(t)$ and $\delta(t)$ in an effort to obtain guarantees about their smallness.  The idea is that the smaller the residual is, the more confident one is in the quality of the numerical solution, and the more one is able to rely on it for analysis.
\subsection{Setting}
We believe that one tool that can be helpful is explicit independent computation of the residual (also known as the \textsl{defect}, especially in the older literature).  If our differential equation is $\dot y = f(t,y)$ where $y: t \mapsto \mathbb{R}^N$ and $f: (t,y) \mapsto \mathbb{R}^N$,
 then a putative solution to the differential equation (call it $z(t)$) would have a \textsl{residual} $r(t)$ defined at all but exceptional points by
\begin{equation}
    r(t) := \dot z(t) - f(t, z(t)) \label{eq:residual}
\end{equation}
which trivially rearranges to show that $z(t)$ is the \textsl{exact} solution of the perturbed equation
\begin{equation}
    \dot z(t) = f(t,z(t)) + r(t)\>. \label{eq:perturbed}
\end{equation}

This can then often be {\em interpreted} as a meaningful perturbation of the model equations.  If this is small compared to real effects that have been neglected, then the computed solution can be regarded as valid.  The effects of neglecting small perturbations---real, or numerical---of course have to be understood, perhaps by using variational analysis, the theory of conditioning of ODE, or even the Gr{\"o}bner--Alexeev nonlinear variation of constants formula~\cite[p.~98]{HaiNorWan1993}, but that has to be done {\em anyway}.  For reference, the Gr{\"o}bner--Alexeev nonlinear variation of constants formula is, assuming that the initial condition $y(t_0)=y_0$ is not perturbed,
\begin{equation}
    y(t)-z(t) = \int_{\tau=t_0}^t G(t,\tau)r(\tau)\,d\tau \>. \label{eq:GroebnerAlexeev}
\end{equation}
The function $G(t,\tau)$, whose existence the theorem guarantees, is essentially $\partial y /\partial y_0$, the dependence of the solution on the initial conditions.  This acts as a kind of ``condition number'' for the differential equation.  If $G$ is large, then even a small residual can have a dramatic effect on the solution of the differential equation.  This formula also shows that if the residual is $O(h^p)$ in size as $h \to 0$, then so is the forward (or ``global'') error $O(h^p)$, with a constant that is changed by the size of $G$ (hence the interpretation of $G$ as a kind of ``condition number'').

An alternative (and equivalent) method that is frequently used (see, e.g.~\cite{calver2017numerical} and the references therein) is to simultaneously solve the so-called \textsl{adjoint equations} $\dot{\xi} + J_f^T(y) \xi = 0$ and use the growth of the vector $\xi$ to estimate how sensitive the solutions of the ODE are to changes in its initial conditions or to persistent disturbances.

Notice that the Gr\"obner--Alexeev formula and the adjoint equations also allows to predict the effect of physical disturbances, not just the effect of approximate computation.
By putting the approximations of the numerical method on the same footing as that of meaningful modelling choices, we will have succeeded in validating the numerical computation. This is in contrast with most standard textbook explanations of why numerical methods work, which use Lipschitz constants or one-sided Lipschitz constants to prove that the solutions will converge to the reference solution on $[0,T]$ as the tolerance (or stepsize) goes to zero.  As we saw with the Lorenz equations, such an explanation can fail in the chaotic case because the numerical solutions would not converge until the stepsize was exponentially small, which would make the solution unaffordable.  And yet Lorenz' numerical solution using $h=0.01$ was perfectly valid.  One would like to have a ready explanation for this success.

This paper advocates for an independent method to give such an explanation.  Given a putative skeleton of a numerical solution to an ODE, we ask for the {\em minimal perturbation} of the ODE that would fit the numerical skeleton.  If this minimal solution is ``too large" in size, we reject the solution.  If it is ``small enough," we accept it.  These terms are loose and depend on the problem context and its structure, and so we advocate giving tools to the users to allow them to decide for themselves.  In this paper, we formulate several such perturbations and show how to use techniques from optimal control theory to find minimal ones. Knowing the minimal possible residual gives access to \textsl{sufficient} conditions for accepting the numerical solution as valid.

It turns out that the problem formulation, and the problem scaling, have a great impact on the interpretation of the residual.  We presume in this paper that a proper scaling has been done, and moreover that the ODE is reasonable to pose as a system of first order equations.  Most codes require higher-order systems to be rewritten as a first-order system\footnote{A referee reminded us that we had forgotten the important case of second-order systems, where special methods such as Runge--Kutta--Nystr\"om methods are available~\cite[pp.~260--275]{HaiNorWan1993}.}, typically using the standard chain of new variables $y_1 = y$, $y_2 = y_1'$, $y_3 = y_2'$, and so on; in these equations, if they are artificial and introduced only to make an existing code applicable, a residual or defect can be quite hard to interpret, because the standard chain equations do not really permit perturbation.  For the purposes of this paper, we ignore this concern. In a future paper we examine optimal residuals for higher-order equation systems.
Interestingly, the issue of the numerical difference between solving an ODE rewritten as a first order system versus solving it in its original higher-order formulation is discussed only, so far as we know, in Section 11.1 of the classic text~\cite{AscMatRus1995} and nowhere else in the literature; but we leave it for now.

The problem of our interest can be viewed as a problem of interpolation, that is, the problem of finding (in some sense) the ``best solution curve'' that passes through the nodes in the numerical skeleton computed by some ODE solver.  Most modern codes choose to use \textsl{polynomial} approximations valid between the nodes of the skeleton, which are designed to be ``asymptotically correct,'' in the limit as the stepsize $h$ goes to zero.  Typically we would want the approximation to an $O(h^p)$ accurate numerical solution to \textsl{also} be $O(h^p)$ accurate, for ``small enough'' $h$.

But for efficiency, one wants $h$ to be ``as large as possible'' consistent with the error tolerances, and we therefore cannot be sure that the code is operating ``in the asymptotic regime'' and indeed there are documented cases where it was not.  There is also the phenomenon of ``order reduction'' (see \cite{HaiWan1996}) to consider, in which the observed order of convergence as $h$ goes to zero is smaller than the expected $p$th order.  The reasons for that reduction in any given problem can be quite obscure, and even specialists can be confused about it.  We would like to provide tools for retrospective diagnostics that can reassure the user that the code, even if not operating in the asymptotic regime, and even if suffering from order reduction, still produced a good solution. To do this we will relax the idea of using piecewise polynomial approximation.

We will, however, tighten the notion of ``approximation'' to be that of ``interpolation.'' It turns out that not all polynomial approximations supplied by modern codes actually interpolate the skeleton.  Some do, for example the Matlab code \texttt{ode45}, which is supplied with a piecewise polynomial interpolant that is actually continuously differentiable (at least in the absence of rounding error).  Others are not: for instance, the Matlab code \texttt{ode113} comes with a piecewise polynomial approximation that is continuously connected to the skeleton only at one node of two (the right endpoint). The lack of continuity at the left is not a difficulty if all one wants from the approximation is dense output for a graph, and indeed the piecewise polynomial is ``almost'' continuous there, having a jump discontinuity of size similar to the tolerance supplied to the integration.  But for other uses, this lack of continuity is, as Shampine says in~\cite{Shampine1986}, a serious drawback.

\subsection{Contributions}
We will insist, in this paper, on a piecewise interpolant that is continuous at every interior node of the skeleton.  It will turn out that we cannot insist on the interpolant being continuously differentiable\footnote{Continuously differentiable interpolants are desirable for several purposes, and sometimes essential, and this is emphasized in~\cite{shampine1985interpolation} and again in~\cite{EnrJacNorTho1986}. However, in this paper we give up $C^1$ continuity to gain minimality, which allows us to give sufficient conditions for accuracy.  There are yet other criteria for ``good'' interpolation, such as preserving monotonicity or convexity to ensure a pleasing visual appearance, and indeed a ``not misleading'' visual appearance, but those considerations are also less important for the purposes of this paper.}, but the $C^0$ continuity will turn out to be crucial.

The reason that we insist our piecewise interpolant be globally $C^0$ is that if we permit discontinuous approximate solutions, then approximation by ``local exact solutions'' will have residual $r(t)$ that is identically zero.  That is, if we approximate $y(t)$ by the piecewise-defined function $y(t) = y_k(t)$ on $t_k \le t \le t_{k+1}$ with $\dot y_k(t)=f(t,y_k(t))$ and either $y(t_k) = y_k$ or, perhaps, $y(t_{k+1}) = y_{k+1}$ at the other end, then $r(t)$ will be zero in $(t_k,t_{k+1})$ but a delta function at one end.  Putting that $r(t)$ into equation~\eqref{eq:GroebnerAlexeev} produces a sequence of jumps at the node; this is a classical way of relating the ``local error'' of the method with the ``global error'' in the computed skeleton.

Instead we insist on a globally continuous interpolant to the skeleton, and ask for the interpolant that gives the ``minimal residual'' in various senses.

The interpolant in our case is required to minimize a measure of the residual; in particular, we choose to minimize the $L^2$-norm or some ``convenient'' variant of the $L^\infty$-norm of the residual, resulting in two variational interpolation problems.  The particular variant mentioned here is simply the sum of the $L^\infty$-norms taken over
each {\em stage} (defined as the interval between any two consecutive nodes\footnote{This is a natural usage for the word ``stage'' in the optimal control literature.  Later in this paper we also use the word ``stage'' in a different sense, one more common in the Runge--Kutta community.  This collision of meanings, this polysemy, is unfortunate, and may cause confusion.}).  This norm variant helps in getting analytical solutions for some relatively simple example ODEs that we consider in this paper.  We call this variant the {\em stage $L^\infty$-norm}.

It is well known that  {\color{black} interpolation} problems can be reformulated as an optimal control problem, and a specialized maximum principle (for multi-stage or multi-process optimal control) can be employed with the aim of finding a solution; see, for example, \cite{KayNoa2013, Kaya2019} and the references therein. We subsequently establish in Lemmas~\ref{L2_lambda0} and \ref{Linf_lambda0} that the respective optimal control problems we derive in the current paper are normal, and in Theorems~\ref{theo:L2} and \ref{theo:Linf} we present the resulting necessary conditions of optimality.

The Dahlquist test problem is conceivably the simplest possible and yet rich enough ODE that is widely used to analyze certain properties of ODE solvers, such as stability; see, for example, \cite{CorKayMoi2019} and the references therein.  We obtain closed-form expressions for the optimal residuals of the Dahlquist test problem in Propositions~\ref{simp_fact_L2} and \ref{simp_fact_Linf}.  These expressions also allow us to make comparisons between the two different norms of the residual in Proposition~\ref{prop:relation}.

We consider another relatively simple but sufficiently interesting ODE, which is a variant of the celebrated ``leaky bucket problem'' (see \cite{CorJan2016}), and obtain analytical expressions for the optimal residuals in Propositions~\ref{prop:sqrt_L2} and \ref{prop:sqrt_Linfty}.

The third example we study is the Van der Pol system, which is not tractable analytically; however, we discretize the ensuing optimal control problem directly~\cite{Betts2020} and numerically solve the resulting large-scale optimization problem using the mathematical programming modelling language AMPL~\cite{AMPL} paired with the popular optimization software Knitro~\cite{Knitro}.  We gain some insight from the numerical solutions for the optimal residual.  In each of the three examples mentioned above, we use a suite of {\sc Matlab}'s ODE solvers~\cite{Matlab} to construct the numerical skeleton.

The paper is organized as follows.  In Section~\ref{sec:motivation} we amplify our motivational remarks above by a detailed consideration of a single example.  In Section~\ref{sec:formulation}, we formulate the problem of finding interpolants minimizing the $L^2$- and stage $L^\infty$-norms and transform these problems into optimal control problems.  In Section~\ref{sec:nec_cond}, we apply a maximum principle to obtain the necessary conditions of optimality, after showing that both optimal control problems are normal.  We study the three example ODEs in Section~\ref{sec:examples}, also presenting numerical experiments. We briefly discuss the experiments in Section~\ref{sec:discussion}. Finally, in Section~\ref{sec:conclusion}, we provide some concluding remarks.

\section{Motivation}\label{sec:motivation}
\begin{quote}
  ``$\ldots$ in theory, there is no difference between theory and practice,
    while in practice, there is.''\\
     ---Benjamin Brewster, \textsl{The Yale Literary Magazine}
    October 1881/June 1882 issue
\end{quote}
One motivation for this paper is that the polynomial approximations extending the skeleton to ``continuous'' functions used in current codes can behave suboptimally. We demonstrate this with an example, the numerical solution of the simple harmonic oscillator $\ddot y + y = 0$ on a representative time interval with representative initial conditions.

Consider the following piece of {\sc Matlab} code, which solves the simple harmonic oscillator $\ddot y + y = 0$ on $0 \le t \le 30$ with $y(0)=3.5$ and $\dot y(0) = 0$ using many different tolerances, and compares the maximum residual error (which we define in Section~\ref{sec:formulation}, but note that this is \textsl{not} the forward or global error, which would require us to know the reference solution to the equation, which we do not use here) to the \textsl{mean} stepsize that the code takes.  The mean stepsize tells us something directly about the amount of work that the code does, because the work is proportional to the number of steps taken. Note that in the papers~\cite{corless2000elementary,corless2002new,ilie2008adaptivity} we find an explanation of the experimentally observed fact that modern stepsize heuristics, which approximately equidistribute the local error, can usually be expected to produce solutions which are accurate to $O(h^p)$ where $h$ is the arithmetic mean of all the stepsizes taken and $p$ is the order of the method.  Most texts say that it is the \textsl{maximum} stepsize which must be used, but this is not true: for the codes to be asymptotically correct, only the mean stepsize needs to go to zero, given the assumption of equidistribution.  Therefore, this script will produce a graph, analogous to a ``work-precision diagram,'' which will tell us something about the behaviour of the error in solving the simple harmonic oscillator with \texttt{ode45}, and in particular will tell us how the polynomial interpolant behaves.

The results are plotted in Figure~\ref{fig:shoode45}.  While a na{\"\i}ve user might have expected that the observed error would decrease as the mean stepsize decreases, over the whole range of tolerances, we see hints that for more complicated problems this might not really be so for large tolerances, only for ``small enough'' tolerances.  We also see that the order of accuracy is not $O(h^5)$ (the red curve) but rather seems to be reduced to $O(h^4)$ (the blue curve), once the error starts to decrease.  Neither of these behaviours is likely to be expected by a na{\"\i}ve user.

The behaviour of modern codes is not always simple, and therefore we believe we should provide more tools for retrospective diagnostics.

In fairness to \texttt{ode45}, we note that it was not designed to achieve small \textsl{residual} error, but instead to have small \textsl{local} error for modest tolerances.  The two notions are related, but not identical. We remark that residual error is easier to explain to students and users of the code, and is useful in other contexts as well, while the concept of ``local error'' is useful only in the context of historical codes for the solution of ODEs.
\begin{figure}
    \centering
    \includegraphics[width=0.6\linewidth]{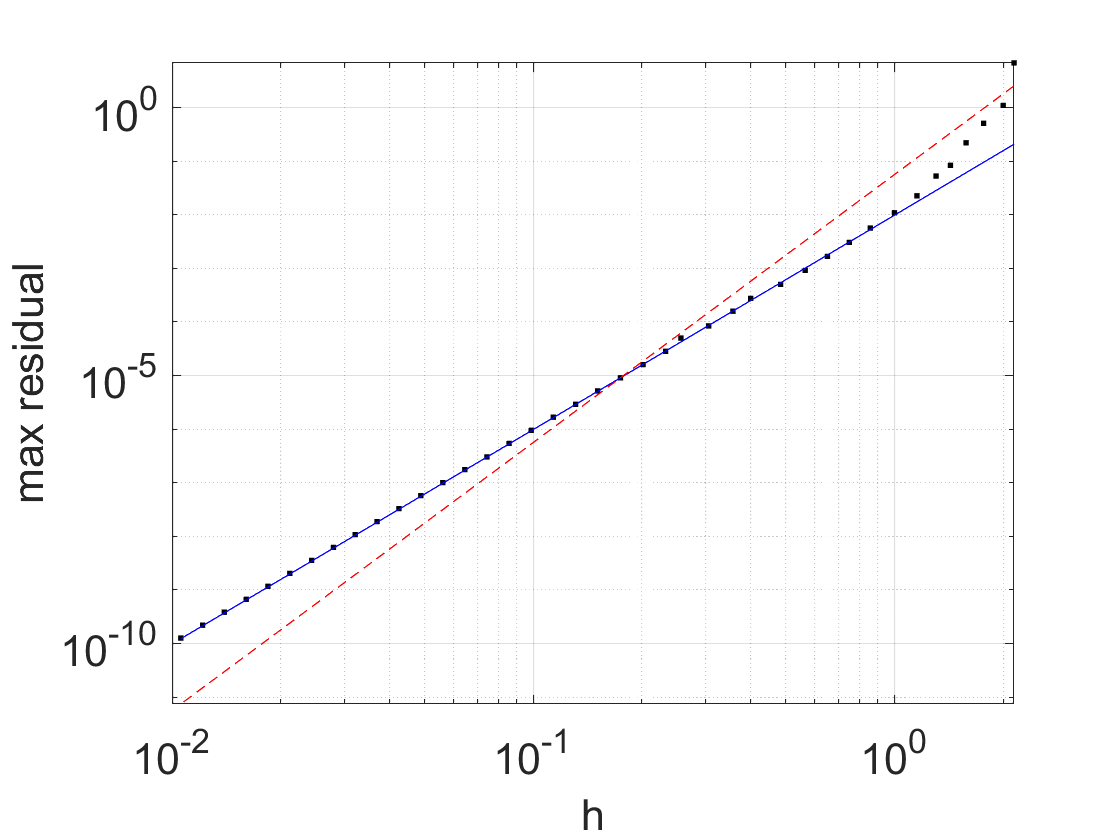}
    \caption{A kind of work-precision diagram for the solution of the simple harmonic oscillator $\ddot y + y = 0$ on $0 \le t \le 30$ with reasonable initial conditions, given to \texttt{ode45} in the form of a first-order system as usual.  On the horizontal axis we have the mean stepsize $h$. The code \texttt{ode45} is a variable-stepsize method, and the stepsize at the tightest tolerance varied systematically between $9\times 10^{-3}$ and $1.2\times 10^{-2}$.  Smaller mean~$h$ implies that the code had to work harder to get to $t=30$.  On the vertical axis, we have the maximum residual norm observed over the solution on $0 \le t \le 30$.  We see that the expected order of the error, namely $O(h^5)$, plotted in red, is not attained; instead, the observed order of the residual error is very clearly $O(h^4)$ (plotted in blue).}
    \label{fig:shoode45}
\end{figure}

\begin{lstlisting}
%
% estimate the order of the residual error by Matlab's
% built-in interpolant from ode45
%
nsamp = 40; % 2^(-40) is about 1.0e-12
err = zeros(1,nsamp);
h = zeros(1,nsamp);
for k=1:nsamp
    disp(k)
    opts = odeset('RelTol',2.^(-k),'AbsTol',2.^(-k-1));
    sol = ode45( @sho, [0, 30], [3.5;0], opts);
    h(k) = mean(diff(sol.x));
    t = RefineMesh(sol.x, 8);
    [y,dy] = deval(sol,t);
    [dx,nt] = size(t);
    res = zeros(size(y));
    for j=1:nt
      res(:,j) = dy(:,j)-sho(t(j),y(:,j));
    end
    err(k) = max(abs(res(:)));
end
k2 = floor(nsamp/2);
const4 = err(k2)/h(k2)^4;
const5 = err(k2)/h(k2)^5;
figure(1), loglog( h, err, 'k.', ...
    h, const4*h.^4, 'b-', ...
    h, const5*h.^5, 'r--')
xlabel('h'), ylabel('max residual')
\end{lstlisting}
\par\noindent
The function \lstinline{sho} is
\begin{lstlisting}
function dydt = sho(t,y)
  dydt = [y(2); -y(1)];
end
\end{lstlisting}
and for the convenience of the reader, the \lstinline{RefineMesh} routine from~\cite{CorFil2013} is available at
\url{http://nfillion.com/coderepository/Graduate_Introduction_to_Numerical_Methods/RefineMesh.m}.
The question arises, then, what should we do in order to understand what the numerical solver has done?  We advocate in this paper that we should not rely solely on the interpolants provided by the code (which are not independent of the method used) but instead do a separate assessment of the ``\textsl{minimal possible}'' maximum residual.
\section{Interpolants Minimizing a Measure of the Residual}
\label{sec:formulation}

The problem is to find an interpolating curve $x:[t_0,t_f]\to\dR^n$ passing through a set of points $z_i$ in $\dR^n$, called the {\em skeleton}, generated at $t=t_i$, $i = 0,\ldots,N$, where $t_0 < t_1 < \cdots < t_N := t_f$, using an ODE solver for the initial value problem $\dot z(t) = f(z(t),t)$,\ \ $z(t_0) = z_0$.  We assume that $z_i \neq z_{i-1}$, $i = 1,\ldots,N$.  The interpolating curve is required to minimize some measure, for example, the $L^2$- or the $L^\infty$-norm, of the {\em residual} $(\dot{z}(t) - f(z(t),t))$.

\subsection{Reformulations of the minimum residual problems}

The problem of {\em $L^2$-minimization}, or minimization of the square $L^2$-norm, of the residual by a curve $x(t)$ interpolating the skeleton $\{z_0,z_1,\ldots,z_N\}$ can be written as
\[
\mbox{(PL2) }\left\{\begin{array}{rl}
\ds\min & \ \ \ds\frac{1}{2}\,\|\dot x(t) - f(x(t),t)\|_{L^2}^2 = \frac{1}{2}\,\sum_{i=1}^{N} \int_{t_{i-1}}^{t_i} \|\dot x(t) - f(x(t),t)\|^2\,dt \\[6mm]
\mbox{subject to} & \ \ x(t_i) = z_i\,,\ \ i = 0,\ldots,N\,,
\end{array} \right.
\]
where $\|\cdot\|$ is the Euclidean norm in $\dR^n$.  Define the {\em control variable} $u:[t_0,t_f]\to\dR^n$ such that $u(t):=\dot{x}(t)-f(x(t),t)$.  Now Problem~(PL2) can  equivalently be re-written as an optimal control problem as follows.
\[
\mbox{(PL2a) }\left\{\begin{array}{rl}
\ds\min & \ \ \ds\frac{1}{2}\,\sum_{i=1}^{N} \int_{t_{i-1}}^{t_i} \|u(t)\|^2\,dt \\[5mm]
\mbox{subject to} & \ \ \dot{x}(t) = f(x(t),t) + u(t)\,,\mbox{ a.e. } t\in[t_0,t_f]\,, \\[2mm]
& \ \ x(t_i) = z_i\,,\ \ i = 0,\ldots,N\,,
\end{array} \right.
\]
where $x(t)$ is referred to as the {\em state variable}. We note that an optimal control~$u$ can be indeterminate at isolated points in $[t_0,t_f]$; therefore, the ODE in Problem~(PL2a) is stated for a.e. $t\in[t_0,t_f]$.  
We call the {\em time interval $[t_{i-1}, t_i]$}, $i = 1,\ldots, N$, the {\em $i$th stage}.

As can be seen in Problem~(PL2), the square $L^2$-norm can be written as the sum of the square $L^2$-norms in each stage.  Subsequently, we define the problem of {\em stage $L^\infty$-minimization}, or minimization of the {\em stage} $L^\infty$-norm, of the residual by defining the objective functional (similarly) as the sum of the $L^\infty$-norms in each stage, since this will eventually allow us to obtain closed-form expressions for the optimal residuals of some of the example problems we investigate.  We point out that this sum of the $L^\infty$-norms in stages $[t_{i-1}, t_i]$ is obviously different from the (classical) $L^\infty$-norm over the whole interval $[t_0, t_f]$, with which it is not possible to obtain any analytical solution for the optimal residual, and we leave this outside the scope of the present paper.  So, we alter Problem~(PL2a) slightly to get the stage $L^\infty$-minimization problem as follows.
\[
\mbox{(PLinf) }\left\{\begin{array}{rl}
\ds\min & \ \ \ds\sum_{i=1}^{N}\,\max_{t_{i-1}\le t\le t_i} \|u(t)\|_\infty \\[5mm]
\mbox{subject to} & \ \ \dot{x}(t) = f(x(t),t) + u(t)\,,\mbox{ a.e. } t\in[t_0,t_f]\,, \\[2mm]
& \ \ x(t_i) = z_i,\ \ i = 0,\ldots,N\,,
\end{array} \right.
\]
where $\|\cdot\|_\infty$ is the $\ell_\infty$-norm in $\dR^n$.  Now, using a standard reformulation in mathematical programming (see, for example, \cite[p.~307]{NocWri2006} and \cite{KayNoa2013}), Problem~(PLinf) can be re-cast as
\[
\mbox{(PLinfa) }\left\{\begin{array}{rl}
\ds\min & \ \ \ds\sum_{i=1}^{N} \alpha^{[i]} \\[5mm]
\mbox{subject to} & \ \ \dot{x}(t) = f(x(t),t) + u(t)\,,\mbox{ a.e. } t\in[t_0,t_f]\,,\\[2mm]
& \ \ x(t_i) = z_i,\ \ i = 0,\ldots,N\,, \\[2mm]
& \ \ \|u(t)\|_\infty \le\alpha^{[i]}\,,\mbox{ a.e. } t\in[t_{i-1},t_i)\,,\ \ i = 1,\ldots,N\,,
\end{array} \right.
\]
where $\alpha^{[i]} \ge 0$ are new parameters that are to be determined for each stage $i$.  Note that, for a.e.\ $t\in[t_{i-1},t_i]$, $\|u(t)\|_\infty \le\alpha$ if, and only if, $|u_j(t)| \le\alpha^{[i]}$, $i = 1,\ldots,N$, $j = 1,\ldots,n$.  Let $u := \alpha^{[i]}\,v$, for $t\in[t_{i-1},t_i)$, $i = 1,\ldots,N$, where $v:[t_0,t_f]\to\dR^n$ is a new control variable.  Then Problem~(PLinfa) can be rewritten as the following differentiable, or smooth, parametric optimal control problem.
\[
\mbox{(PLinfb) }\left\{\begin{array}{rl}
\ds\min & \ \ \ds\sum_{i=1}^{N} \alpha^{[i]} \\[5mm]
\mbox{subject to} & \ \ \dot{x}(t) = f(x(t),t) + \alpha^{[i]}\,v(t)\,,\mbox{ a.e. } t\in[t_{i-1},t_i)\,,\ i = 1,\ldots,N\,, \\[2mm]
& \ \ x(t_i) = z_i\,,\ \ i = 0,\ldots,N\,, \\[2mm]
& \ \ |v_j(t)| \le 1\,,\mbox{ a.e. } t\in[t_0,t_f]\,,\ \ j = 1,\ldots,n\,.
\end{array} \right.
\]
Problem~(PLinfb) can further be transformed into a more standard or classical form by defining a new state variable, $y(t) := \alpha^{[i]}$, for all $t\in[t_{i-1},t_i)$, $i = 1,\ldots,N$, as follows.
\[
\mbox{(PLinfc) }\left\{\begin{array}{rl}
\ds\min & \ \ds\int_{t_0}^{t_f} y(t)\,dt \\[5mm]
\mbox{subject to} & \ \ \dot{x}(t) = f(x(t),t) + y(t)\,v(t)\,,\mbox{ a.e. } t\in[t_{i-1},t_i)\,,\ i = 1,\ldots,N\,, \\[2mm]
& \ \ x(t_i) = z_i\,,\ \ i = 0,\ldots,N\,, \\[2mm]
& \ \ \dot{y}(t) = 0\,, \\[2mm]
& \ \ |v_j(t)| \le 1\,,\mbox{ a.e. } t\in[t_0,t_f]\,,\ \ j = 1,\ldots,n\,.
\end{array} \right.
\]

\section{Necessary conditions of optimality}
\label{sec:nec_cond}
Both Problems~(PL2a) and (PLinfc) are optimal control problems with {\em point-state constraints}, $x(t_i) = z_i$, $i = 1,\ldots,N-1$, given in the interior of $[t_0,t_f]$.  They can be further transformed into multi-process, or multi-stage, optimal control problems, in a fashion similar to that done for the interpolation problems in~\cite{KayNoa2013,Kaya2019}{\color{black}, as well as for density estimation in~\cite{HegKay2025}}.  The necessary conditions of optimality for general multi-process or multi-stage optimal control problems can be found in~\cite{AugMau2000, ClaVin1989, DmiKag2011}.

We first define a new independent variable $s$ in terms of $t$ to map the ``duration" $\tau_i$ of each {\em stage} $i$ to unity:
\[
t = t_{i-1} + s\,\tau_i\,,\quad s\in[0,1]\,,\quad \tau_i := t_i - t_{i-1}\,,\quad i = 1,\ldots,N\,.
\]
With this definition, the {\em time horizon} of each stage $i$ is rescaled as $[0,1]$ in the new independent (time) variable $s$.  Note that $\tau_i > 0$, $i=1,\ldots,N$.  Let
\[
x^{[i]}(s) := x(t)\,,\ y^{[i]}(s) := y(t)\,\ \mbox{and}\ u^{[i]}(s) := u(t)\,,\ \mbox{for } s\in[0,1],\ t\in[t_{i-1},t_{i}]\,,
\]
for $i = 1,\ldots,N$.  Here $x^{[i]}(s) := (x^{[i]}_1(s),\ldots,x^{[i]}_n(s))\in\dR^n$ denotes the value of the state variable vector $x(t)$ in stage $i$, and other stage variables are to be interpreted similarly. With the usage of stages, it is necessary to ensure continuity of the state variables $x(t)$ at the junction of any two consecutive stages, but this is readily ascertained since $x^{[i]}(1) = x^{i+1}(0) = z_i$, $i = 1,\ldots,N-1$.

\subsection{Necessary conditions of optimality for \boldmath{$L^2$}-minimization}
\label{sec:L2-min}

Using the definitions above, a multi-stage optimal control problem associated with~(PL2a) can now be written as
\[
\mbox{(PL2b)}\left\{\begin{array}{rl}
\min &\ \ \ds\frac{1}{2}\,\sum_{i=1}^{N}\tau_i\int_{0}^{1} \|u^{[i]}(s)\|^2\,ds \\[6mm]
\mbox{s.t.} &\ \ \dot{x}^{[i]}(s) = \tau_i \left(f(x^{[i]}(s),s) + u^{[i]}(s)\right),\mbox{ a.e. } s\in[0,1]\,, \\[2mm]
 &\ \  x^{[i]}(0) = z_{i-1}\,,\ x^{[i]}(1) = z_i\,, \ \ \ i = 1,\ldots,N\,.
\end{array}\right.
\]

In what follows, we will state a maximum principle, i.e., necessary conditions of optimality, for Problem~(PL2b), using \cite[Theorem~3.1 and Corollary~3.1]{ClaVin1989} and \cite[Section~4]{AugMau2000}.  A relevant maximum principle can also be found in~\cite{DmiKag2011}.  First, define the Hamiltonian function for the $i$th stage of Problem~(PL2b) as
\[
H^{[i]}(x^{[i]},\lambda_0,\lambda^{[i]},u^{[i]},s) := \tau_i \left(\frac{1}{2}\,\lambda_0\,\|u^{[i]}\|^2 + (\lambda^{[i]})^T \left(f(x^{[i]},s) + u^{[i]}\right)\right),
\]
where $\lambda_0$ is a scalar (multiplier) parameter, and $\lambda^{[i]}:[0,1]\rightarrow\dR^n$, such that $\lambda^{[i]} := (\lambda^{[i]}_1,\ldots,\lambda^{[i]}_n)$, is the adjoint variable vector (or the vector of multiplier functions) in the $i$th stage.  For neatness of the above expression, the time dependence of the variables has been suppressed.   For further neatness, let
\[
H^{[i]}[s] := H^{[i]}(x^{[i]}(s),\lambda_0,\lambda^{[i]}(s),u^{[i]}(s),s)\,.
\]
First, we recall formally that $L^\infty(0,t_f;\dR^n)$ denotes the space of essentially bounded measurable functions equipped with the essential supremum norm. Furthermore, $W^{1,\infty}(0,t_f;\dR^n)$ is the Sobolev space consisting of functions $x : [0,t_f] \to \dR^n$ whose first derivatives lie in $L^\infty$.  Suppose that $x^{[i]}\in W^{1,\infty}(0,1;\dR^n)$ and $u^{[i]}\in L^\infty(0,1;\dR^n)$, $i = 1,\ldots,N$, solve Problem~(PL2b).  Then there exist a number $\lambda_0\ge0$ and the vector function $\lambda^{[i]}\in W^{1,\infty}(0,1;\dR^n)$ such that\linebreak $\overline\lambda^{[i]}(s):= (\lambda_0,\lambda^{[i]}(s)) \neq \bf0$, $i = 1,\ldots,N$, for every $s\in[0,1]$, and, in addition to the state differential equations and the other constraints given in Problem~(PL2b), the following conditions hold:
\begin{eqnarray}
&& \dot{\lambda}_j^{[i]}(s) = -H^{[i]}_{x^{[i]}_j}[s]\,,\mbox{ a.e. } s\in[0,1],\ \ i = 1,\ldots,N,\ j = 1,\ldots,n\,, \label{L2_lambdax}  \\[1mm]
&& \lambda_j^{i+1}(0) =  \lambda_j^{[i]}(1) + \delta_j^{[i]}\,,\ \ i = 1,\ldots,N-1,\ j = 1,\ldots,n\,, \label{L2_lambda_jump}   \\[1mm]
&& u^{[i]}(s)\in\argmin_{w\in\dR} H^{[i]}(x^{[i]}(s),\lambda_0,\lambda^{[i]}(s),w,s)\,,\ \mbox{a.e. } s\in[0,1]\,,  \label{L2_controli}
\end{eqnarray}
where $H^{[i]}_{x^{[i]}_j} := \partial{H^{[i]}}/\partial{x^{[i]}_j}$, and $\delta_j^{[i]}$, $i = 1,\ldots,N-1$, $j = 1,\ldots,n$, are real constants.

We define the ``overall" adjoint variable vector $\lambda_j(t)$, $j=1,\ldots,n$, to be formed by concatenating the stage adjoint variable vectors as follows.
\begin{equation}  \label{overall}
\lambda(t) := \lambda^{[i]}(s)\,,\quad  t = t_{i-1} + s\,\tau_i,\quad
s\in[0,1]\,,\quad i = 1,\ldots,N\,.
\end{equation}
Condition~\eqref{L2_lambda_jump} states that the adjoint variables $\lambda_j(t)$, $j = 1,\ldots,n$, may have jumps as they go from one stage to the other.

The optimality conditions \eqref{L2_lambdax}--\eqref{L2_controli} can now be re-written rather more explicitly, along with the state equations, as follows.
\begin{eqnarray}
\dot{x}(t) &=& f(x(t),t) + u(t)\,,\quad x(0) = z_0\,,\ x(t_1) = z_1\,,\ldots,\ x(t_N) = z_N\,,\mbox{ a.e. } t\in[0,t_N]\,, \label{L2_x_eqn} \\[1mm]
\dot{\lambda}(t) &=& -f_x(x(t),t)^T\,\lambda(t)\,,\mbox{ for all } t\in[t_{i-1},t_i)\,,\ i = 1,\ldots,N\,,  \label{L2_adjoint_DE} \\[1mm]
\lambda_j(t_i^+) &=&  \lambda_j(t_i^-) + \delta_j^{[i]}\,,\ \ i = 1,\ldots,N-1,\ j = 1,\ldots,n\,, \label{L2_lambda_jump2}   \\[1mm]
\lambda_0\,u(t) &=& -\lambda(t)\,,\mbox{ for all } t\in[t_{i-1},t_i)\,,\ i = 1,\ldots,N\,, \label{L2_optcontrol}
\end{eqnarray}
where $f_x(x(t),t)$ is the Jacobian of $f(x(t),t)$ with respect to $x$.  In the jump condition~\eqref{L2_lambda_jump2}, $\lambda_3(t_i^+) := \lim_{t\to t_i^+}\lambda_3(t)$ and $\lambda_3(t_i^-) := \lim_{t\to t_i^-}\lambda_3(t)$.

The problems which result in $\lambda_0 = 0$ are called {\em abnormal} in the optimal control theory literature, for which the necessary conditions in \eqref{L2_x_eqn}--\eqref{L2_optcontrol} are independent of the objective
functional and therefore not fully informative. The problems that result in $\lambda_0 > 0$ are referred to as {\em normal}, which is the case of Problem~(PL2b) as asserted by Lemma~\ref{L2_lambda0} below.

\begin{lemma}[Normality]  \label{L2_lambda0}
Problem~{\em(PL2b)} is normal.  Subsequently, one can set $\lambda_0 = 1$ for Problem~{\em(PL2b)} and express the optimal control as $u(t) = -\lambda(t)$, for a.e.\ $t\in[t_0,t_f)$.
\end{lemma}
\begin{proof}
Suppose that $\lambda_0 = 0$.  Then Condition~\eqref{L2_optcontrol} implies that $\lambda(t) = 0$ for all $t\in[0,t_f]$ and thus $(\lambda_0, \lambda(t)) = 0$, which is not allowed by the maximum principle cited above.  Therefore, $\lambda_0 >0$, and so the problem is normal and that, without loss of generality, one can take $\lambda_0 = 1$.  Then the substitution of $\lambda_0 = 1$ in~\eqref{L2_optcontrol} completes the proof.
\end{proof}

\begin{remark}[Optimal control] \label{rem:L2_opt_eqns}\rm
We re-iterate that, by Lemma~\ref{L2_lambda0}, $u(t) = -\lambda(t)$, i.e., the residual is the negative of the adjoint variable.  Furthermore, by \eqref{L2_lambda_jump2}, the adjoint variable $\lambda$ will typically be discontinuous at $t_i$, $i = 1,\ldots,N-1$; so will be the residual $u$.
\proofbox
\end{remark}

The necessary conditions of optimality for Problem~(PL2), or equivalently those for Problem~(PL2b), which are listed in \eqref{L2_x_eqn}--\eqref{L2_optcontrol} can now be simply written, also using Lemma~\ref{L2_lambda0}, as a two-point boundary-value problem (TPBVP) in each of the $N$ stages, in the following theorem.
\begin{theorem}[$L^2$-residual]  \label{theo:L2}
Suppose that $x(t)$ and $u(t)$ solve Problem~{\rm(PL2)}.  Then they solve
\begin{eqnarray}
&&\dot{x}(t) = f(x(t),t) - \lambda(t)\,,\ \ x(t_{i-1}) = z_{i-1},\ \  x(t_i) = z_i\,, \label{L2_state_eqn} \\[2mm]
&&\dot{\lambda}(t) = -f_x(x(t),t)^T\,\lambda(t)\,, \label{L2_costate_eqn}
\end{eqnarray}
for all $t\in[t_{i-1},t_i]$, $i = 1,\ldots,N$, with possible jumps in the values of the adjoint variable vector $\lambda(t)$ at $t_i$, $i = 1,\ldots,N-1$.
\end{theorem}

\begin{remark} \label{rem:L2soln} \rm
Theorem~\ref{theo:L2} asserts that the TPBVP in \eqref{L2_state_eqn}--\eqref{L2_costate_eqn} should be solved in $[t_{i-1},t_i]$ (in each of the $N$ stages) for unknown functions $x$ and $\lambda$.
\proofbox
\end{remark}

\subsection{Necessary conditions of optimality for stage \boldmath{$L^\infty$}-minimization}

Using the rescaling and related definitions made at the beginning of Section~\ref{sec:nec_cond}, a multi-stage optimal control problem for~(PLinfc) can be written as
\[
\mbox{(PLinfd)}\left\{\begin{array}{rl}
\min &\ \ \ds\sum_{i=1}^N \tau_i \int_{0}^{1} y^{[i]}(s)\,ds \\[6mm]
\mbox{s.t.} &\ \ \dot{x}^{[i]}(s) = \tau_i \left(f(x^{[i]}(s),s) + y^{[i]}(s)\,v^{[i]}(s)\right),\mbox{ a.e. } s\in[0,1]\,, \\[2mm]
 &\ \  x^{[i]}(0) = z_{i-1}\,,\ x^{[i]}(1) = z_i\,, \ \ \ i = 1,\ldots,N\,,\\[2mm]
 & \ \ \dot{y}^{[i]}(s) = 0\,,\mbox{ all } s\in[0,1]\,, \\[2mm]
 & \ \ |v^{[i]}_j(s)| \le 1\,,\mbox{ a.e. } s\in[0,1]\,, \ \ \ i = 1, \ldots, N,\ \ j = 1,\ldots,n\,.
\end{array}\right.
\]

The Hamiltonian function for the $i$th stage of Problem~(PLinfd) is
\[
H^{[i]}(x^{[i]},y^{[i]},\lambda_0,\lambda^{[i]},\mu^{[i]},v^{[i]}) := \tau_i\left(\lambda_0\,y^{[i]} + (\lambda^{[i]})^T \left(f(x^{[i]},s) + y^{[i]}\,v^{[i]}\right) + \mu^{[i]}\cdot 0\right),
\]
where $\mu^{[i]}:[0,1]\rightarrow\dR$ are the additional adjoint variables (compared to the $L^2$-minimization formulation in Section~\ref{sec:L2-min}) in the $i$th stage.  Let
\[
H^{[i]}[s] := H^{[i]}(x^{[i]}(s),y^{[i]}(s),\lambda_0,\lambda^{[i]}(s),\mu^{[i]}(s),v^{[i]}(s),s)\,.
\]
We invoke the maximum principle as in \cite[Theorem~3.1 and Corollary~3.1]{ClaVin1989} and \cite[Section~4]{AugMau2000} again, this time for Problem~(PLinfd): suppose that $x^{[i]}\in W^{1,\infty}(0,1;\dR^n)$, $y^{[i]}\in W^{1,\infty}(0,1;\dR)$ and $v^{[i]}\in L^\infty(0,1;\dR^n)$ solve Problem~(PLinfd).  Then there exist a number $\lambda_0\ge0$, the vector function $\lambda^{[i]}\in W^{1,\infty}(0,1;\dR^n)$, and $\mu^{[i]}\in W^{1,\infty}(0,1;\dR)$, such that $\overline\lambda^{[i]}(s):= (\lambda_0,\lambda^{[i]}(s),\mu^{[i]}(s)) \neq \bf0$, for every $s\in[0,1]$, $i = 1,\ldots,N$, and, in addition to the state differential equations and the other constraints given in Problem~(PLinfd), the following conditions hold:
\begin{eqnarray}
&& \dot{\lambda}_j^{[i]}(s) = -H^{[i]}_{x_j}[s]\,,\mbox{ a.e. } s\in[0,1],\ \ i = 1,\ldots,N,\ j = 1,\ldots,n\,, \label{Linf_lambdax}  \\[1mm]
&& \lambda_j^{i+1}(0) =  \lambda_j^{[i]}(1) + \delta_j^{[i]}\,,\ \ i = 1,\ldots,N-1,\ j = 1,\ldots,n\,, \label{Linf_lambda_jump}   \\[1mm]
&& \dot{\mu}^{[i]}(s) = -H^{[i]}_{y^{[i]}}[s]\,,\mbox{ a.e. } s\in[0,1],\ \ i = 1,\ldots,N\,,\label{Linf_mu}  \\[1mm]
&& \mu^{[i]}(0) = 0 =  \mu^{[i]}(1)\,,\ \ i = 1,\ldots,N,\ j = 1,\ldots,n\,, \label{Linf_mu_transversality}  \\[1mm]
&& v^{[i]}(s)\in\argmin_{|w|\le 1} H^{[i]}(x^{[i]}(s),y^{[i]}(s),\lambda_0,\lambda^{[i]}(s),\mu^{[i]}(s),w,s)\,,\ \mbox{a.e. } s\in[0,1]\,,\label{Linf_controli}
\end{eqnarray}
where $\delta_j^{[i]}$ are real constants.  As in the case of optimality conditions for $L^2$-minimization, here one also has a jump condition on the adjoint variable $\lambda$ at the junction/skeleton points, as given in~\eqref{Linf_lambda_jump}.  The new adjoint variables $\mu^{[i]}$, on the other hand, are continuous in each stage, and this is reflected by~\eqref{Linf_mu} and \eqref{Linf_mu_transversality}, the latter providing the transversality conditions in each stage.  We define the ``overall" adjoint variable $\mu(t)$, in the same way as we had defined $\lambda(t)$ in~\eqref{overall}:
\[
\mu(t) := \mu^{[i]}(s)\,,\quad  t = t_{i-1} + s\,\tau_i,\quad s\in[0,1]\,,\quad \tau_i := t_i - t_{i-1}\,,\quad i = 1,\ldots,N\,.
\]
The adjoint variable $\mu(t)$ is continuous in view of~\eqref{Linf_mu_transversality}.  The optimality conditions \eqref{Linf_lambdax}--\eqref{Linf_controli} can now be re-written, along with the state equations, as follows.
\begin{eqnarray}
\dot{x}(t) &=& f(x(t),t) + \alpha\,v(t)\,,\mbox{ a.e. } t\in[t_0,t_f],\ x(0) = z_0,\ x(t_1) = z_1,\ldots,\ x(t_N) = z_N,
\label{Linf_x_eqn} \\[1mm]
\dot{\lambda}(t) &=& -f_x(x(t),t)^T\,\lambda(t)\,,\mbox{ a.e. } t\in[t_{i-1},t_i)\,,\ i = 1,\ldots,N\,, \label{Linf_adjoint_DE} \\[1mm]
\lambda_j(t_i^+) &=&  \lambda_j(t_i^-) + \delta_j^{[i]}\,,\ \ i = 1,\ldots,N-1,\ j = 1,\ldots,n\,, \label{Linf_lambda_jump2}   \\[1mm]
\dot{\mu}(t) &=& -\lambda_0 - \lambda(t)^T v(t)\,,\mbox{ a.e. } t\in[t_0,t_f],\ \ \mu(t_{i-1}) = 0\,,\ \mu(t_i) = 0\,,   \label{Linf_mu2} \\[1mm]
v_j(t) &=& \left\{\begin{array}{ll}
\ \ 1\,, & \mbox{if}\  \lambda_j(t) < 0\,, \\[2mm]
-1\,, & \mbox{if}\ \lambda_j(t) > 0\,, \\[2mm]
\mbox{\small undetermined}\,, & \mbox{if}\ \lambda_j(t) = 0\,,
\end{array}\right. \mbox{\ \ a.e. } t\in[t_{i-1},t_i),\ i = 1,\ldots,N,\ j = 1,\ldots,n, \label{Linf_opt_control}
\end{eqnarray}
where, in obtaining~\eqref{Linf_opt_control} from~\eqref{Linf_controli}, we have used $y(t) = \alpha^{[i]} > 0$, for $t\in[t_{i-1},t_i)$.

The following lemma can be considered as the companion of Lemma~\ref{L2_lambda0}, for the stage $L^\infty$-minimization of the residual.

\begin{lemma}[Normality]  \label{Linf_lambda0}
Problem~{\em(PLinfd)} is normal.  Subsequently, one can set $\lambda_0 = 1$ for Problem~{\em(PLinfd)}.
\end{lemma}
\begin{proof}
From \eqref{Linf_opt_control}, one has $\lambda_j(t)\,v_j(t) \le 0$,  $i=1,\ldots,N$, and so $\lambda(t)^T v(t) \le 0$, for a.e.\ $t\in[t_0,t_N]$.  Therefore, $\dot{\mu}(t) = -\lambda_0 + |\lambda(t)^T v(t)|$.  Suppose that $\lambda_0 = 0$.  Then $\dot{\mu}(t) = |\langle\lambda(t),v(t)\rangle| \ge 0$, and so the boundary conditions in \eqref{Linf_mu} are satisfied only if $|\langle\lambda(t),v(t)\rangle| = 0$, or, also using \eqref{Linf_opt_control}, only if $\lambda(t) = 0$, for all $t\in[t_0,t_N]$.  This results in $\mu(t) = 0$ and thus the adjoint variable vector $(\lambda_0,\lambda(t),\mu(t)) = 0$, for a.e.\ $t\in[t_0,t_N]$, which is not allowed by the maximum principle.  Therefore, $\lambda_0 \neq 0$, and so one can set $\lambda_0 = 1$.
\end{proof}

The optimality conditions \eqref{Linf_x_eqn}--\eqref{Linf_opt_control} can now be written as a two-point boundary-value problem in the following theorem, which can be considered as the companion of Theorem~\ref{theo:L2}, for the stage $L^\infty$-minimization of the residual.
\begin{theorem}[Stage $L^\infty$-residual]  \label{theo:Linf}
Suppose that $x(t)$ and $u(t) := \alpha^{[i]}\,v(t)$, for $t\in[t_{i-1},t_i]$, $i = 1,\ldots,N$, solve Problem~{\rm(Linf)}. Then $x(t)$, $v(t)$ and $\alpha^{[i]}$ solve
\begin{eqnarray}
\dot{x}(t) &=& f(x(t),t) + \alpha^{[i]}\,v(t)\,,\ \ \ x(t_{i-1}) = z_{i-1}\,,\ \ x(t_i) = z_i\,,  \label{Linf_x_eqn2} \\[1mm]
\dot{\lambda}(t) &=& -f_x(x(t),t)^T\,\lambda(t)\,, \label{Linf_adjoint_DE2} \\[1mm]
\dot{\mu}(t) &=& -1 - \lambda(t)^T v(t)\,,\ \ \mu(t_{i-1}) = 0\,,\ \mu(t_i) = 0\,,  \label{Linf_mu2a}
\end{eqnarray}
where
\begin{equation}  \label{Linf_opt_control_u}
v_j(t) = \left\{\begin{array}{ll}
\ \ 1\,, & \mbox{if}\ \lambda_j(t) < 0\,, \\[2mm]
-1\,, & \mbox{if}\ \lambda_j(t) > 0\,, \\[2mm]
\mbox{\small undetermined}\,, & \mbox{if}\ \lambda_j(t) = 0\,,
\end{array}\right.
\end{equation}
for a.e.\ $t\in[t_{i-1},t_i]$, $i = 1,\ldots,N$, and $j = 1,\ldots,n$, with possible jumps in the values of the adjoint variable vector $\lambda(t)$ at $t_i$, $i = 1,\ldots,N-1$.
\end{theorem}

\begin{remark} \label{rem:Linfsoln} \rm
Theorem~\ref{theo:Linf} asserts that the TPBVP in \eqref{Linf_x_eqn2}--\eqref{Linf_mu2a} can be independently solved in $[t_{i-1},t_i]$ (in each of the $N$ stages) for unknown functions $x$, $\lambda$ and $\mu$, and the parameter $\alpha^{[i]}$.
\proofbox
\end{remark}

\noindent
{\bf Bang--Bang and Singular Control.}\ \
If $\lambda_j(t) = 0$ only at isolated points in $[s_1,s_2]\subset[t_0,t_f]$, then $v_j(t)\in\{-1,1\}$ for a.e.\ $t\in[s_1,s_2]$ by~\eqref{Linf_opt_control_u}.  In this case, $v_j(t)$ is called {\em bang--bang control} over $[s_1,s_2]$, since either the value of $v_j(t)$ changes between $1$ and $-1$, or simply $v_j(t) \equiv 1$ or $v_j(t) \equiv -1$, over this interval.  However, if $\lambda_j(t) = 0$, for a.e.\ $t\in[s_1,s_2]\subset[t_0,t_f]$, then, by~\eqref{Linf_opt_control_u}, $v_j(t)$ is undetermined over $[s_1,s_2]$ and is referred to as {\em singular control}.  If $v_j(t)$ is of singular type throughout the interval $[t_0,t_f]$, then it is called {\em totally singular}.

\begin{lemma}[Partial singularity] \label{lem:tot-sing}
Suppose that $n=1$.  Then the optimal control $v(t)$ cannot be singular over the entire interval $[t_{i-1},t_i]$ for any given $i = 1,\ldots,N$.  Therefore, $v(t)$ can also not be totally singular.
\end{lemma}
\begin{proof}
Suppose that (the only control variable) $v(t)$ is singular, i.e., $\lambda(t) = 0$, for a.e.\ $t\in[t_{i-1},t_i]$.  Then the ODE in~\eqref{Linf_mu2a} becomes $\dot{\mu}(t) = -1$, which does not satisfy the boundary conditions $\mu(t_{i-1}) = 0$ and $\mu(t_i) = 0$, violating the maximum principle and thus proving the lemma.  The second statement of the lemma follows immediately from the definition of total singularity.
\end{proof}


\section{Examples}
\label{sec:examples}

In this section, we consider three initial value problems of the form $\dot{x} = f(x(t),t)$, $x(t_0)=x_0$, $t_0\le t\le t_f$.  For the analytical results (in particular, for Examples~1 and 2), we suppose that a numerical skeleton $\{z_0,z_1,\ldots,z_N\}$, with $z_0 = x(t_0)$ and $z_N = x(t_f)$, of the numerical solution of the IVP is obtained by a numerical technique, for example, a Runge--Kutta or a multi-step method.  For graphical illustrations of the optimal residuals obtained using our optimal control approach, we obtain the numerical skeleton with the {\sc Matlab} ODE solvers, {\tt ode15s}, {\tt ode45} and {\tt ode113}.  For each of these solvers, we set {\color{black} {\tt reltol = 1.0e-8} and {\tt abstol = 1.0e-8}}.

For comparison, we also use the {\sc Matlab} function {\tt deval}, which uses a given skeleton to find the values of $x$ at a specified array of interior points of the domain by a polynomial approximation technique specific to each solver~\cite{ShaRei1997}. The particular approximation methods used in {\tt deval} are briefly discussed in that reference, where we learn that the interpolant for \texttt{ode45} is $O(h^4)$ accurate, but the messy details are left to the code \texttt{ntrp45}.  Close inspection of \texttt{ntrp45} shows that the polynomial interpolant for \texttt{ode45}, which was supplied by Dormand and Prince by private communication to the authors of~\cite{ShaRei1997}, is of degree at most 4, uses all stages\footnote{In this section of the paper, we use the word ``stage'' as the Runge--Kutta community does, to refer to each new function evaluation in each line of the Butcher tableau.  The polysemy of the word ``stage'' is perhaps unfortunate, and might cause confusion; we did warn in an earlier footnote that this would happen.} of the RK method including the first stage needed for the following step and only those and is hence ``free'', and has residual $O(h^4)$.  We give some details here because they are not easily available elsewhere (one has to read the code itself) and because writing the interpolant explicitly underscores the difference between the polynomial approach and the approach we propose.

The method used by \texttt{ode45} is a $(5,4)$ Runge--Kutta pair by Dormand and Prince~\cite{Dormand1980}, which uses ``local extrapolation'' meaning that the solution is advanced by the fifth-order method while the error estimate in the fourth-order method is used for stepsize control. The method has the FSAL (First Same As Last) property, which means that while it uses $7$ stages on any given step, the final stage will be re-used on the next step if the current step is successful, making the process almost equivalent to a six-stage method in computational cost.

The Butcher tableau for the method is of the form
\begin{align*}
\renewcommand\arraystretch{1.5}
\begin{array}{c|ccccccc}
     0\; \\
     \tfrac{1}{5}\; & \;\tfrac{1}{5}\; \\
     \tfrac{3}{10}\; & \tfrac{3}{40} & \;\tfrac{9}{40}\; \\
     \tfrac{4}{5}\; & \tfrac{44}{45} & -\tfrac{56}{15} & \;\tfrac{32}{9}\; \\
     \tfrac{8}{9}\; & \tfrac{19372}{6561} & -{\tfrac{25360}{2187}} & \tfrac{64448}{6561} & \;-{\tfrac{212}{729}}\; \\
     1\; & \tfrac{9017}{3168} & -{\tfrac{355}{33}} & \tfrac{46732}{5247} & \tfrac{49}{176} & \;-{\tfrac{5103}{18656}}\; \\
     1\; & {\tfrac{35}{384}} & 0 & {\tfrac{500}{1113}} & {\tfrac{125}{192}} & -{
\tfrac{2187}{6784}} & {\tfrac{11}{84}} \\
     \hline
         & {\tfrac{5179}{57600}} & 0 & {\tfrac{7571}{16695}} & {\tfrac{393}{640}} & -{
\tfrac{92097}{339200}} & \;{\tfrac{187}{2100}}\; & \tfrac{1}{40} \\
         & d_1 & d_2 & d_3 & d_4 & d_5 & d_6 & \;d_7\;
   \end{array} \end{align*}
The stage values $k_j$ are given by $k_1 = f(t_n, y_n)$ and
\begin{equation}
    k_j = f\left( t_n+c_j h, y_n + h\sum_{\ell=1}^{j-1} a_{j,\ell} k_\ell \right)
\end{equation}
for $j=2$, $\ldots$, $6$. The seventh stage, $k_7$, which will be re-used on the next step if the current step is accepted, is
\begin{equation}
    k_7 = f\left( t_n + h, y_n + h\sum_{\ell=1}^7 b_\ell k_\ell \right)\>.
\end{equation}
The vector for computing the fourth-order solution used to estimate the error for stepsize control purposes is $b = \left[{\tfrac{35}{384}}, 0, {\tfrac{500}{1113}}, {\tfrac{125}{192}}, -{
\tfrac{2187}{6784}}, {\tfrac{11}{84}}\right]
$.

The continuous extension is the polynomial interpolant given by
\begin{equation}
    z(s) = y_n + h\sum_{j=1}^7 d_j(s) k_j \label{eq:dormandprincecext}
\end{equation}
where $s = (t-t_n)/h$ and
\begin{align}
    d_1(s) &= -\tfrac{1}{384}{s \left(435 s^{3}-1184 s^{2}+1098 s -384\right)} \nonumber \\
    d_2(s) &= 0 \nonumber \\
    d_3(s) &= \tfrac{500}{1113}{ s^{2} \left(6 s^{2}-14 s +9\right)} \nonumber \\
    d_4(s) &= -\tfrac{125}{192} s^{2} \left(9 s^{2}-16 s +6\right) \nonumber \\
    d_5(s) &= \tfrac{729}{6784} s^{2} \left(35 s^{2}-64 s +26\right) \nonumber \\
    d_6(s) &= -\frac{11}{84} s^{2} \left(3 s -2\right) \left(5 s -6\right) \nonumber \\
    d_7(s) &= \tfrac{1}{2}s^{2} \left(5 s -3\right) \left(s -1\right) \>.
\end{align}
We give this interpolant in factored form so that it is intelligible for the reader, who is not expected to be an expert in interpolation schemes for Runge--Kutta methods.  Note that all of the $d_i$ are zero when $s=0$, ensuring that the polynomial interpolates the start of the step.  Note that $d_7$ is zero also when $s=1$, meaning that the seventh stage is not involved in the value of $y$ at the end of the step (this is the FSAL property).  Note also that the derivatives of each of the polynomials are zero at $s=0$, except for the first polynomial (it is $1$ there, which is not evident without computation). This means that the derivative at $s=0$ of the interpolant will match the value of the first stage, i.e. the derivative of $y$ at $s=0$.  Finally, computation shows that the derivative of all the polynomials except $d_7$ is zero at $s=1$, while the derivative of $d_7$ is $1$ at $s=1$, which means that the derivative of the interpolant matches the derivative of the computed solution $y$ at $s=1$, ensuring that the interpolant is continuously differentiable as a piecewise function, apart from rounding errors.

Because the interpolant for \texttt{ode45} is piecewise continuously differentiable, it is piecewise continuous, and hence must have a residual no smaller than the minimal residual.  In fact, because the skeleton is $O(h^5)$ accurate and the interpolant is only $O(h^4)$ accurate, we usually see that the residual using the interpolant of \texttt{deval} is substantially larger than the minimal residual.

The approximations for \texttt{ode15s} and \texttt{ode113} are \textsl{not} continuously differentiable, by the way. They are not even continuous at both the mesh points $t_n$ and $t_{n+1}$.  This implies that computing a residual using these approximations may not give an upper bound on the minimal residual.  That is, the computed residual using these approximations is only an estimate of the residual, and may not be reliable in any consistent way.

We used both Maple and the B-series package in Julia~\cite{ketcheson2023computing,ranocha2021bseries} to analyze the interpolant for \texttt{ode45}.  Computation with the B-series package confirmed that the interpolant is $4$th order accurate, as we shall see in a moment.
In detail, according to the B-series package,
the leading terms of the residual are
\begin{align}
\frac{dz}{dt}  - f(z) =  s  \left( 1-s \right)\Biggl( &\frac{1}{120}  \left( 5  s^{2} - 9  s + 3 \right)  F_{f}\mathopen{}\left( \rootedtree[.[.[.[.[.]]]]] \right)\mathclose{} + \frac{1}{120}  \left( 5  s^{2} - 5  s + 1 \right)  F_{f}\mathopen{}\left( \rootedtree[.[.[.[.][.]]]] \right)\mathclose{} \nonumber\\&{}+ \frac{1}{40}  \left( 5  s^{2} - 5  s + 1 \right)  F_{f}\mathopen{}\left( \rootedtree[.[.[.[.]][.]]] \right)\mathclose{} + \frac{1}{30}  \left( 5  s^{2} - 5  s + 1 \right)  F_{f}\mathopen{}\left( \rootedtree[.[.[.[.]]][.]] \right)\mathclose{} \nonumber\\&{}+ \frac{1}{120}\left( 5  s^{2} - 5  s + 1 \right)  F_{f}\mathopen{}\left( \rootedtree[.[.[.][.][.]]] \right)\mathclose{} + \frac{1}{540} \left( 90  s^{2} - 92  s + 19 \right)  F_{f}\mathopen{}\left( \rootedtree[.[.[.][.]][.]] \right)\mathclose{} \nonumber\\&{}+ \frac{1}{720}  \left( 90  s^{2} - 92  s + 19 \right)  F_{f}\mathopen{}\left( \rootedtree[.[.[.]][.[.]]] \right)\mathclose{} + \frac{1}{360}  \left( 90  s^{2} - 92  s + 19 \right)  F_{f}\mathopen{}\left( \rootedtree[.[.[.]][.][.]] \right)\mathclose{} \nonumber\\&{}+ \frac{1}{2160}  \left( 90  s^{2} - 92  s + 19 \right)  F_{f}\mathopen{}\left( \rootedtree[.[.][.][.][.]] \right)\mathclose{}\Biggr) h^4  + O(h^5)
\end{align}
Remember that $t = t_n + sh$ so $dz/dt = (dz/ds)/h$ by the chain rule.

The output of the B-series script has been tidied up by use of Maple so that one can see at a glance that the residual is zero at either end of the step ($s=0$ and $s=1$).  We use a common notation for elementary differentials and rooted trees.  We are not aware of any published residual analysis of this method.  We only print the fourth-order terms here, but more are available if needed.  The B-series package is quite effective and efficient!

We refer back to Figure~\ref{fig:shoode45}.  The theory above predicts that the residual will behave like $O(h^4)$ as the mean stepsize goes to zero (assuming that \texttt{ode45} equidistributes the error, which it approximately does).  We see that for the simple harmonic oscillator, this is true: the \textsl{actual} residual error diminishes like $O(h^4)$, instead of $O(h^5)$.  [We do not show this here, but for the Van der Pol oscillator a phenomenon known as ``order reduction'' happens and we only get $O(h^3)$ decay of the residual error.] We remind you that \texttt{ode45} uses local extrapolation and is expected to have (``forward'' or ``global'') error that diminishes like $O(h^5)$.  We do not show this here, but this happens.  What we are seeing here is that the \textsl{interpolant} supplied by \texttt{deval} for \texttt{ode45} is only $O(h^4)$ accurate, and so the residual error it measures is not as accurate as it could be.

In what follows we demonstrate using the tools of optimal control to understand the behaviour of numerical methods.  For the first two examples the analysis can be carried out ``by hand.''
In Example~3, which involves the Van der Pol system of two first-order ODEs, an analytical solution is not possible to obtain; therefore, the optimal control problems are solved by the {\em first-discretize-then-optimize} approach~\cite{Betts2020}: Direct discretization of an (infinite-dimensional) optimal control problem gives rise to a large-scale finite-dimensional optimization problem, which is then solved by using standard optimization software.  We use the mathematical programming modelling language AMPL~\cite{AMPL} employing the optimization software Knitro~\cite{Knitro}, version 13.0.1.  We set AMPL's tolerance parameters as {\tt feastol = 1e-12} and {\tt opttol = 1e-12} for Knitro.

\subsection{Example 1: the Dahlquist test problem}

Consider the Dahlquist test problem,
\begin{equation} \label{simpleeqn}
\dot{z}(t) = a\,z(t)\,,\quad z(0) = z_0\,,
\end{equation}
where $a$ is a nonzero real constant.

We recall that in~\cite{CorKayMoi2019} the $L^\infty$-norm of the {\em relative} residual is minimized in a given stage for the Dahlquist test problem by also using an optimal control approach.  In Sections~\ref{subsubsec:Dahlquist_L2} and \ref{subsubsec:Dahlquist_Linf} below, we obtain analytical expressions for the {\em absolute} residuals that minimize the $L^2$- and stage $L^\infty$-norms, respectively, for the same ODE, and compare the two residuals.

\subsubsection{\boldmath{$L^2$}-minimization of the residual}
\label{subsubsec:Dahlquist_L2}

\begin{proposition}[\boldmath{$L^2$}-residual of \eqref{simpleeqn}] \label{simp_fact_L2}
For $i = 1,\ldots,N$, the residual
\begin{equation} \label{simp_res_L2}
u(t) = -2a\,\frac{z_i - e^{a(t_i-t_{i-1})}z_{i-1}}{1 - e^{2a(t_i-t_{i-1})}}\,e^{a(t_i-t)}\,,
\end{equation}
for all $t\in[t_{i-1},t_i]$, solves {\em (PL2)} with $f(x(t),t) = a\,x(t)$.  The resulting interpolant is given by
\begin{equation} \label{simp_x_L2}
x(t) = -\frac{z_i - e^{-a(t_i-t_{i-1})} z_{i-1}}{e^{-a(t_i-t_{i-1})} - e^{a(t_i-t_{i-1})}}\,e^{a(t-t_{i-1})} - \frac{1}{2a}\,u(t)\,,
\end{equation}
for all $t\in[t_{i-1},t_i]$.
\end{proposition}
\begin{proof}
The two-point boundary-value problem in \eqref{L2_state_eqn}--\eqref{L2_costate_eqn} in Theorem~\ref{theo:L2} reduces, for \eqref{simpleeqn} and $i=1,\ldots,N$, to
\begin{eqnarray*}
&&\dot{x}(t) = a\,x(t) - \lambda(t),\quad x(t_{i-1}) = z_{i-1}\,,\ \ x(t_i) = z_i\,, \\
&&\dot{\lambda}(t) = -a\,\lambda(t)\,,
\end{eqnarray*}
for all $t\in[t_{i-1},t_i]$.  Also, recall from Remark~\ref{rem:L2_opt_eqns} that $u(t) = -\lambda(t)$, and that the residual function $u$ will typically have jumps at $t_i$, $i=1,\ldots,N-1$.  Using elementary linear algebraic techniques for the solution of systems of constant coefficient ODEs, and manipulations, one can easily obtain the required expressions given in the proposition.
\end{proof}

\begin{remark}  \label{rem:u_L2} \rm
Note from \eqref{simp_res_L2} that the residual $u(t)$ is an exponential in $t$. The modulus of the residual $|u(t)| = r\,e^{a(t_i-t)}$, with $r>0$ a constant, is decreasing if $a>0$ and increasing if $a<0$, in any stage $i = 1,\ldots,N$.  In particular, the { squared} $L^2$-norm of the residual over $[t_{i-1},t_i]$, in the $i$th stage, $i=1,\ldots,N$, is given as
\[
\|u\|_{L^2}^{{ 2}} = -{ 2}\,a\,\frac{\left(z_i - e^{a(t_i-t_{i-1})} z_{ i-1}\right)^2}{1 - e^{2a(t_i-t_{i-1})}}\,.
\]
\proofbox
\end{remark}

\subsubsection{Stage {\boldmath$L^\infty$}-minimization of the residual}
\label{subsubsec:Dahlquist_Linf}

\begin{proposition}[Stage \boldmath{$L^\infty$}-residual of \eqref{simpleeqn}] \label{simp_fact_Linf}
For $i = 1,\ldots,N$, the residual
\begin{equation} \label{simp_res_Linfty}
u(t) = -a\,\frac{z_{ i} - e^{a(t_i-t_{i-1})} z_{ {i-1}}}{1 - e^{a(t_i-t_{i-1})}} =: \overline{u}^{[i]}\,,
\end{equation}
for all $t\in[t_{i-1},t_i]$, solves {\em (PLinf)} with $f(x(t),t) = a\,x(t)$.  The resulting interpolant is given by
\begin{equation} \label{simp_x_Linfty}
x(t) = \left(z_{i-1} + \frac{\overline{u}^{[i]}}{a}\right)\,e^{a(t-t_{i-1})} - \frac{\overline{u}^{[i]}}{a}\,,
\end{equation}
for all $t\in[t_{i-1},t_i]$.
\end{proposition}
\begin{proof}
The DE \eqref{Linf_adjoint_DE2} for the adjoint variable in Theorem~\ref{theo:Linf} can be written for this example simply as $\dot{\lambda}(t) = -a\,\lambda(t)$, the solution of which is $\lambda(t) = c\,e^{-a\,t}$, where $c$ is a real constant. Clearly, $c\neq0$, as otherwise $\lambda(t) = 0$ for all $t\in[t_{i-1},t_i]$, which, by Lemma~\ref{lem:tot-sing} is not permissible.  The function $\lambda(t)$ does not change sign, either; therefore, the optimal control cannot be singular, i.e., the optimal control is bang--bang and that $u(t) = -\sgn(\lambda(t))\,\alpha^{[i]} = \overline{u}^{[i]}$, which is constant.  The state equation in \eqref{Linf_x_eqn2} can then be written as
\[
\dot{x}(t) - a\,x(t) = \overline{u}^{[i]}\,\quad x(t_{i-1}) = z_{i-1}\,,\ \ x(t_i) = z_i\,,
\]
solution of which yields~\eqref{simp_x_Linfty}.  Evaluating \eqref{simp_x_Linfty} at $t=t_i$ and re-arranging gives \eqref{simp_res_Linfty}.
\end{proof}

\begin{remark}[Constancy of the residual] \rm
The residual $u(t)$ in \eqref{simp_res_Linfty} is constant for all $t\in[t_{ {i-1}},t_{ i}]$, with the constant value typically changing from one stage to the other. We observe that
\begin{equation}  \label{alpha_i}
\alpha^{[i]} = \left|\overline{u}^{[i]}\right| = \left|a\,\frac{z_{ i} - e^{a(t_i-t_{i-1})} z_{ {i-1}}}{1 - e^{a(t_i-t_{i-1})}}\right|\,.
\end{equation}
\proofbox
\end{remark}

It is interesting to compare the residuals which solve Problems~(PL2) and (PLinf), respectively.  So we next provide a companion to Propositions~\ref{simp_fact_L2}--\ref{simp_fact_Linf}.

Let $u_{L^2}(t)$ denote $u(t)$ given in \eqref{simp_res_L2}.  Recall that $\alpha^{[i]} = |u^{[i]}(t)|$, where $u^{[i]}(t)$ is given in \eqref{simp_res_Linfty}.  The following fact establishes a relationship between the (pointwise) residuals minimizing their $L^2$- and stage $L^\infty$-norms, respectively.
\begin{proposition}[Comparison of stage \boldmath{$L^\infty$}- and \boldmath{$L^2$}-residuals]  \label{prop:relation}
For $i = 1,\ldots,N$, one has that
\begin{equation} \label{relation}
\max_{t_{i-1}\le t\le t_i} |u_{L^2}(t)| = \left\{\begin{array}{ll}
\ds\frac{2}{1 + e^{a(t_i-t_{i-1})}}\,\alpha^{[i]}, & \mbox{ if\ \ } a<0\,, \\[4mm]
\ds\frac{2\,e^{a(t_i-t_{i-1})}}{1 + e^{a(t_i-t_{i-1})}}\,\alpha^{[i]}, & \mbox{ if\ \
} a>0\,.
\end{array}\right.
\end{equation}
It follows that
\begin{equation} \label{eqn:compare}
\alpha^{[i]} < \max_{t_i\le t\le t_{i+1}} |u_{L^2}(t)|
\end{equation}
for any $a\neq0$.
\end{proposition}
\begin{proof}
Suppose that $a<0$. Then, since $|u_{L^2}|$ is increasing over $[t_{i-1},t_i]$ by Remark~\ref{rem:u_L2},
\begin{eqnarray*}
\max_{t_{i-1}\le t\le t_i} |u_{L^2}(t)| &=& |u_{L^2}(t_i)| =
2\,\left|a\,\frac{z_i - e^{a(t_i-t_{i-1})} z_{i-1}}{1 -
    e^{2a(t_i-t_{i-1})}}\right| \\
&=& \frac{2}{1 + e^{a(t_i-t_{i-1})}}\,\left|a\,\frac{z_i -
    e^{a(t_i-t_{i-1})} z_{i-1}}{1 - e^{a(t_i-t_{i-1})}}\right| \\
&=& \frac{2}{1 + e^{a(t_i-t_{i-1})}}\,\alpha^{[i]}\,,
\end{eqnarray*}
by using \eqref{alpha_i}.  Suppose that $a>0$.  In this case, since $|u_{L^2}|$ is decreasing over $[t_{i-1},t_i]$ by Remark~\ref{rem:u_L2},
\begin{eqnarray*}
\max_{t_{i-1}\le t\le t_i} |u_{L^2}(t)| &=& |u_{L^2}(t_{i-1})| =
2\,\left|a\,\frac{z_i - e^{a(t_i-t_{i-1})} z_{i-1}}{1 -
    e^{2a(t_i-t_{i-1})}}\right|\,e^{a(t_i-t_{i-1})} \\
&=& \frac{2\,e^{a(t_i-t_{i-1})}}{1 + e^{a(t_i-t_{i-1})}}\,\left|a\,\frac{z_i -
    e^{a(t_i-t_{i-1})} z_{i-1}}{1 - e^{a(t_i-t_{i-1})}}\right| \\
&=& \frac{2\,e^{a(t_i-t_{i-1})}}{1 + e^{a(t_i-t_{i-1})}}\,\alpha^{[i]}\,.
\end{eqnarray*}
This proves \eqref{relation}.  The comparison made in \eqref{eqn:compare}
can simply be furnished by showing that the coefficients of $\alpha^{[i]}$
displayed in \eqref{relation} are greater than 1: if $a<0$, then $1 +
e^{a(t_i-t_{i-1})}<2$; if $a>0$, then $2\,e^{a(t_i-t_{i-1})} =
e^{a(t_i-t_{i-1})} + e^{a(t_i-t_{i-1})} > 1 + e^{a(t_i-t_{i-1})}$.
\end{proof}

\begin{remark}[Residuals for large \boldmath{$a$}]  \rm
As $|a|(t_{i+1}-t_i)\to\infty$, we have that $\max_{t_i\le t\le t_{i+1}} |u_{L^2}(t)|\to 2\alpha^{[i]}$.  Therefore, for relatively large values of growth rate $|a|$ in the Dahlquist test problem in \eqref{simpleeqn} and the step-size $(t_{i+1}-t_i)$, the stage $L^\infty$-minimization of the residual (with pointwise values of at most $\alpha^{[i]}$) stands as a far better approach to use than the $L^2$-minimization of the residual (with pointwise values approaching $2\,\alpha^{[i]}$).
\proofbox
\end{remark}

\subsubsection{A graphical illustration of the residuals}

Figure~\ref{fig:ex1} depicts the residual graphs for various criteria with $a=3$, $z_0 = 1$, and $t_f = 1$, for the Dahlquist test problem stated in \eqref{simpleeqn}.  To obtain the numerical skeleton, the popular {\sc Matlab} solvers {\tt ode15s}, {\tt ode45} and {\tt ode113} (see \cite{ShaRei1997}) have been used.  The number of mesh points corresponding to each solver turns out to be $N = 78$, $37$ and $37$, respectively.  With each numerical skeleton, approximating curves were generated using {\sc Matlab}'s {\tt deval}, as well as through the $L^2$- and stage $L^\infty$-minimization of the ODE residual, as given in \eqref{simp_x_L2} and~\eqref{simp_x_Linfty}.

The residuals for each case and approach, nine of them altogether, are depicted in Figure~\ref{fig:ex1}, where the residual $u_M(t)$ is produced using {\tt deval}, and the residuals $u_{L^2}(t)$ and $u_{L^\infty}(t)$ of $L^2$- and stage $L^\infty$-minimization are computed as expressed in \eqref{simp_res_L2} and \eqref{simp_res_Linfty}, respectively.  { The reader can have access to the {\sc Matlab} code ({\tt residuals{\_}zdot{\_}az.m}) that generated the graphs in Figure~\ref{fig:ex1} at \url{https://github.com/rcorless/OptimalResiduals}.}

First of all, we observe that \eqref{eqn:compare} is verified in each of the graphs in Figure~\ref{fig:ex1}{\color{black}(f)--(j).  In the graphs in Figure~\ref{fig:ex1}(e)--(h), the worst residuals look indistinguishable.} The graphs in~Figure~\ref{fig:ex1} further show that, with {\tt ode15s} and {\tt ode45}, both $\max_{0\le t\le 1}|u_{L^2}(t)|$ and $\max_{0\le t\le 1}|u_{L^\infty}(t)|$ are better, i.e. smaller, than $\max_{0\le t\le 1}|u_M(t)|$.  With {\tt ode113}, however, $\max_{0\le t\le 1}|u_M(t)|$ turns out to be the smallest.

This is surprising: the polynomial approximation from the code has produced a residual that is (apparently) \textsl{smaller} than the minimum possible.  The resolution is that the approximation from \texttt{deval} for the solution by \texttt{ode113} is not continuous.  As previously stated, the minimum possible residual from a \textsl{discontinuous} approximation is, in fact, zero, by using the local exact solution as the approximation.

Using discontinuous approximations $z(t)$ introduces Dirac delta functions into the derivative $\dot z(t)$ which means that the Gr\"obner--Alexeev nonlinear variation of constant formula (equation~\eqref{eq:GroebnerAlexeev}) will also have jumps at the nodes.  This is, indeed, one classical analysis of the connection of ``local error'' to ``global error.''

For the present context, it means that the residual computed by the approximation produced by \texttt{deval} for the solution by \texttt{ode113} is not reliable.  Our analysis gives a reliable estimate, in contrast.

All cases in Figure~\ref{fig:ex1} considered, the {\tt ode45} skeleton with the interpolant minimizing the stage $L^\infty$-norm of the residual in part~(i) stands out as the most satisfactory.

\afterpage{\clearpage}
\begin{figure}[h]
\begin{minipage}{52mm}
\begin{center}
\hspace*{0mm}
\includegraphics[width=55mm]{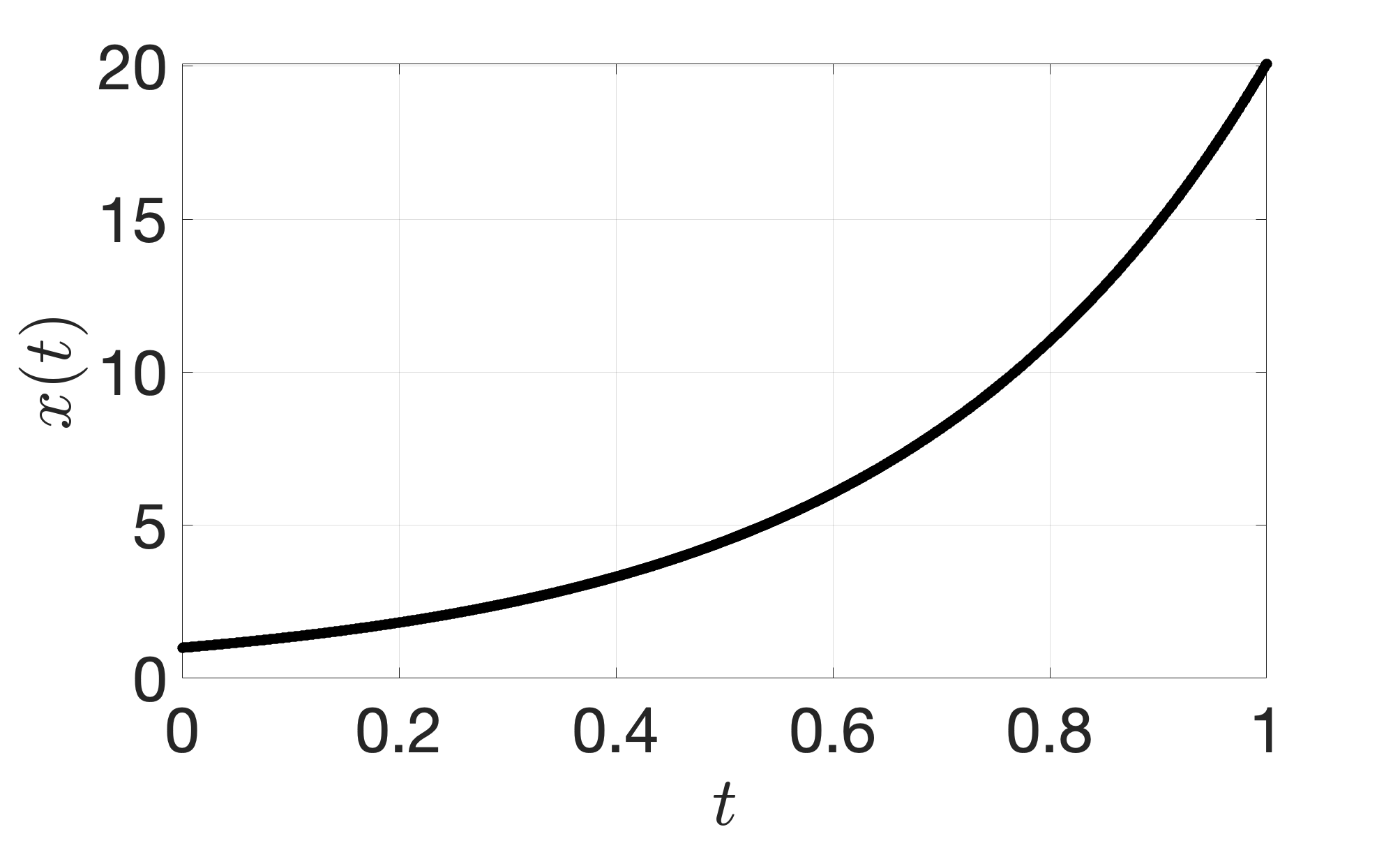} \\
\ \ \ \ \ (a)
\end{center}
\end{minipage}
\\[5mm]
\begin{minipage}{52mm}
\begin{center}
\includegraphics[width=57mm]{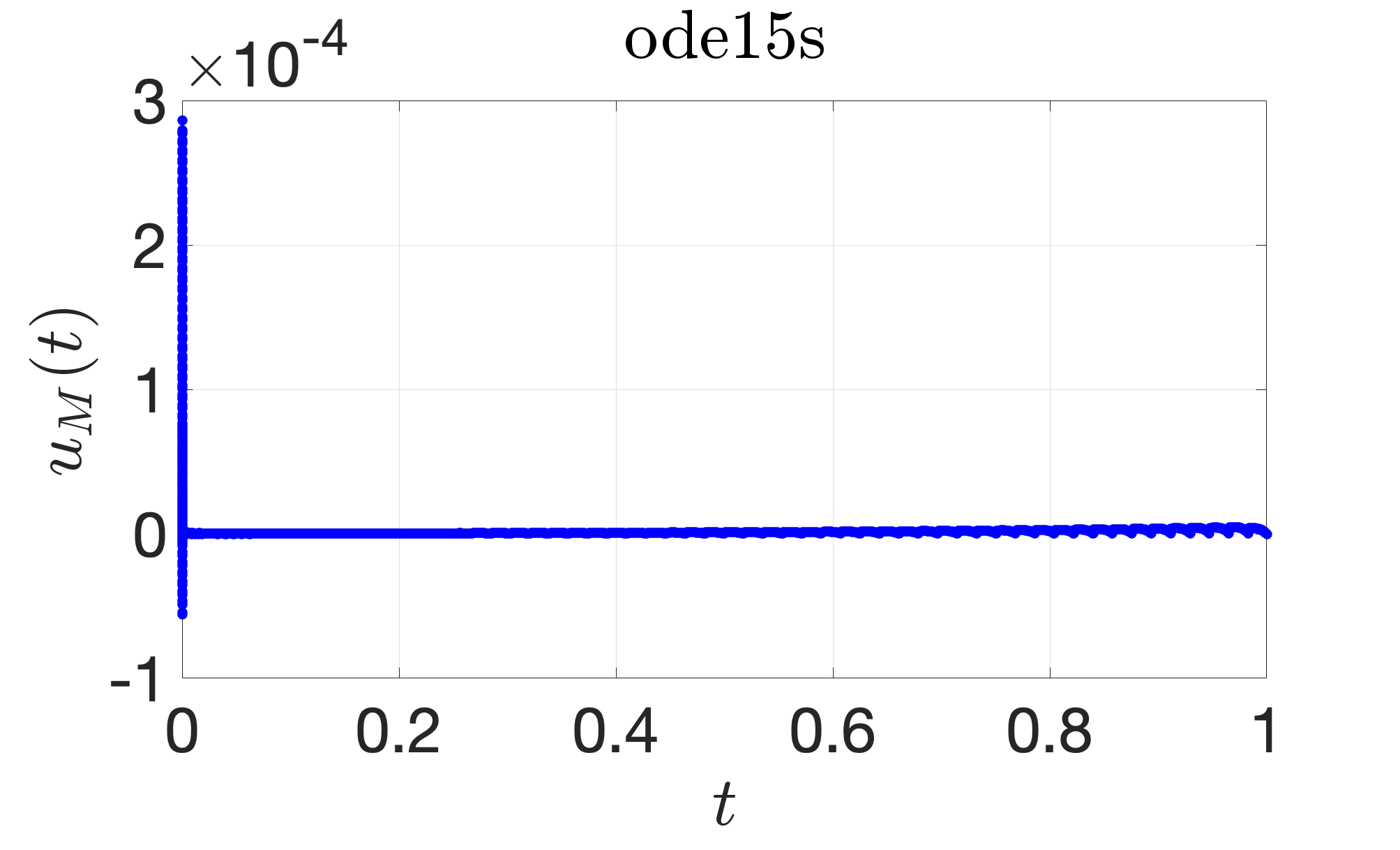} \\
\ \ \ \ \ \ (b)
\end{center}
\end{minipage}
\begin{minipage}{52mm}
\begin{center}
\includegraphics[width=57mm]{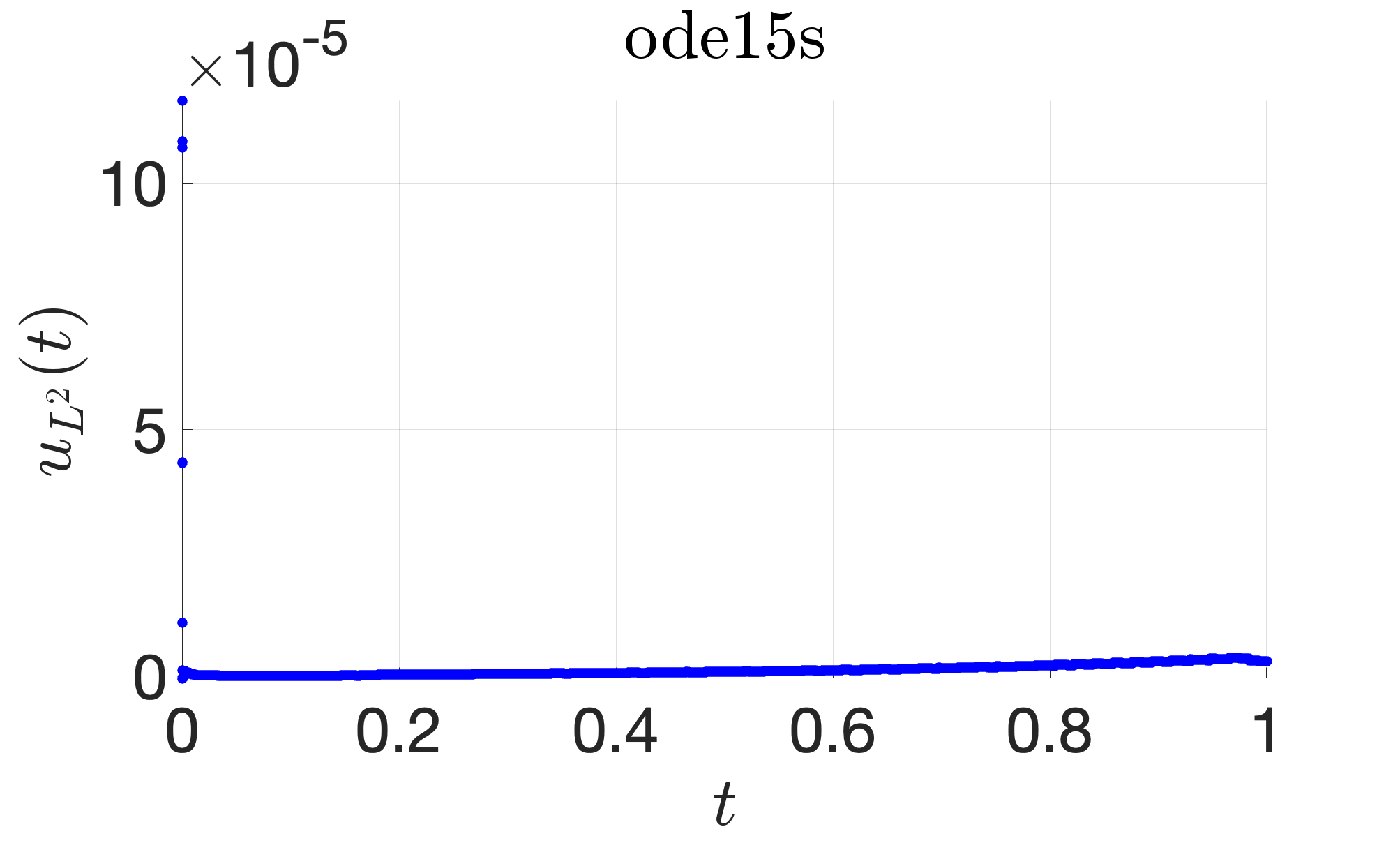} \\
\ \ \ \ \ \ (e)
\end{center}
\end{minipage}
\begin{minipage}{52mm}
\begin{center}
\includegraphics[width=57mm]{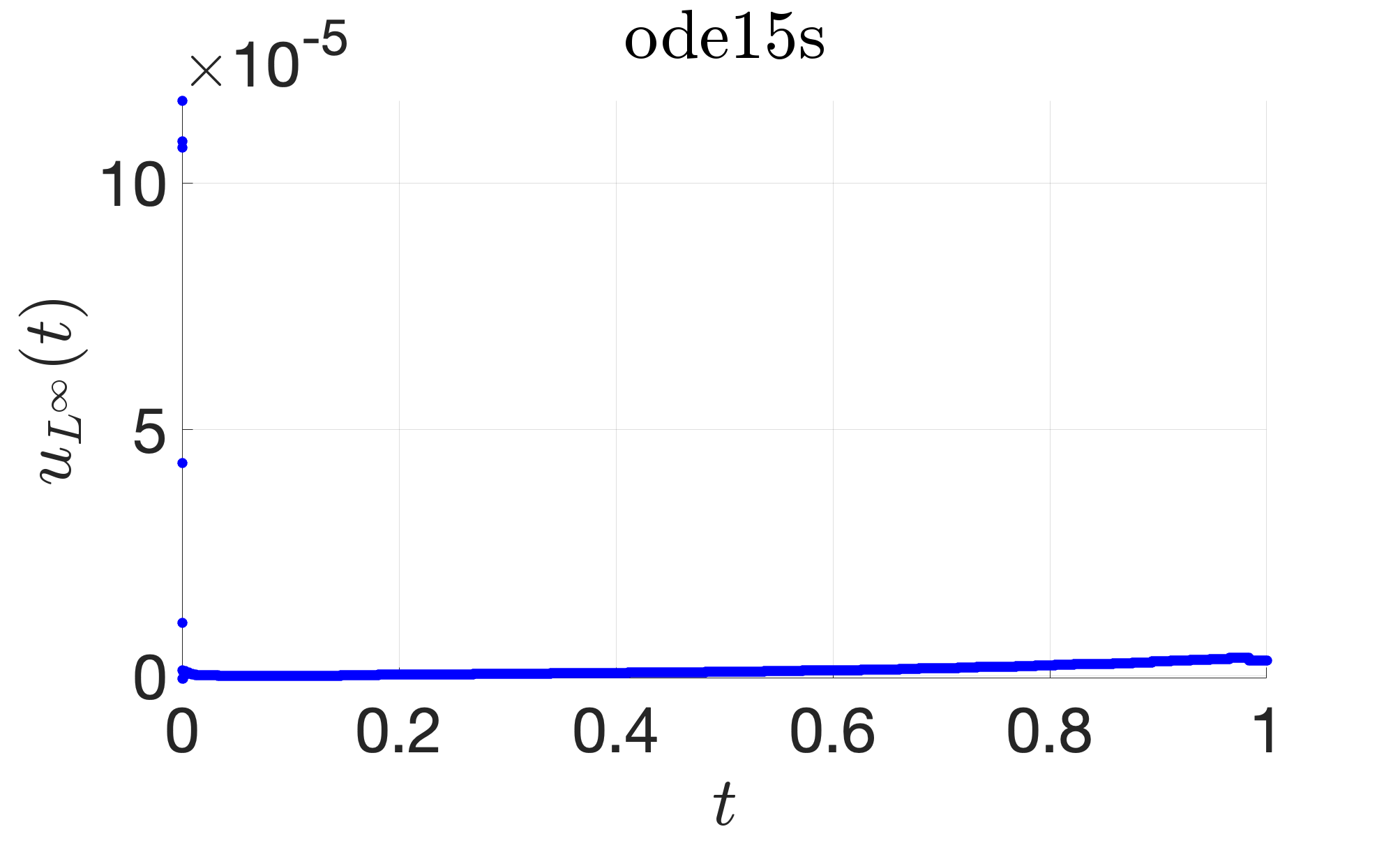} \\
\ \ \ \ \ \ (h)
\end{center}
\end{minipage}
\\[5mm]
\begin{minipage}{52mm}
\begin{center}
\includegraphics[width=57mm]{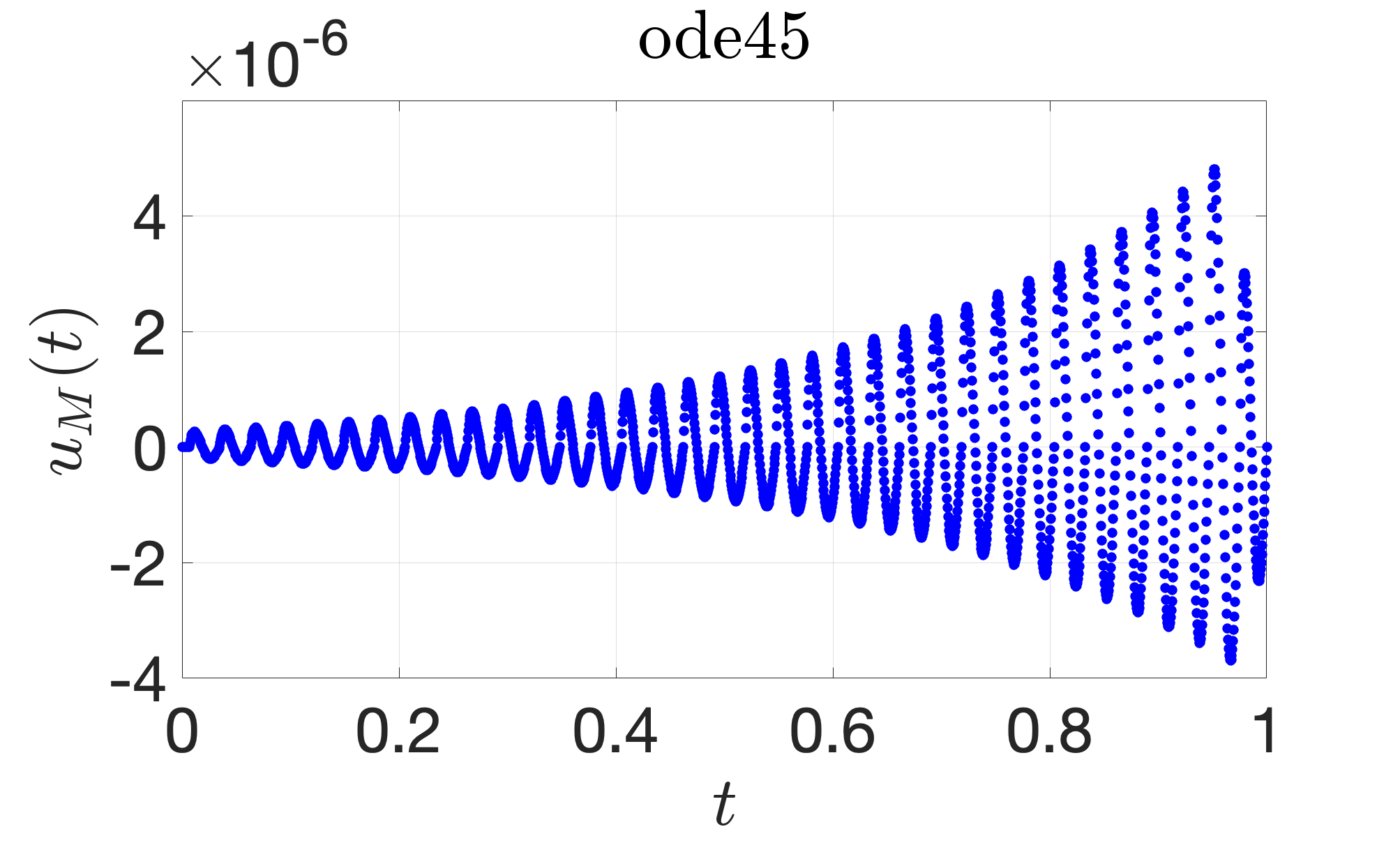} \\
\ \ \ \ \ \ (c)
\end{center}
\end{minipage}
\begin{minipage}{52mm}
\begin{center}
\includegraphics[width=57mm]{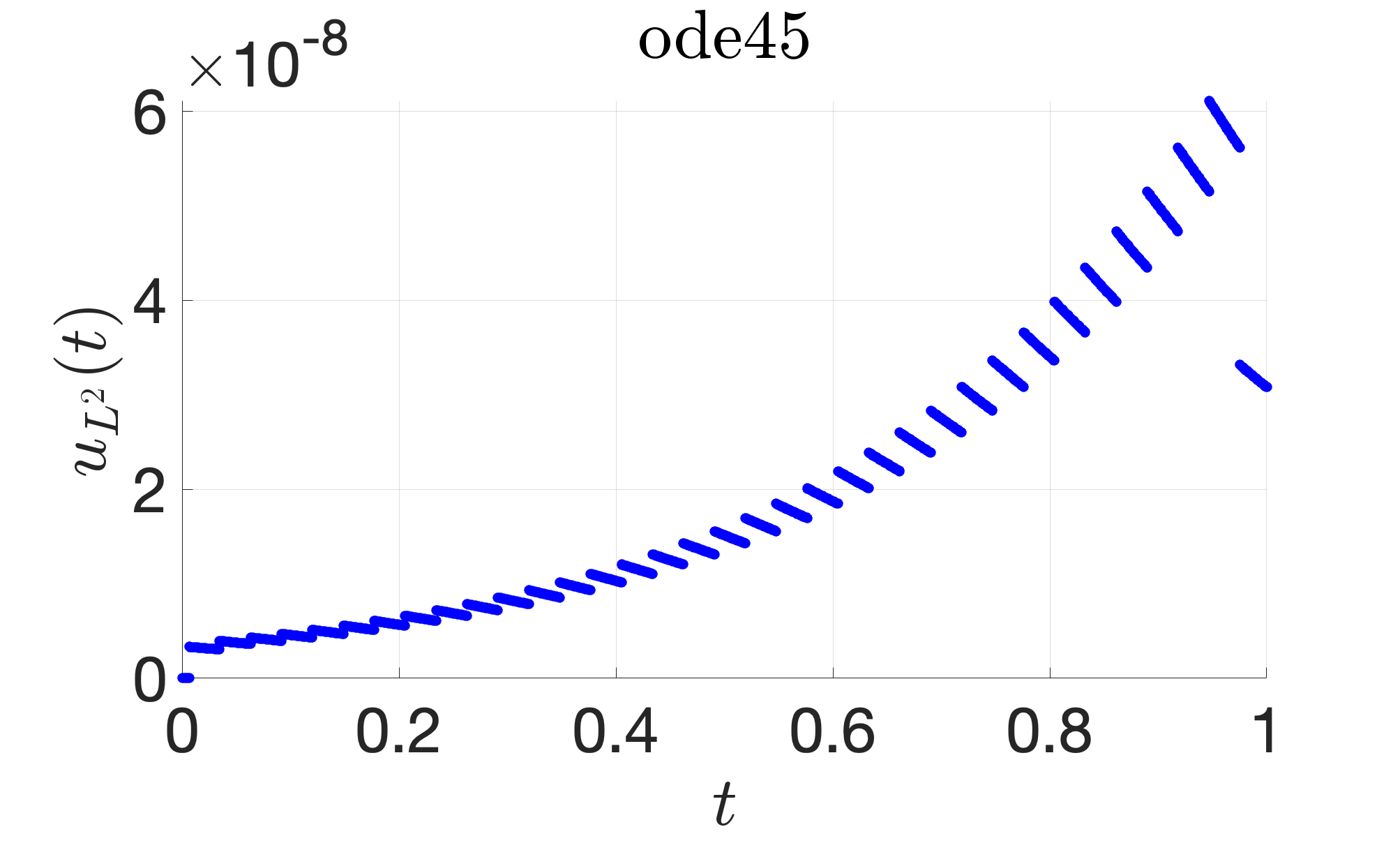} \\
\ \ \ \ \ \ (f)
\end{center}
\end{minipage}
\begin{minipage}{52mm}
\begin{center}
\includegraphics[width=57mm]{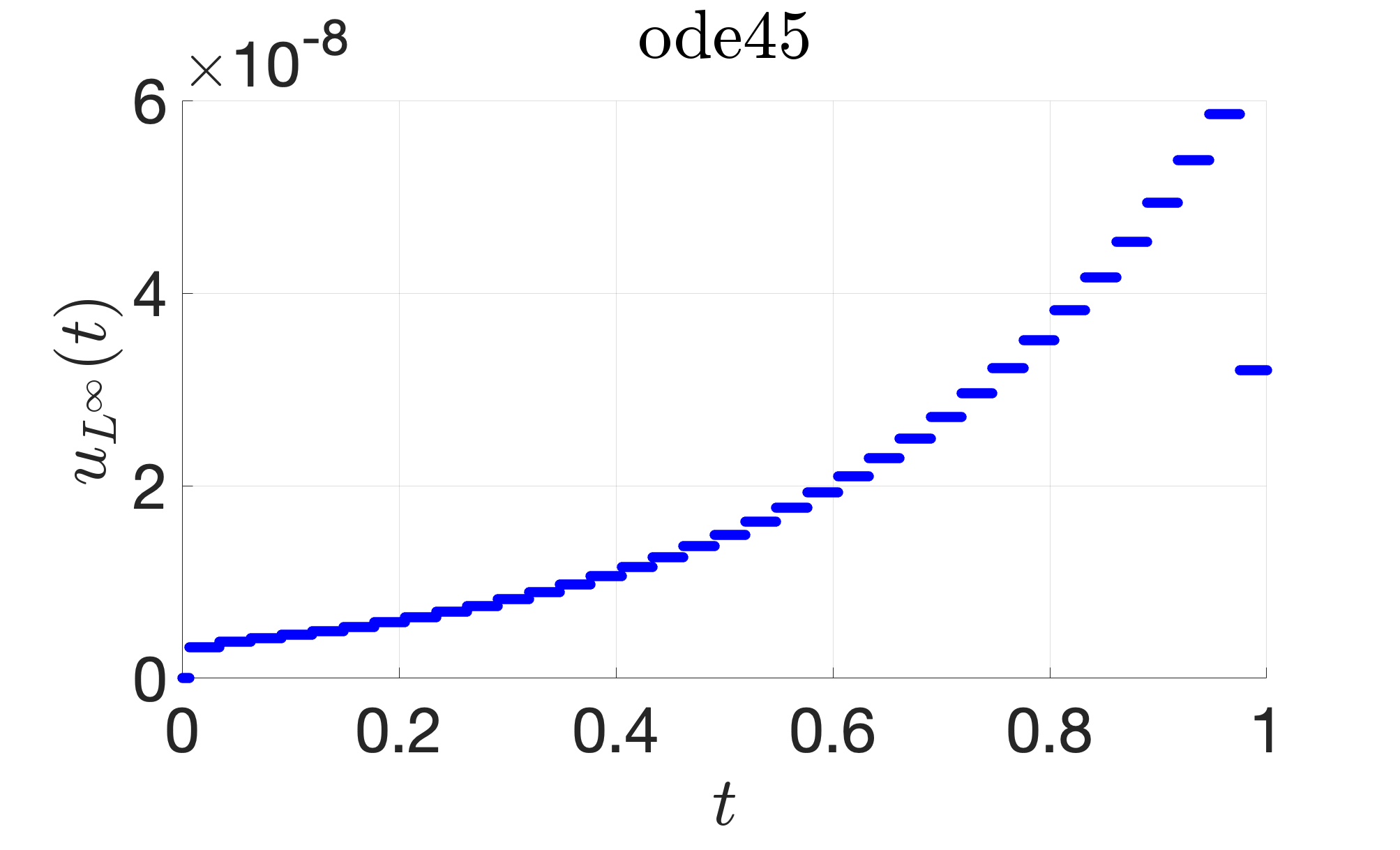} \\
\ \ \ \ \ \ (i)
\end{center}
\end{minipage}
\\[5mm]
\begin{minipage}{52mm}
\begin{center}
\includegraphics[width=57mm]{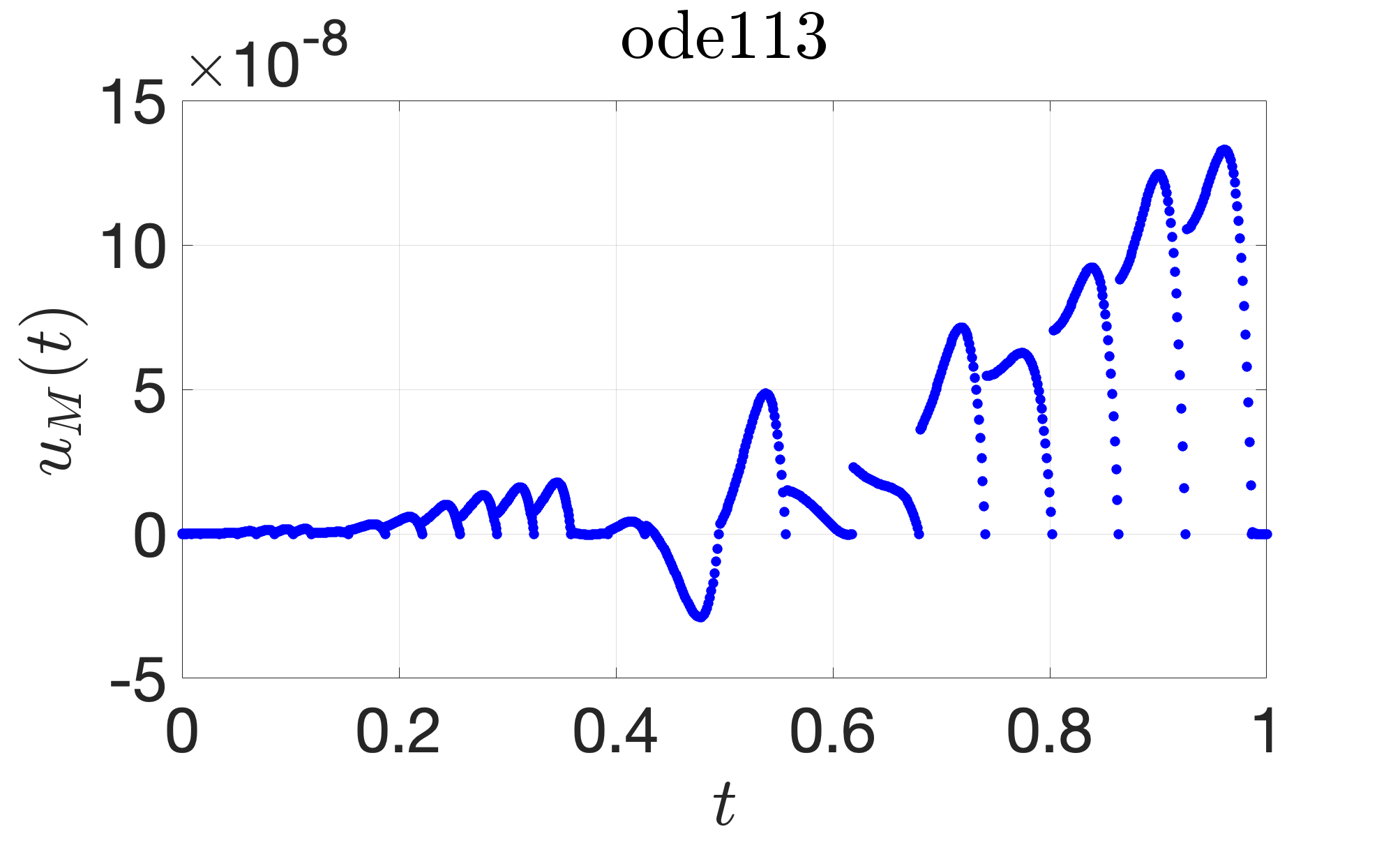} \\
\ \ \ \ \ \ (d)
\end{center}
\end{minipage}
\begin{minipage}{52mm}
\begin{center}
\includegraphics[width=57mm]{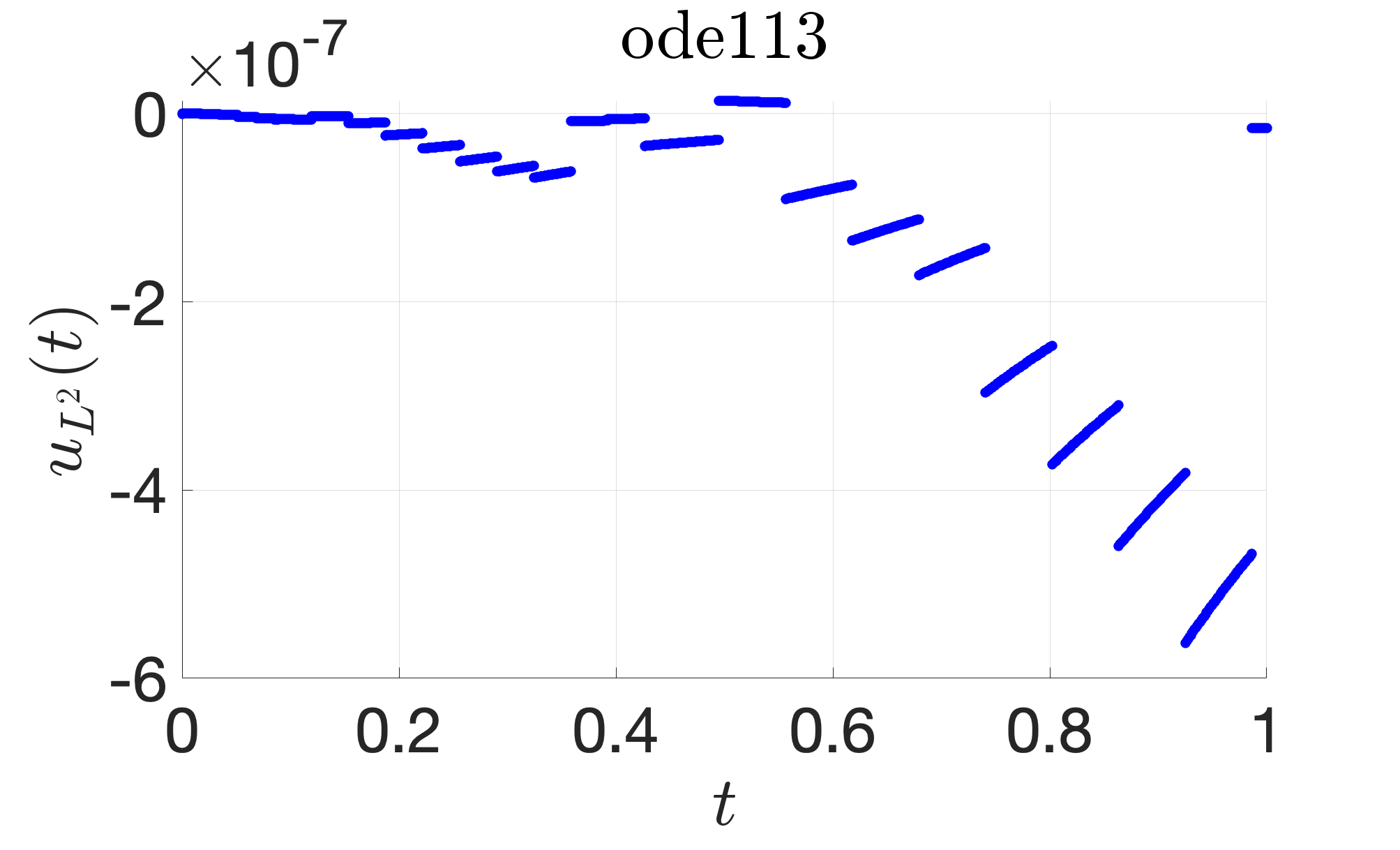} \\
\ \ \ \ \ \ (g)
\end{center}
\end{minipage}
\begin{minipage}{52mm}
\begin{center}
\includegraphics[width=57mm]{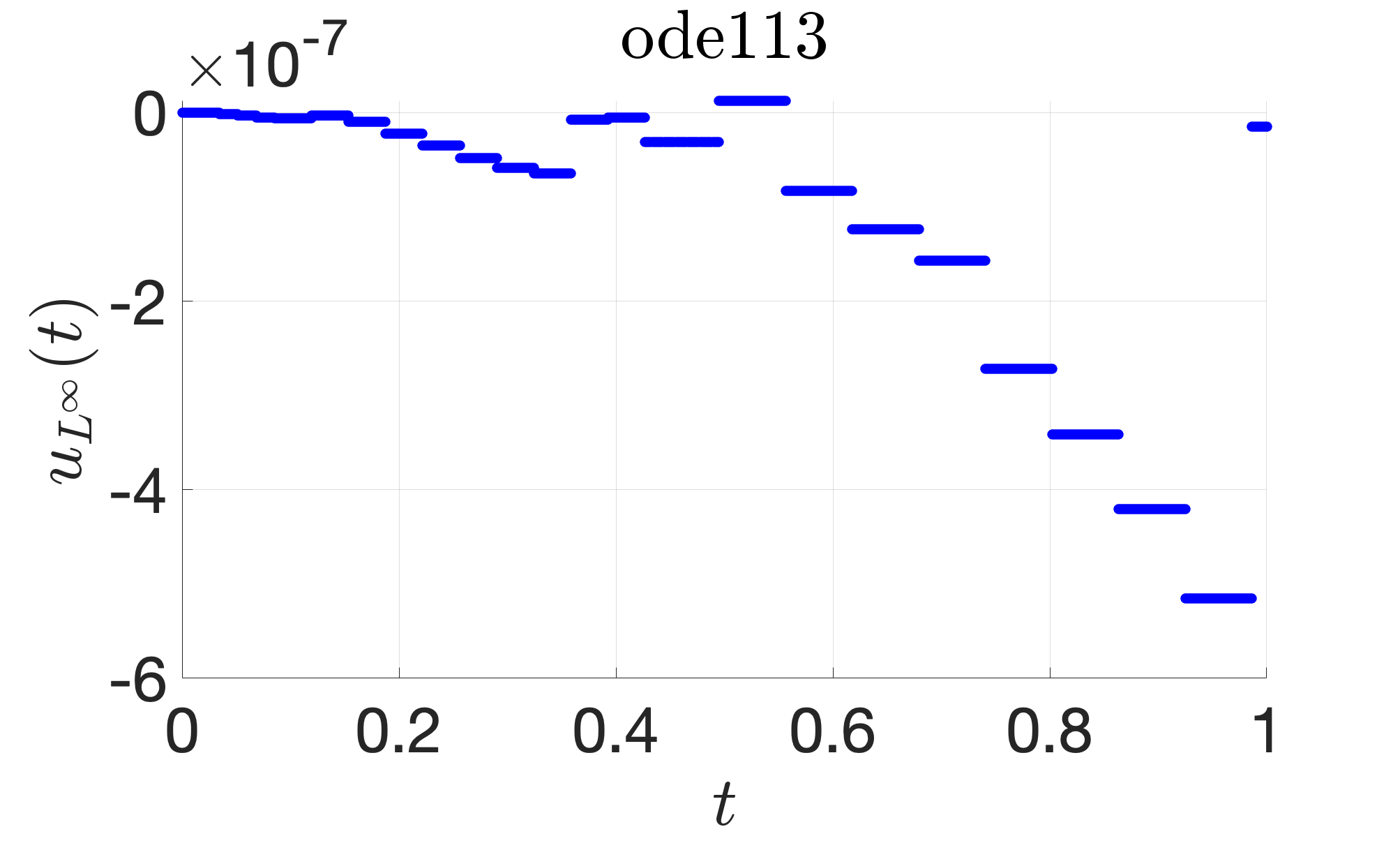} \\
\ \ \ \ \ \ (j)
\end{center}
\end{minipage}
\caption{\sf Example 1 -- (a) interpolant for $\dot{z}(t) = 3\,z(t)$, $0\le t\le1$, $z(0)=1$, and interpolant residuals computed using (b),(c),(d)~{\sc Matlab}'s {\tt deval}, (e),(f),(g) the $L^2$-residual in \eqref{simp_res_L2} and (h),(i),(j) the stage $L^\infty$-residual in \eqref{simp_res_Linfty}. The skeletons in each row were obtained using the {\sc Matlab} packages {\tt ode15s}, {\tt ode45} and {\tt ode113}, as indicated on the graphs. We see incidentally that the polynomial interpolant for \texttt{ode45} is indeed continuously differentiable, while those for \texttt{ode15s} and \texttt{ode113} are not, although the residual {\color{black} (in (d))} for \texttt{ode113} is zero as $s \to 1^-$ on each step.}
\label{fig:ex1}
\end{figure}

\newpage
\subsection{Example 2: a scalar nonlinear ODE}

Now, we consider another scalar but nonlinear ODE given by
\begin{equation}  \label{sqrteqn}
\dot{z}(t) = \sqrt{z(t)}\,,\quad z(0) = 1\,.
\end{equation}
This ODE is reminiscent of the model for the celebrated ``leaky bucket problem'' given by $\dot{z}(t) = -\sqrt{z(t)}$.  An elementary {\em relative residual analysis} of the leaky bucket problem can be found in~\cite{CorJan2016}, intended for classroom use. Because this scalar autonomous problem is so simple, all the optimization can be done analytically, and this can help students to understand the effect of numerical methods.  In this present paper we first carry out an analytical investigation similar to that of~\cite{CorJan2016}, in order to explain the idea, but then generalize the approach by using numerical optimization in order to deal with more complicated problems.

For the numerical methods \texttt{ode15s} and \texttt{ode45} this is an informative example, as we will see.  But ``unfortunately'' the numerical method that \texttt{ode113} uses actually gets the \textsl{exact} solution to this problem (apart from roundoff)---this is a coincidence, and arises in part because the exact solution is a polynomial of degree two.  Other methods, such as a second-order Taylor series method, also would get the exact solution.  This means that the minimal residual would be zero, if exact computation were used.  In practice, when we show the residuals that we actually compute, we will see only rounding error for the results of \texttt{ode113} on this example.  Nonetheless we have included the graphs for completeness.

\subsubsection{\boldmath{$L^2$}-minimization of the residual}
\label{subsubsec:sqrt_L2}

\begin{proposition}[\boldmath{$L^2$}-residual of \eqref{sqrteqn}]  \label{prop:sqrt_L2}
The interpolant, which solves Problem {\em (PL2)} with\linebreak $f(x(t),t) = \sqrt{x(t)}$, is a quadratic, namely
\begin{equation} \label{sqrt_x_L2}
x(t) = \frac{t^2}{4} + c_1\,t + c_2\,,
\end{equation}
for all $t\in[t_{i-1},t_i]$, and all $i = 1,\ldots,N$, where
\[
c_1 = \frac{1}{t_i-t_{i-1}}\,\left(z_i - z_{i-1} - \frac{t_i^2 - t_{i-1}^2}{4}\right)\quad\mbox{and}\quad c_2 = z_{i-1} - \frac{t_{i-1}^2}{4} - t_{i-1}\,c_1\,.
\]
The residual is in turn given by
\begin{equation} \label{sqrt_res_L2}
u(t) = \frac{t}{2} + c_1 - \sqrt{x(t)}\,,
\end{equation}
for all $t\in[t_{i-1},t_i]$, and all $i = 1,\ldots,N$.
\end{proposition}
\begin{proof}
The two-point boundary-value problem in \eqref{L2_state_eqn}--\eqref{L2_costate_eqn} in Theorem~\ref{theo:L2} reduces, for \eqref{sqrteqn} and $i=1,\ldots,N$, to
\begin{eqnarray*}
&&\dot{x}(t) = \sqrt{x(t)} - \lambda(t)\,,\quad x(t_{i-1}) = z_{i-1}\,,\ \ x(t_i) = z_i\,, \\
&&\dot{\lambda}(t) = -\frac{1}{2\,\sqrt{x(t)}}\,\lambda(t)\,,
\end{eqnarray*}
for all $t\in[t_{i-1},t_i]$.  Also, recall from Remark~\ref{rem:L2_opt_eqns} that $u(t) = -\lambda(t)$, and that the residual function $u$ will typically have jumps at $t_i$, $i=1,\ldots,N-1$.  Taking derivatives of both sides of the state equation, one gets
\[
\ddot{x}(t) = \frac{1}{2\,\sqrt{x(t)}}\,(\sqrt{x(t)} - \lambda(t)) - \dot{\lambda}(t) =
\frac{1}{2} - \frac{1}{2\,\sqrt{x(t)}}\,\lambda(t) - \dot{\lambda}(t) = \frac{1}{2}\,,
\]
integration of which, twice, yields \eqref{sqrt_x_L2}, with the constants $c_1$ and $c_2$ obtained from the boundary conditions for the state equation.  Finally, $u(t) = -\lambda(t) = \dot{x}(t) - \sqrt{x(t)}$, with the substitution of $\dot{x}(t) = t/2 + c_1$, furnishes \eqref{sqrt_res_L2}.
\end{proof}

\subsubsection{Stage {\boldmath$L^\infty$}-minimization of the residual}
\label{subsubsec:sqrt_Linf}

\begin{proposition}[Stage \boldmath{$L^\infty$}-residual of \eqref{sqrteqn}]  \label{prop:sqrt_Linfty}
The residual $u(t) =: \overline{u}^{[i]}$, a constant, for\linebreak $t\in[t_{i-1},t_i]$, and all $i = 1,\ldots,N$, which solves~{\em(PLinf)} with $f(x(t),t) = \sqrt{x(t)}$, is given by the implicit equation
\begin{equation}  \label{sqrt_res_Linfty}
\overline{u}^{[i]}\,\ln\left(\frac{\sqrt{z_i}+\overline{u}^{[i]}}{\sqrt{z_{i-1}}+\overline{u}^{[i]}}\right) = \sqrt{z_i}-\sqrt{z_{i-1}} - \frac{t_i-t_{i-1}}{2}\,.
\end{equation}
A corresponding implicit solution for $x(t)$, for all $t\in[t_{i-1},t_i]$, is then
\begin{equation}  \label{sqrt_x_Linfty}
\sqrt{x(t)} - \overline{u}^{[i]}\,\ln\left(\frac{\sqrt{x(t)} + \overline{u}^{[i]}}{\sqrt{z_{i-1}} + \overline{u}^{[i]}}\right) = \sqrt{z_{i-1}} + \frac{t-t_{i-1}}{2}\,.
\end{equation}
\end{proposition}
\begin{proof}
The DE \eqref{Linf_adjoint_DE2} for the adjoint variable in Theorem~\ref{theo:Linf} can be written for this example simply as $\dot{\lambda}(t) = -(1/(2\,\sqrt{x(t)}))\,\lambda(t)$, the solution of which is
\[
\lambda(t) = c\,e^{-\int 1/(2\,\sqrt{x(t)})\,dt}\,,
\]
where $c$ is a real constant. Clearly, $c\neq0$, as otherwise $\lambda(t) = 0$ for all $t\in[t_{i-1},t_i]$ and so the optimal control is singular over $[t_{i-1},t_i]$, which is not possible by Lemma~\ref{lem:tot-sing}.  With $c\neq0$, one has that $\lambda(t) \neq 0$ and does not change sign; therefore, the optimal control nonsingular and $u(t) = -\sgn(\lambda(t))\,\alpha = \overline{u}^{[i]}$, a constant for all $t\in[t_{ {i-1}},t_{ i}]$.  Then the state equation in \eqref{Linf_x_eqn2} simply becomes
\[
\dot{x}(t) - \sqrt{x(t)} = \overline{u}^{[i]}\,\quad x(t_{i-1}) = z_{i-1}\,,\ \ x(t_i) = z_i\,.
\]
The DE above can be re-arranged, integrated, and manipulations can be carried out to get
\begin{eqnarray}
&& \int_{t_{i-1}}^t \frac{dx}{\sqrt{x(t)} + \overline{u}^{[i]}} = \int_{t_{i-1}}^t dt \nonumber \\[2mm]
&& 2\left[(\sqrt{x(t)} - \overline{u}^{[i]}\,\ln\left(\sqrt{x(t)} + \overline{u}^{[i]}\right) - (\sqrt{z_{i-1}} - \overline{u}^{[i]}\,\ln\left(\sqrt{z_{i-1}} + \overline{u}^{[i]}\right)\right] = t - t_{i-1} \label{x_expanded} \\
&& \sqrt{x(t)} - \sqrt{z_{i-1}} - \overline{u}^{[i]}\,\ln\left(\frac{\sqrt{x(t)} + \overline{u}^{[i]}}{\sqrt{z_{i-1}} - \overline{u}^{[i]}}\right) = \frac{t - t_{i-1}}{2} \nonumber
\end{eqnarray}
which yields~\eqref{sqrt_x_Linfty}.  Evaluating \eqref{sqrt_x_Linfty} at $t=t_i$ and re-arranging gives \eqref{sqrt_res_Linfty}.
\end{proof}

\begin{remark}[Lambert $W$ function]  \rm
Since $x(t) > 0$, after manipulations on \eqref{x_expanded}, the interpolating curve $x(t)$, implicitly given in~\eqref{sqrt_x_Linfty}, can alternatively be expressed as
\[
x(t) = \left\{-\overline{u}^{[i]}\,\left(W\left(-\frac{\overline{u}^{[i]}}{e}\left[\frac{t - t_{i-1}}{2} + \sqrt{z_{i-1}} - \overline{u}^{[i]}\,\ln\left(\sqrt{z_{i-1}} + \overline{u}^{[i]}\right)\right]^{-1/\overline{u}^{[i]}}\right) + 1\right)\right\}^2\,,
\]
for all $t\in[t_{i-1},t_i]$, $i=1,\ldots,N$, where $W(\cdot)$ is the Lambert~$W$ function. (The {\em Lambert~$W$ function} $w$ is defined implicitly as $w =W(a)$, where $w\,e^w = a$; see \cite{CorHarJefKnu96}.)
\proofbox
\end{remark}

\subsubsection{A graphical illustration of the residuals}

Figure~\ref{fig:ex2} depicts the graphs of the residuals for various criteria.  As with Example~1, to obtain the numerical skeleton, the {\sc Matlab} solvers {\tt ode15s}, {\tt ode45} and {\tt ode113} have been used.  The number of mesh points
corresponding to the solutions produced by each solver turns out to be $N = 25$, $12$ and $22$, respectively.  With each numerical skeleton, interpolating curves were generated using {\sc Matlab}'s {\tt deval}, as well as through the $L^2$- and stage $L^\infty$-minimization of the ODE residual, as given in \eqref{sqrt_x_L2} and \eqref{sqrt_x_Linfty}.  The residuals for each case and approach, nine of them altogether, are depicted in Figure~\ref{fig:ex2}, where the residuals resulting from $L^2$- and stage $L^\infty$-minimization are computed using \eqref{sqrt_res_L2} and \eqref{sqrt_res_Linfty}.  { The reader can have access to the {\sc Matlab} code ({\tt residuals{\_}zdot{\_}sqrtz.m}) that generated the graphs in Figure~\ref{fig:ex2} at \url{https://github.com/rcorless/OptimalResiduals}.}

As in Example~1, with {\tt ode15s} and {\tt ode45}, both $\max_{0\le t\le 1}|u_{L^2}(t)|$ and $\max_{0\le t\le 1}|u_{L^\infty}(t)|$ are less than $\max_{0\le t\le 1}|u_M(t)|$.
Overall, $L^\infty$-minimization seems to result in slightly smaller residuals than $L^2$-minimization.

All cases in Figure~\ref{fig:ex2} considered, the {\tt ode45} skeleton with the stage $L^\infty$-minimum interpolation facilitates the smallest residual, by an order of magnitude of two compared with the residual obtained by {\sc Matlab}'s {\tt deval}.

\afterpage{\clearpage}
\begin{figure}[t]
\begin{minipage}{52mm}
\begin{center}
\hspace*{0mm}
\includegraphics[width=55mm]{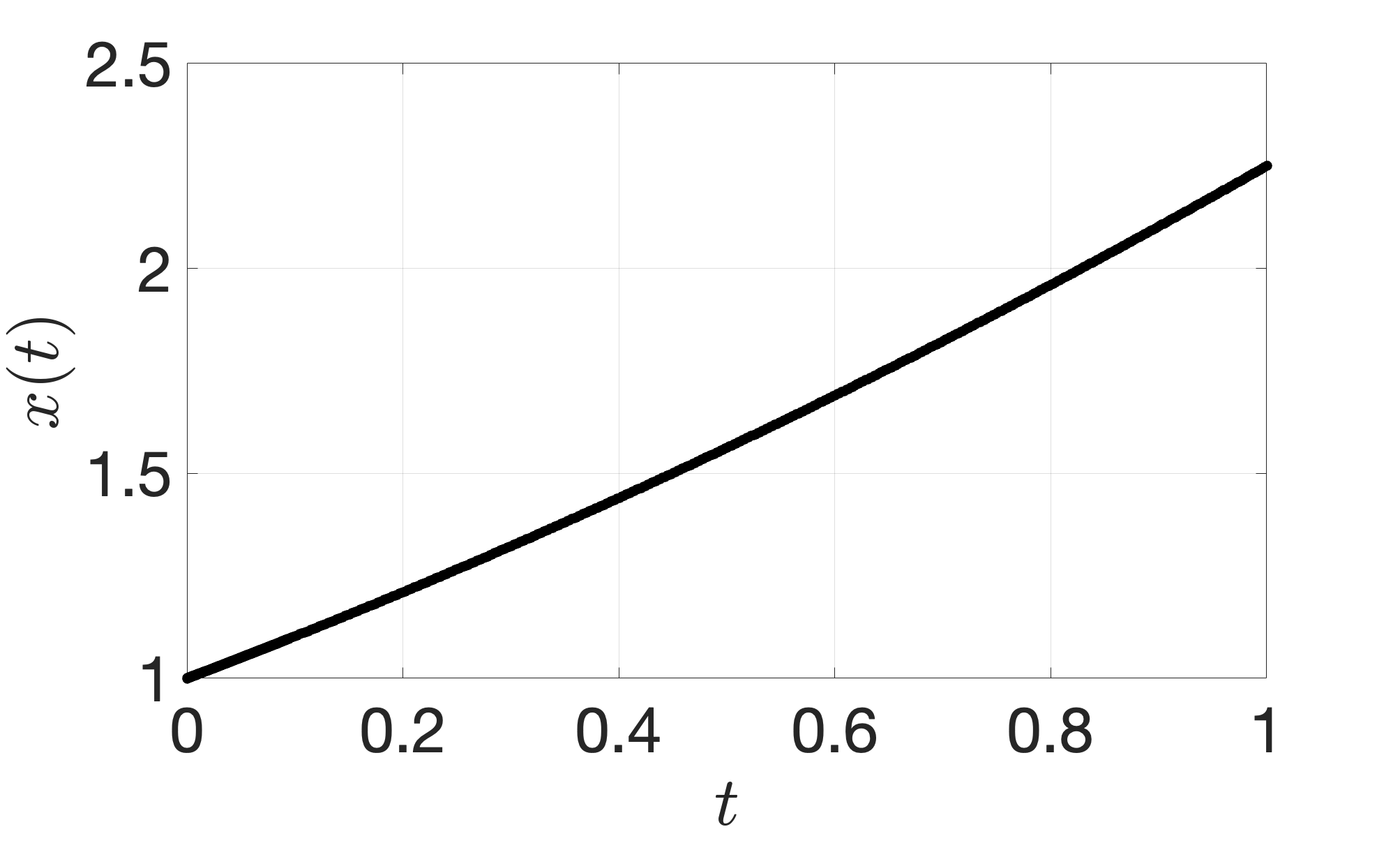} \\
\ \ \ \ \ (a)
\end{center}
\end{minipage}
\\[5mm]
\begin{minipage}{52mm}
\begin{center}
\includegraphics[width=57mm]{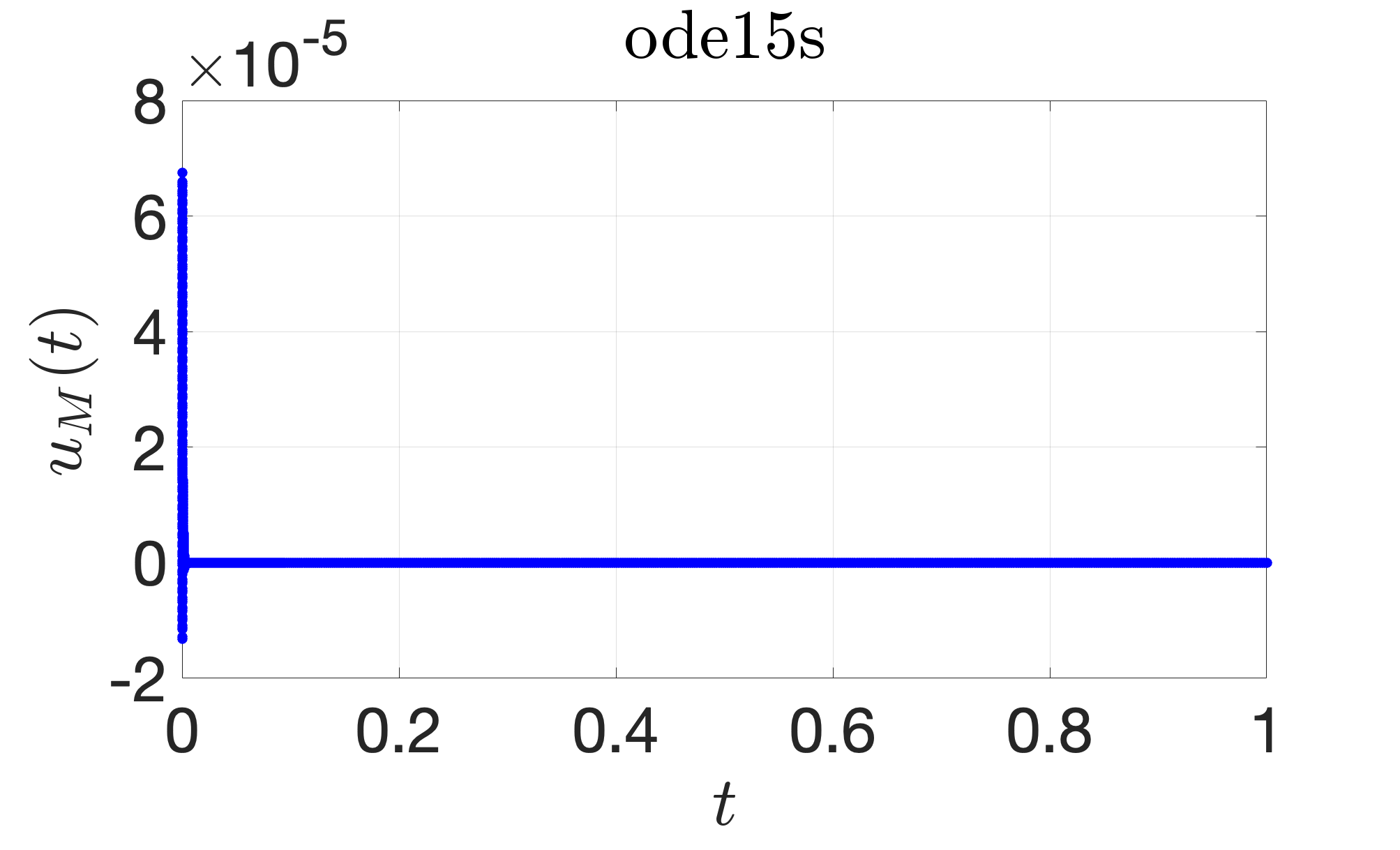} \\
\ \ \ \ \ \ (b)
\end{center}
\end{minipage}
\begin{minipage}{52mm}
\begin{center}
\includegraphics[width=57mm]{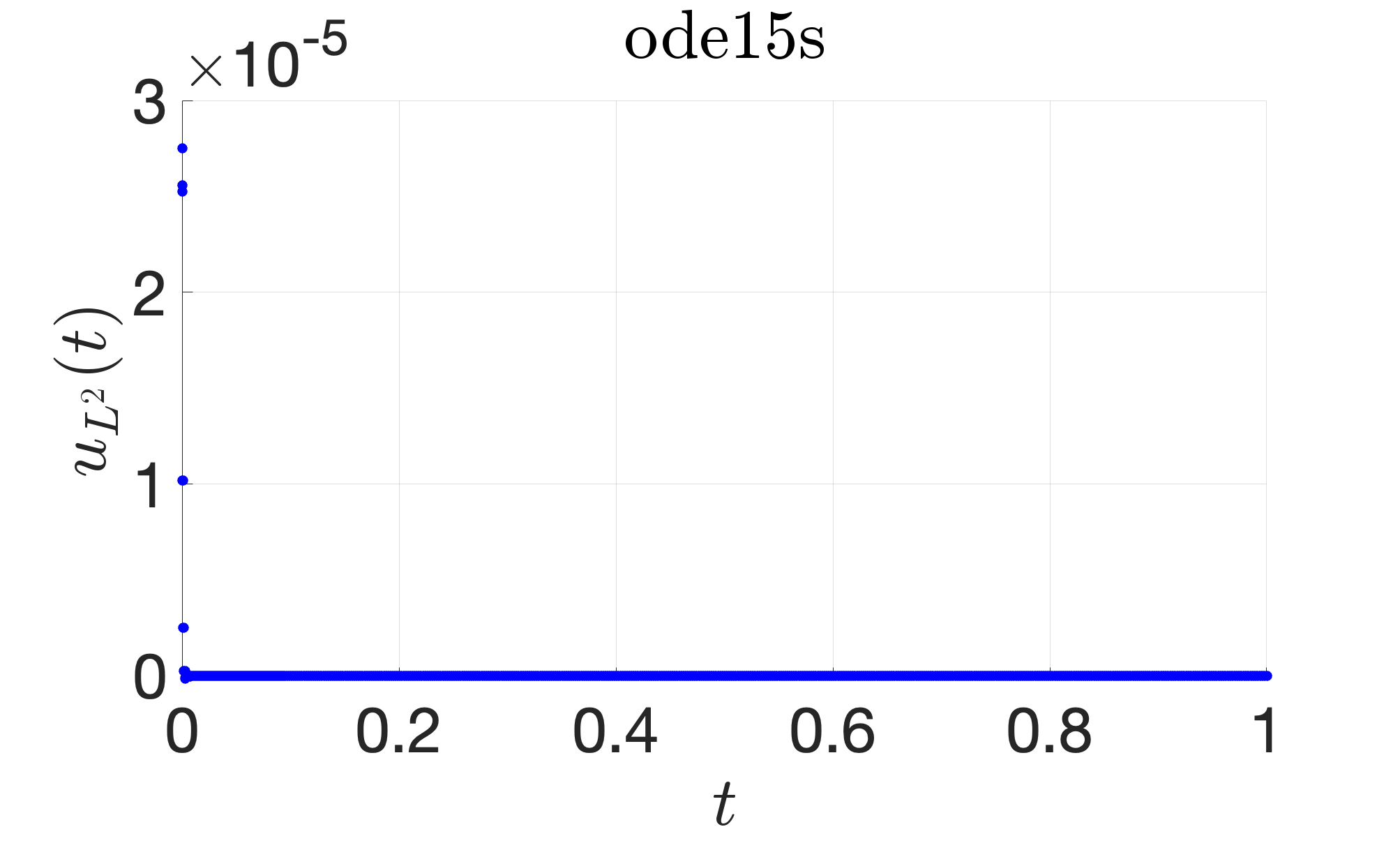} \\
\ \ \ \ \ \ (e)
\end{center}
\end{minipage}
\begin{minipage}{52mm}
\begin{center}
\includegraphics[width=57mm]{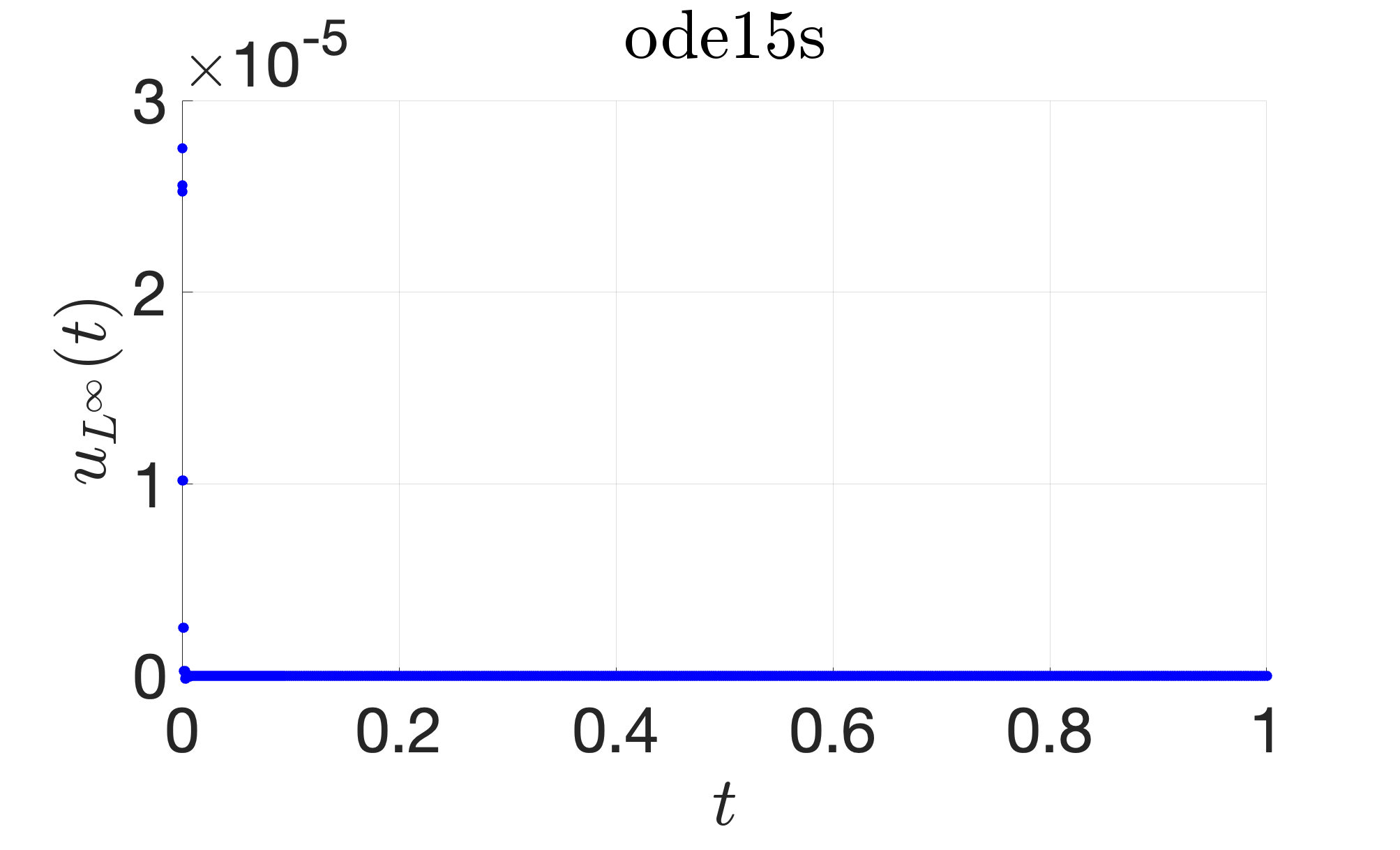} \\
\ \ \ \ \ \ (h)
\end{center}
\end{minipage}
\\[5mm]
\begin{minipage}{52mm}
\begin{center}
\includegraphics[width=57mm]{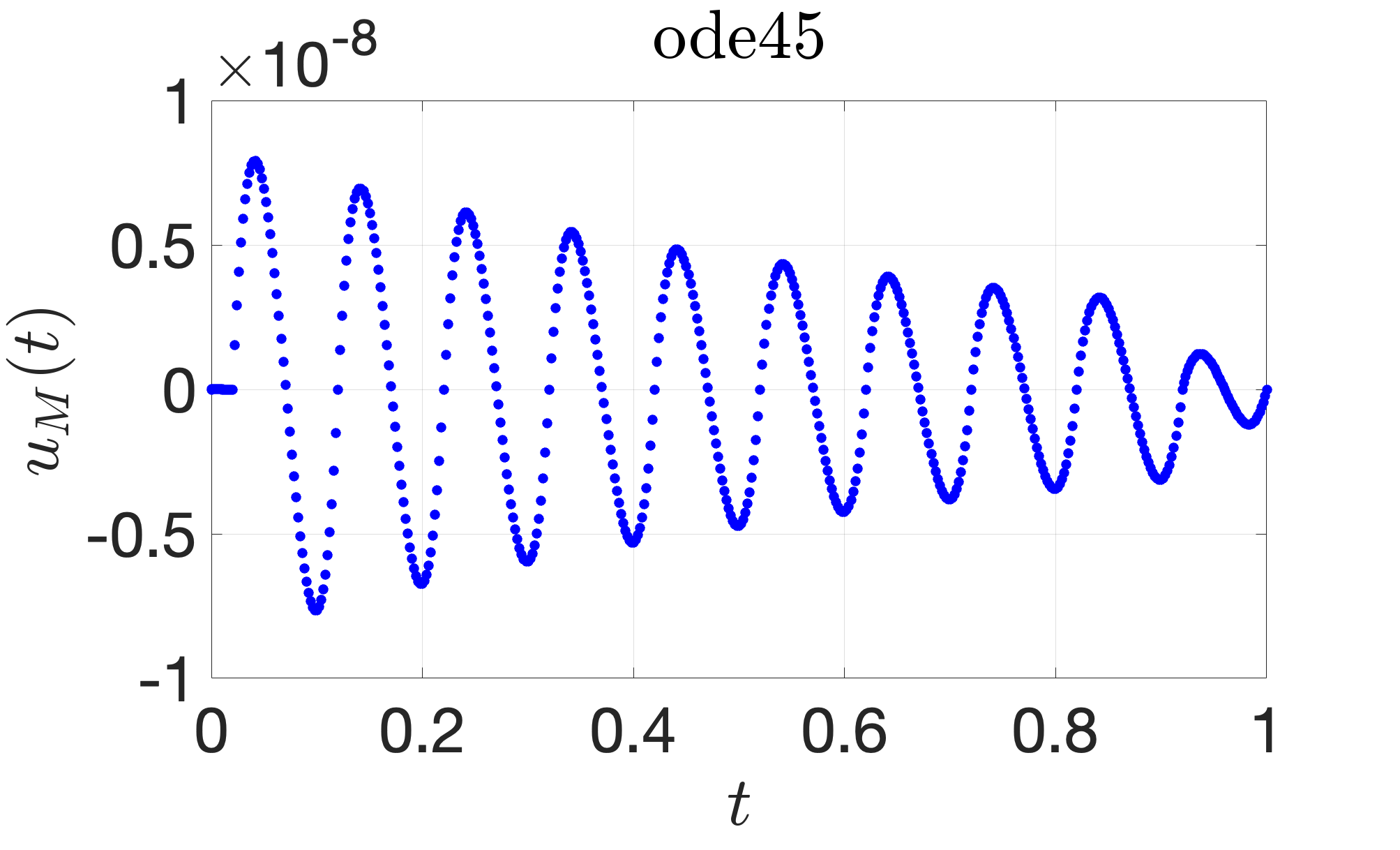} \\
\ \ \ \ \ \ (c)
\end{center}
\end{minipage}
\begin{minipage}{52mm}
\begin{center}
\includegraphics[width=57mm]{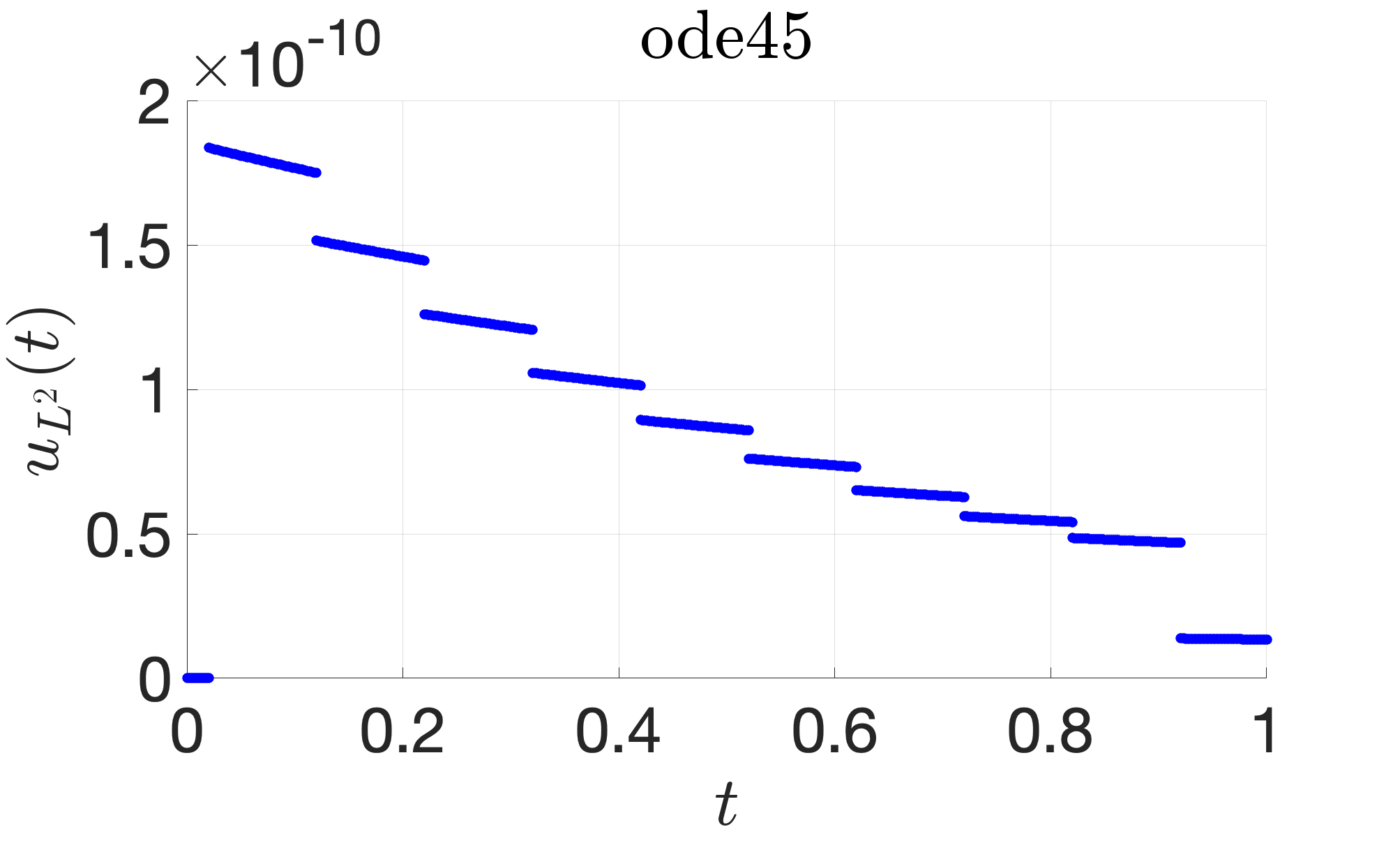} \\
\ \ \ \ \ \ (f)
\end{center}
\end{minipage}
\begin{minipage}{52mm}
\begin{center}
\includegraphics[width=57mm]{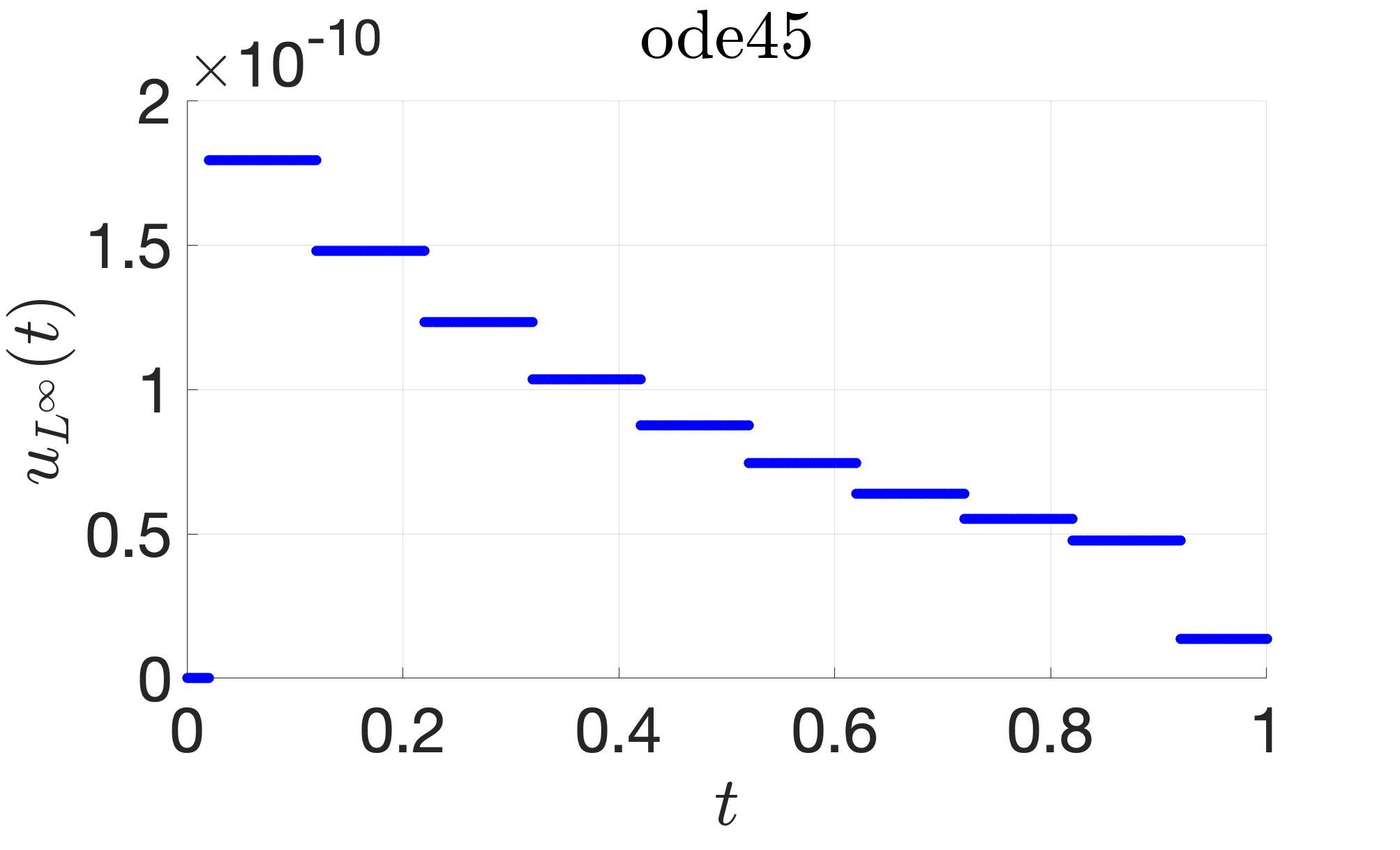} \\
\ \ \ \ \ \ (i)
\end{center}
\end{minipage}
\\[5mm]
\begin{minipage}{52mm}
\begin{center}
\includegraphics[width=57mm]{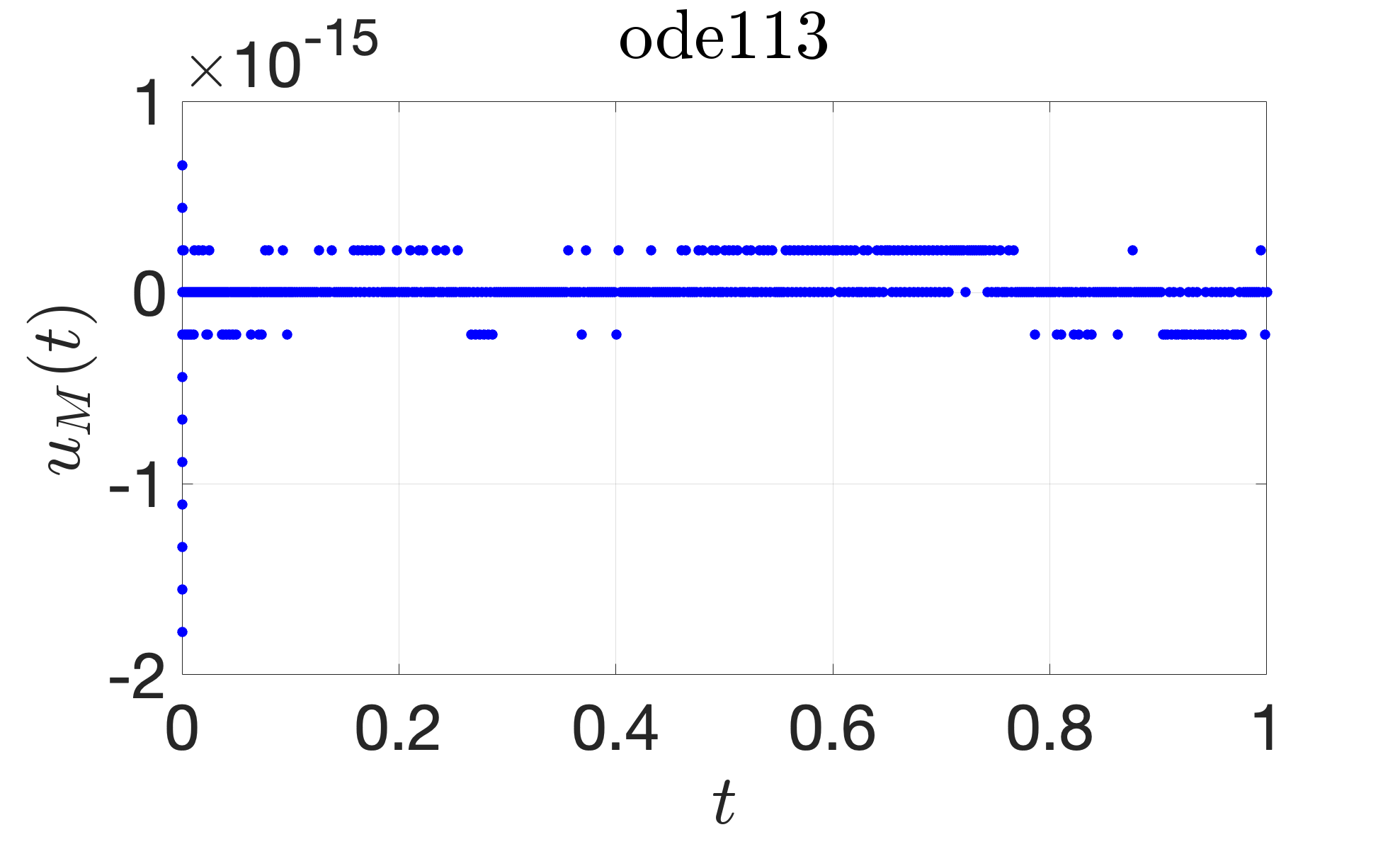} \\
\ \ \ \ \ \ (d)
\end{center}
\end{minipage}
\begin{minipage}{52mm}
\begin{center}
\includegraphics[width=57mm]{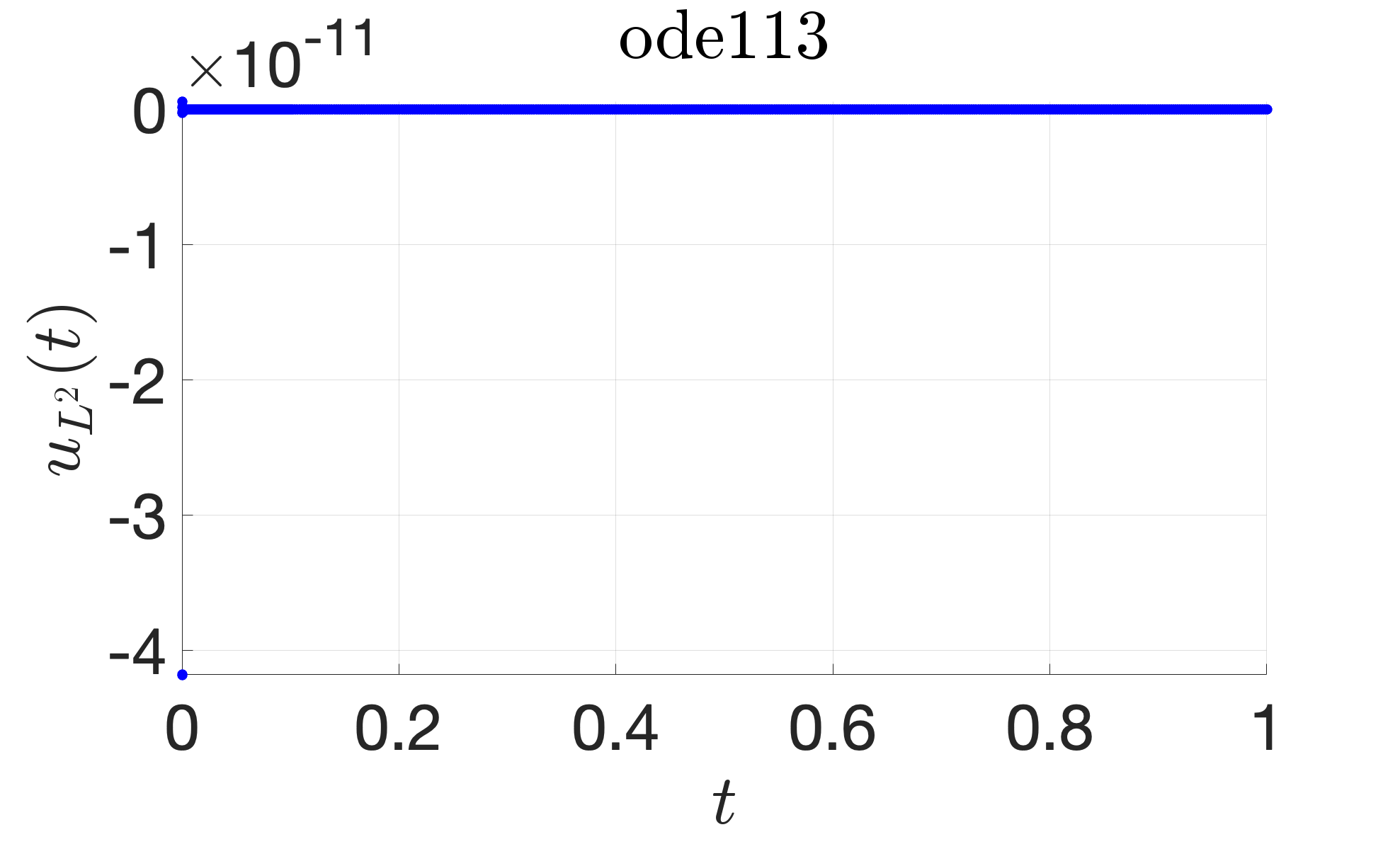} \\
\ \ \ \ \ \ (g)
\end{center}
\end{minipage}
\begin{minipage}{52mm}
\begin{center}
\includegraphics[width=57mm]{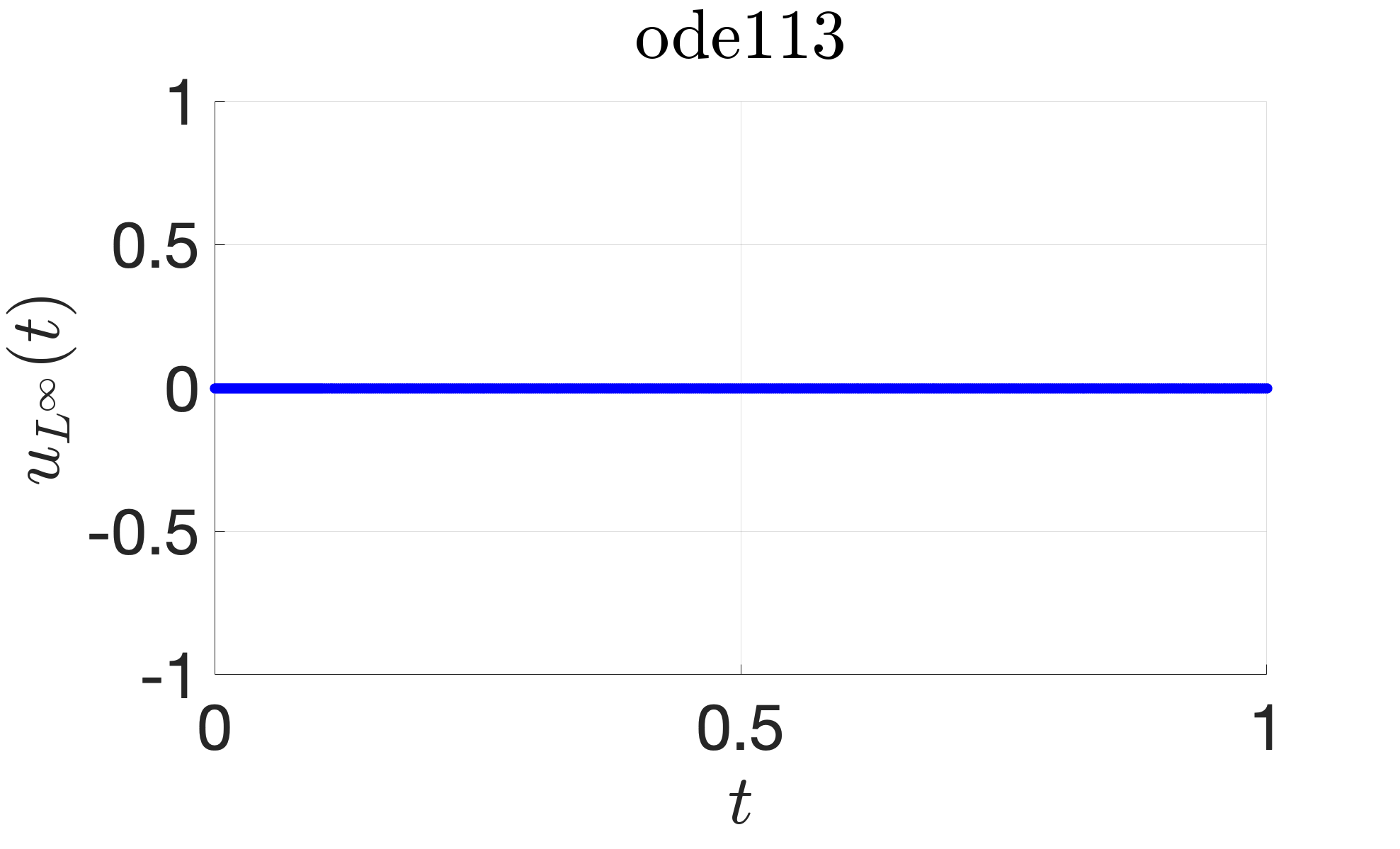} \\
\ \ \ \ \ \ (j)
\end{center}
\end{minipage}
\caption{\sf Example 2 -- (a) interpolant for $\dot{z}(t) = \sqrt{z(t)}$, $0\le t\le1$, $z(0)=1$, and interpolant residuals computed using (b),(c),(d)~{\sc Matlab}'s {\tt deval}, (e),(f),(g) the $L^2$-residual in \eqref{sqrt_res_L2} and (h),(i),(j) the stage $L^\infty$-residual in \eqref{sqrt_res_Linfty}. The skeletons in each row were obtained using the {\sc Matlab} packages {\tt ode15s}, {\tt ode45} and {\tt ode113}, as indicated on the graphs. Figures (d), (g), and (j) are not very informative, because \texttt{ode113} actually gets the exact solution to this problem: the minimal residual is zero, apart from roundoff error.}
\label{fig:ex2}
\end{figure}

\newpage
\subsection{Example 3: the Van der Pol equation}

Next, we consider the Van der Pol oscillator, a second order ODE used to model electrical circuits.  It can be represented as a system of two first-order nonlinear ODEs as follows.
\begin{equation} \label{ode:vdp}
\begin{array}{ll}
\dot{z}_1(t) = z_2(t)\,,  &\  z_1(0) = z_{1,0}\,,  \\[2mm]
\dot{z}_2(t) = -z_1(t) - (z_1(t)^2 - 1)\,z_2(t)\,, &\  z_2(0) = z_{2,0}\,,
\end{array}
\end{equation}
for $t\in[t_0,t_f]$.

Unlike Examples~1 and 2, we cannot obtain analytical solutions for the interpolating curves that minimize the $L^2$- and stage $L^\infty$-residuals of equation~\eqref{ode:vdp}.  As explained at the beginning of Section~\ref{sec:examples}, we can instead solve a direct discretization of the optimization problems~(PL2a) and (PLinfb) to get approximate solutions~\cite{Betts2020}.

For discretizing Problem~(PL2a), we use the trapezoidal rule with 2000 subintervals in each stage.  For discretizing Problem~(PLinfb), on the other hand, we employ Euler's method, since the optimal control (or the optimal residual) is in general of bang--bang type and therefore is not even continuous, making the interpolants continuous but non-differentiable.  Although Euler's method (in theory) requires differentiability of the solution curves, it has previously been successfully implemented to solve problems with bang--bang and singular types of optimal control, in particular in estimating the switching locations (or the ``switching times'') of the control function---see e.g.~\cite{BurKayMou2024, Kaya2022, KayMau2014} and the references therein.  With all the numerical skeletons generated by {\tt ode15s}, {\tt ode45} and {\tt ode113}, we use 2000 subintervals for Euler's method as well in each stage.

Figure~\ref{fig:ex3a} shows the graphs of the residuals resulting from the interpolants of the numerical skeleton of 110 nodes generated by {\tt ode15s}.  In the graphs, subscripts 1 and 2 are used to denote the two components of the residual vector.  All graphs in Figure~\ref{fig:ex3a}(a)--(f) considered, $\max_{0\le t\le 1} \|u_{L^\infty}(t)\|$ appears to be smaller than both $\max_{0\le t\le 1} \|u_M(t)\|$ and $\max_{0\le t\le 1} \|u_{L^2}(t)\|$.  The optimality condition \eqref{Linf_opt_control_u} seems to be verified:  Figure~\ref{fig:ex3b}, especially when zoomed in, shows that $(v_{L^\infty})_i(t) = \sgn(\lambda_i(t))$, $i = 1,2$, as required by the theory.    We provide the {\sc Matlab} fig as well as portable network graphics (png) files of those graphs involving switchings not only in Figure~\ref{fig:ex3b}, but also Figures~\ref{fig:ex3d} and \ref{fig:ex3f}, for the reader to be able to zoom in and explore, at \url{https://github.com/rcorless/OptimalResiduals}.

Figure~\ref{fig:ex3c} presents the graphs of the residuals resulting from the interpolants of the numerical skeleton of 36 nodes generated by {\tt ode45}.  In this case, $\max_{0\le t\le 1} \|u_{L^2}(t)\|$ is the smallest, and $\max_{0\le t\le 1} \|u_{L^\infty}(t)\|$ the largest.  We note that, as can be seen in Figure~\ref{fig:ex3d}, the optimality condition \eqref{Linf_opt_control_u} seems to be verified.

In Figure~\ref{fig:ex3e}, we provide the graphs of the residuals resulting from the interpolants of the numerical skeleton of 56 nodes generated by {\tt ode113}.  As in the case of {\tt ode45}, this time as well, $\max_{0\le t\le 1} \|u_{L^2}(t)\|$ is the smallest.  We also observe in Figure~\ref{fig:ex3f} that the graphs of the adjoint variables $\lambda_1(t)$ and $\lambda_2(t)$, and the optimal controls $v_1$ and $v_2$ seem to verify the optimality condition \eqref{Linf_opt_control_u}.

In each of the cases above, $(v_{L^\infty})_1(t)$ and $(v_{L^\infty})_2(t)$ seem to have at most one switching in each stage.

\afterpage{\clearpage}
\begin{figure}
\begin{minipage}{80mm}
\begin{center}
\includegraphics[width=80mm]{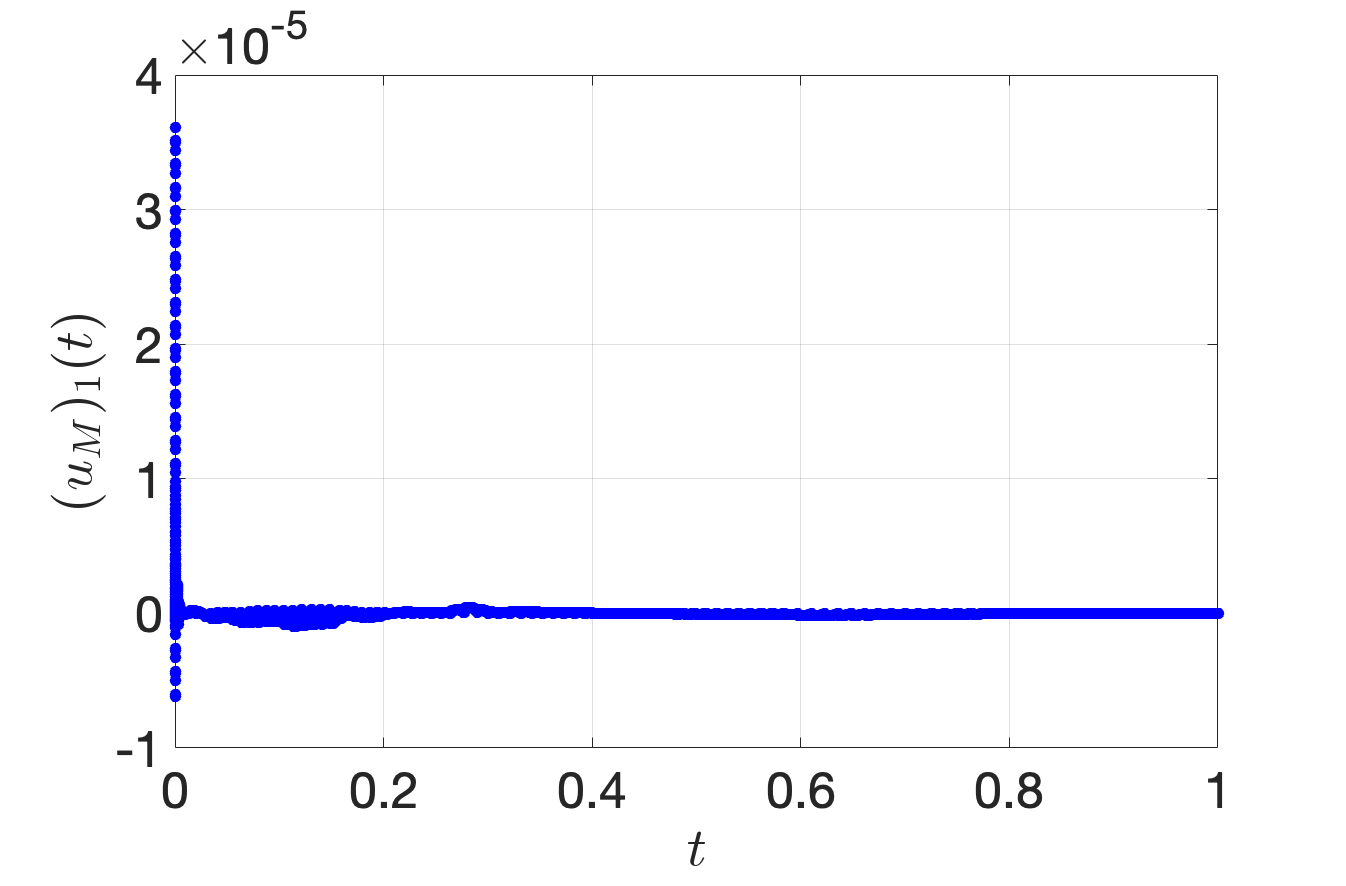} \\[3mm]
(a)
\end{center}
\end{minipage}
\begin{minipage}{80mm}
\begin{center}
\includegraphics[width=80mm]{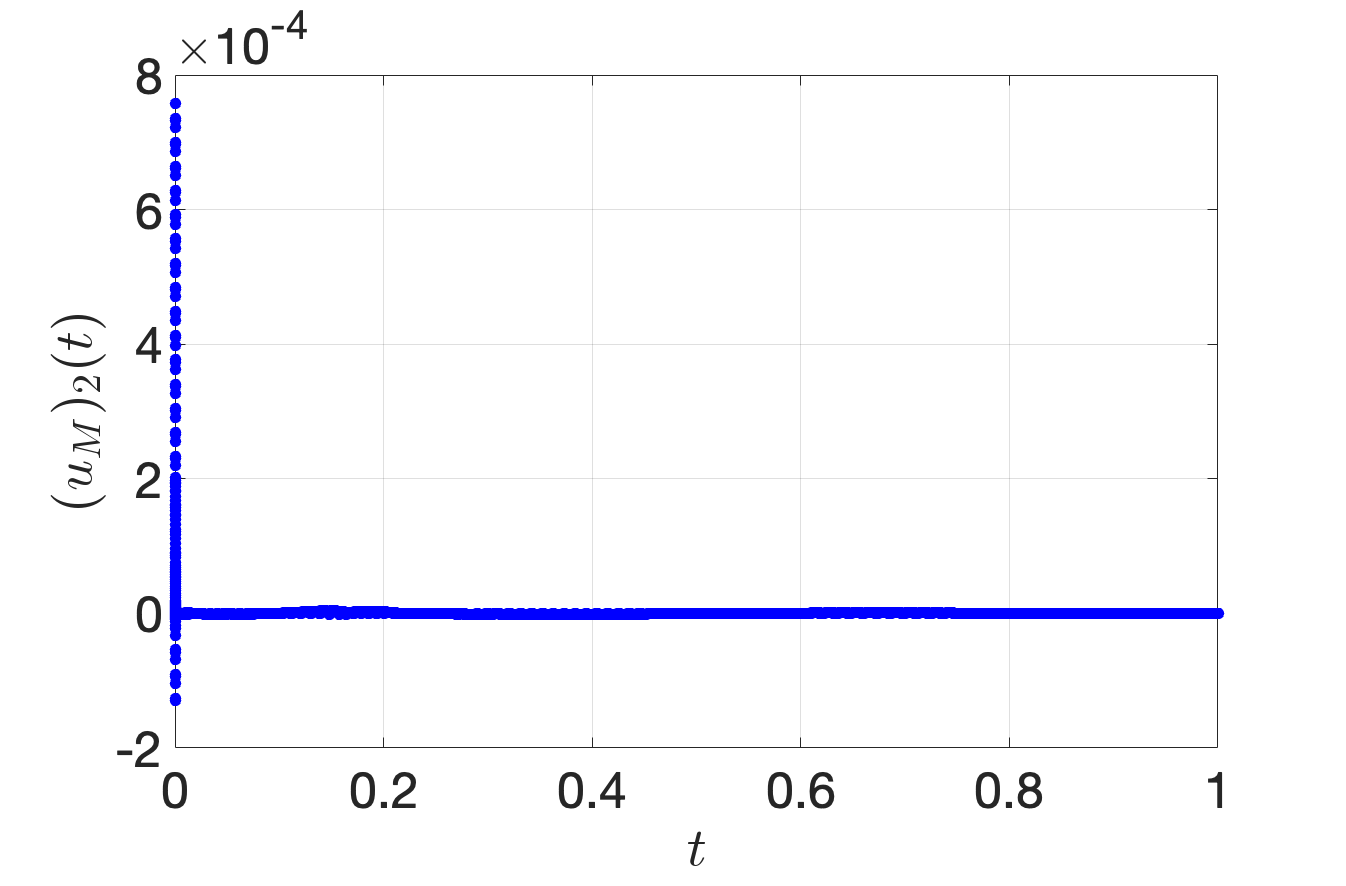} \\[3mm]
(b)
\end{center}
\end{minipage}
\\[10mm]
\begin{minipage}{80mm}
\begin{center}
\includegraphics[width=80mm]{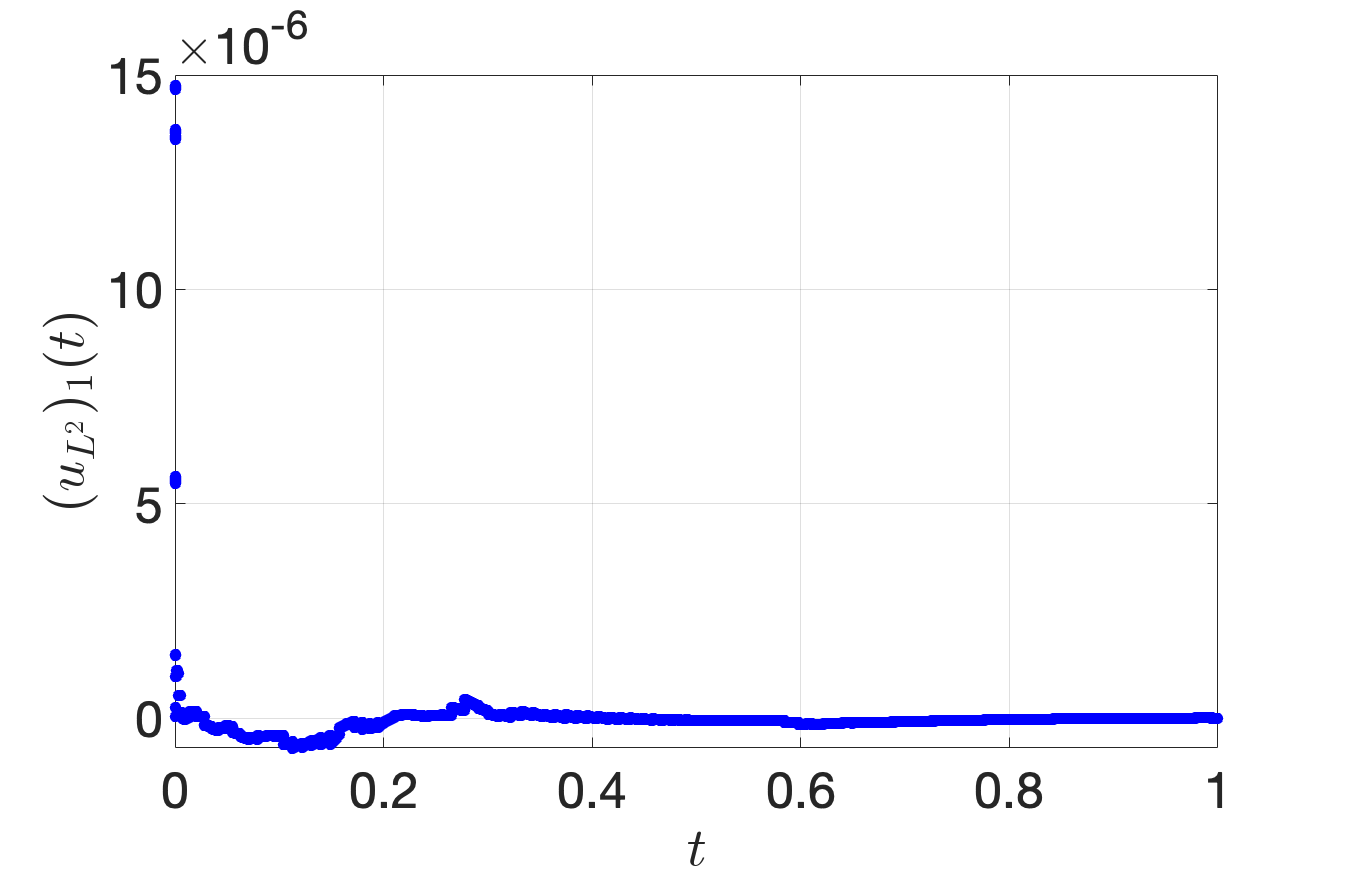} \\[3mm]
(c)
\end{center}
\end{minipage}
\begin{minipage}{80mm}
\begin{center}
\includegraphics[width=80mm]{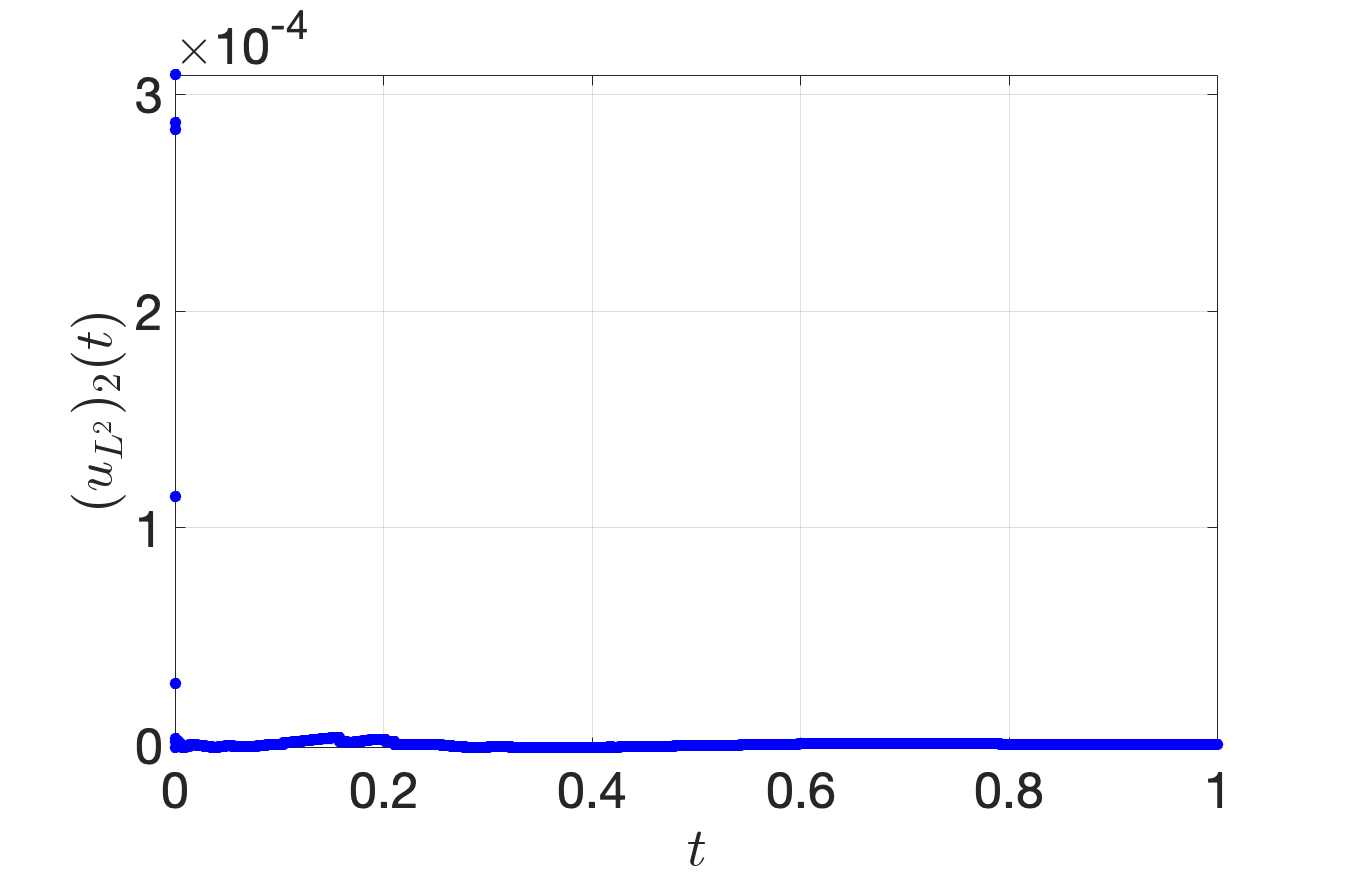} \\[3mm]
(d)
\end{center}
\end{minipage}
\\[10mm]
\begin{minipage}{80mm}
\begin{center}
\includegraphics[width=80mm]{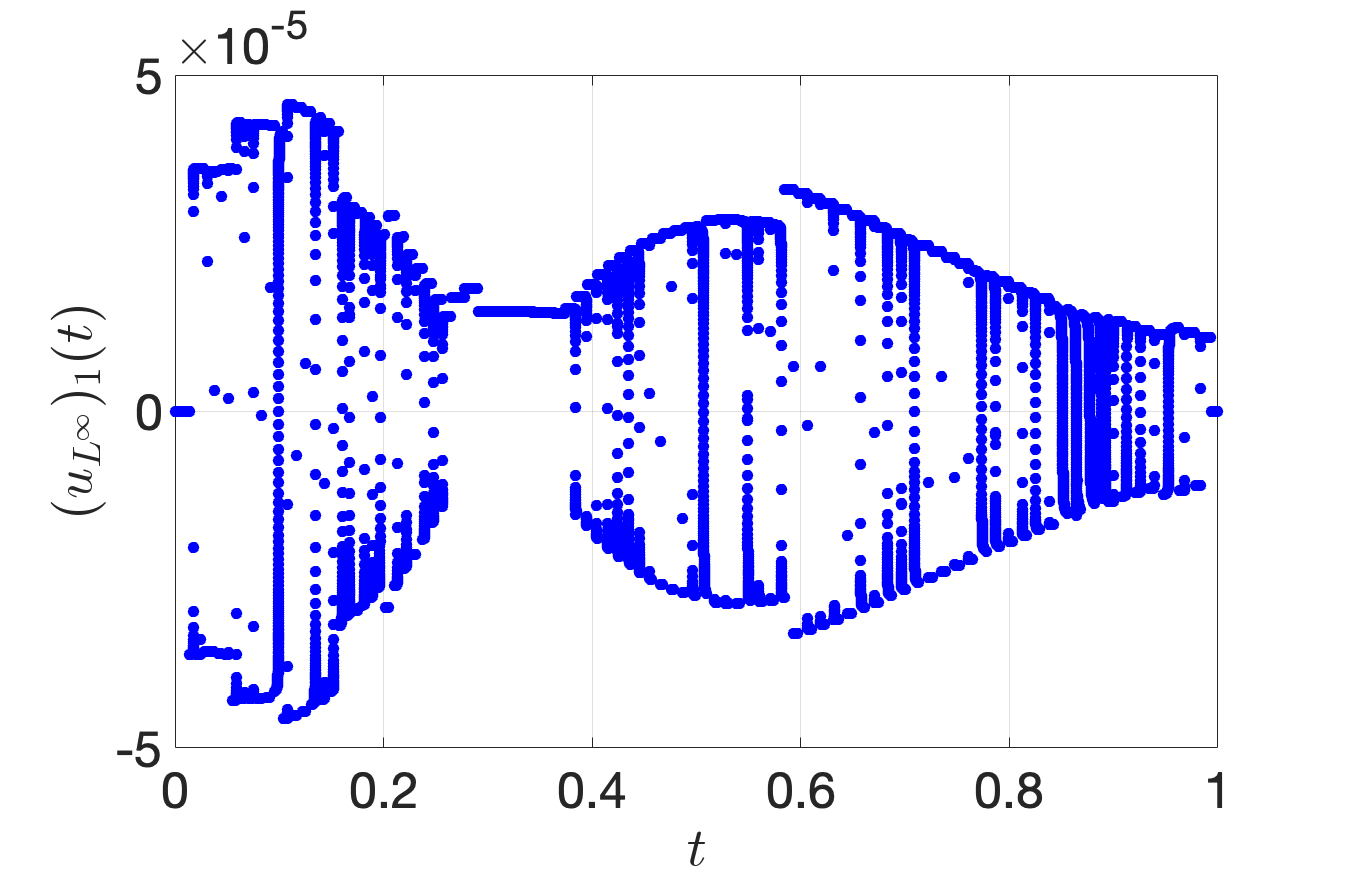} \\[3mm]
(e)
\end{center}
\end{minipage}
\begin{minipage}{80mm}
\begin{center}
\includegraphics[width=80mm]{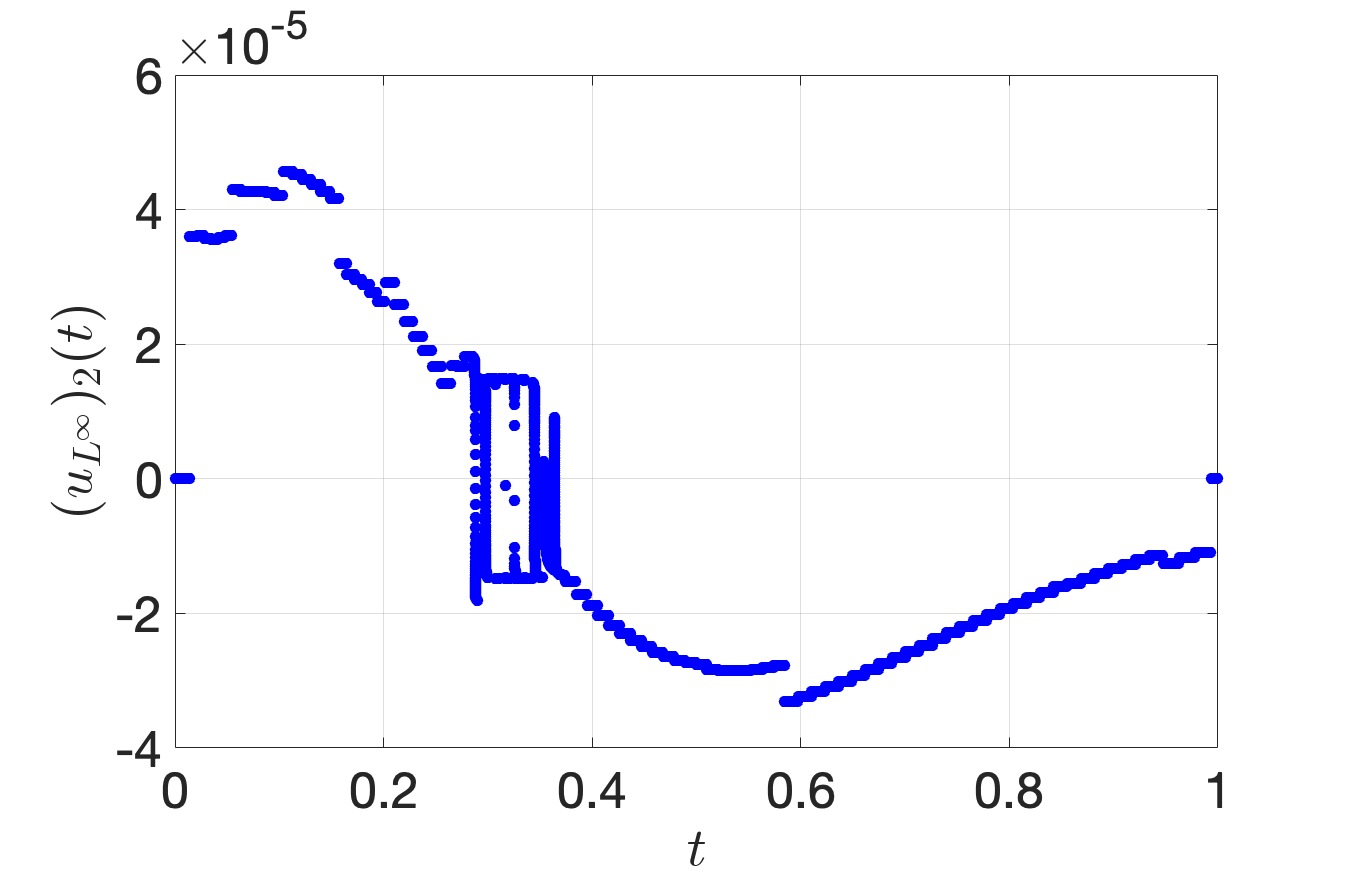} \\[3mm]
(f)
\end{center}
\end{minipage}
\
\caption{\sf Example 3 -- interpolant residual vector components
computed for equation~\eqref{ode:vdp} and {\sc Matlab}'s {\tt ode15s}, with $0\le t\le1$ and $x(0)=(-1,-3)$, using (a)--(b)~{\sc Matlab}'s {\tt deval}, (c)--(d) $L^2$-minimization and (e)--(f)~stage $L^\infty$-minimization.}
\label{fig:ex3a}
\end{figure}

\afterpage{\clearpage}
\begin{figure}[t!]
\begin{minipage}{80mm}
\begin{center}
\includegraphics[width=80mm]{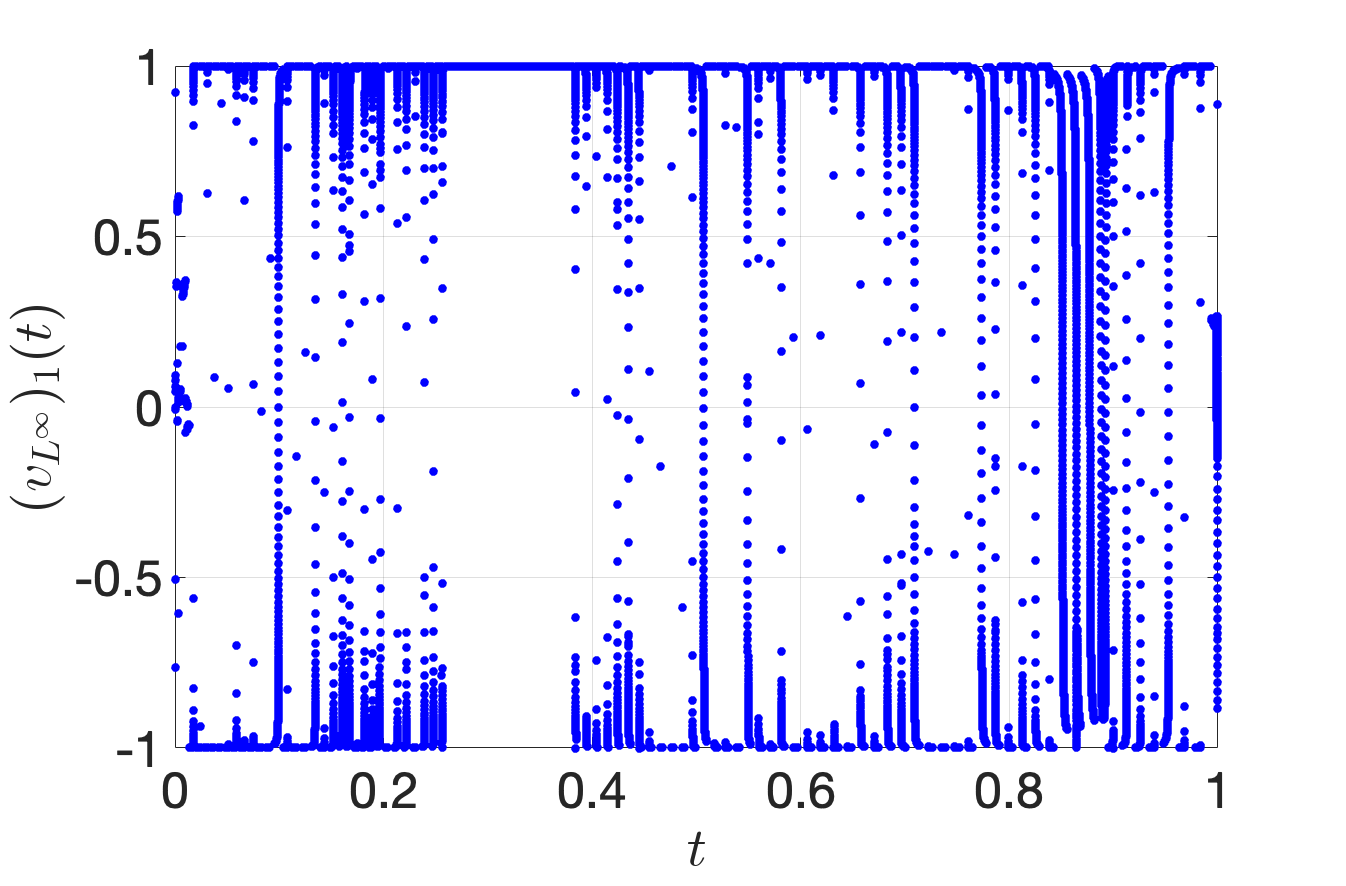} \\[3mm]
(a)
\end{center}
\end{minipage}
\begin{minipage}{80mm}
\begin{center}
\includegraphics[width=80mm]{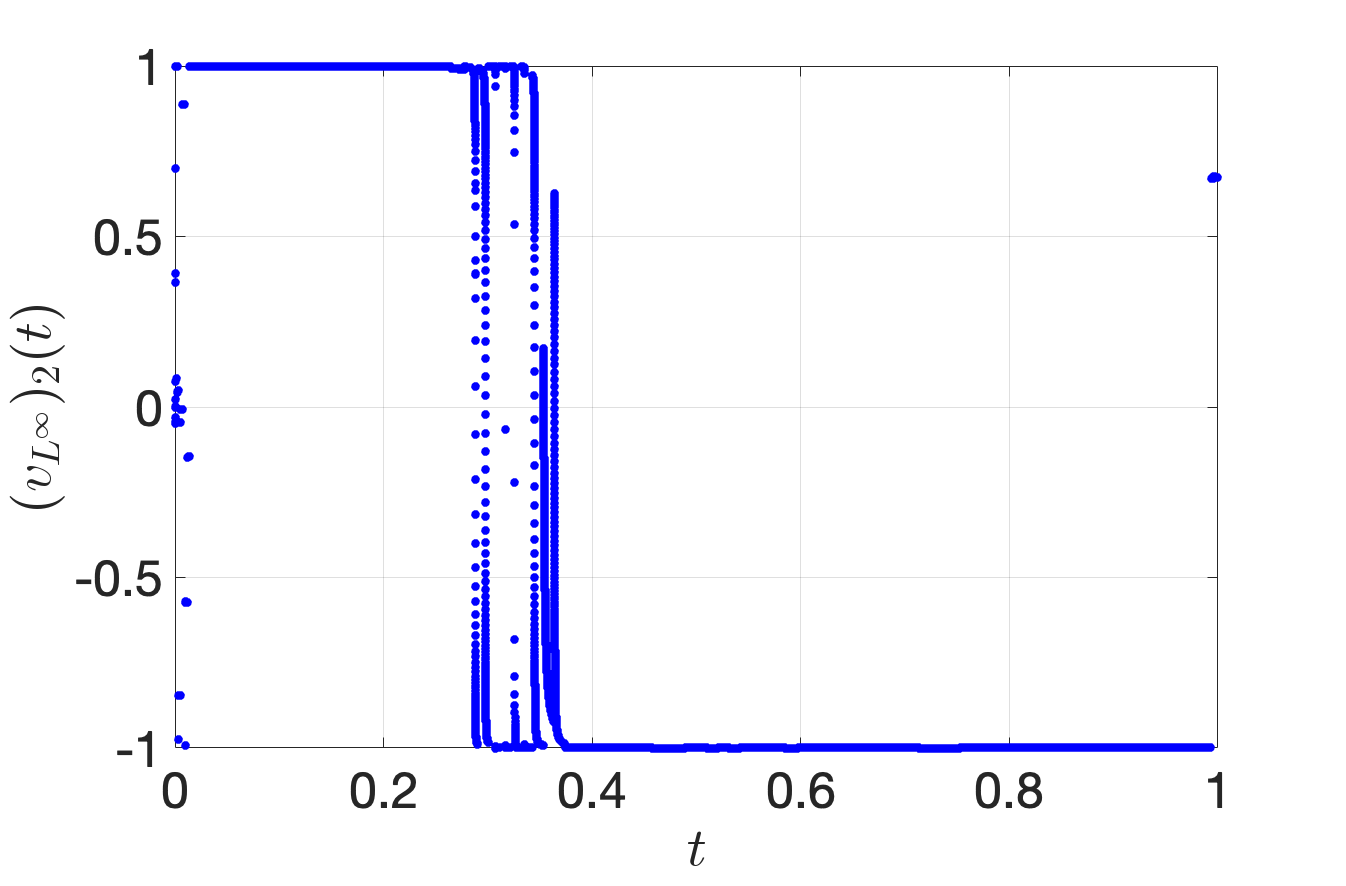} \\[3mm]
(b)
\end{center}
\end{minipage}
\\[10mm]
\begin{minipage}{80mm}
\begin{center}
\includegraphics[width=80mm]{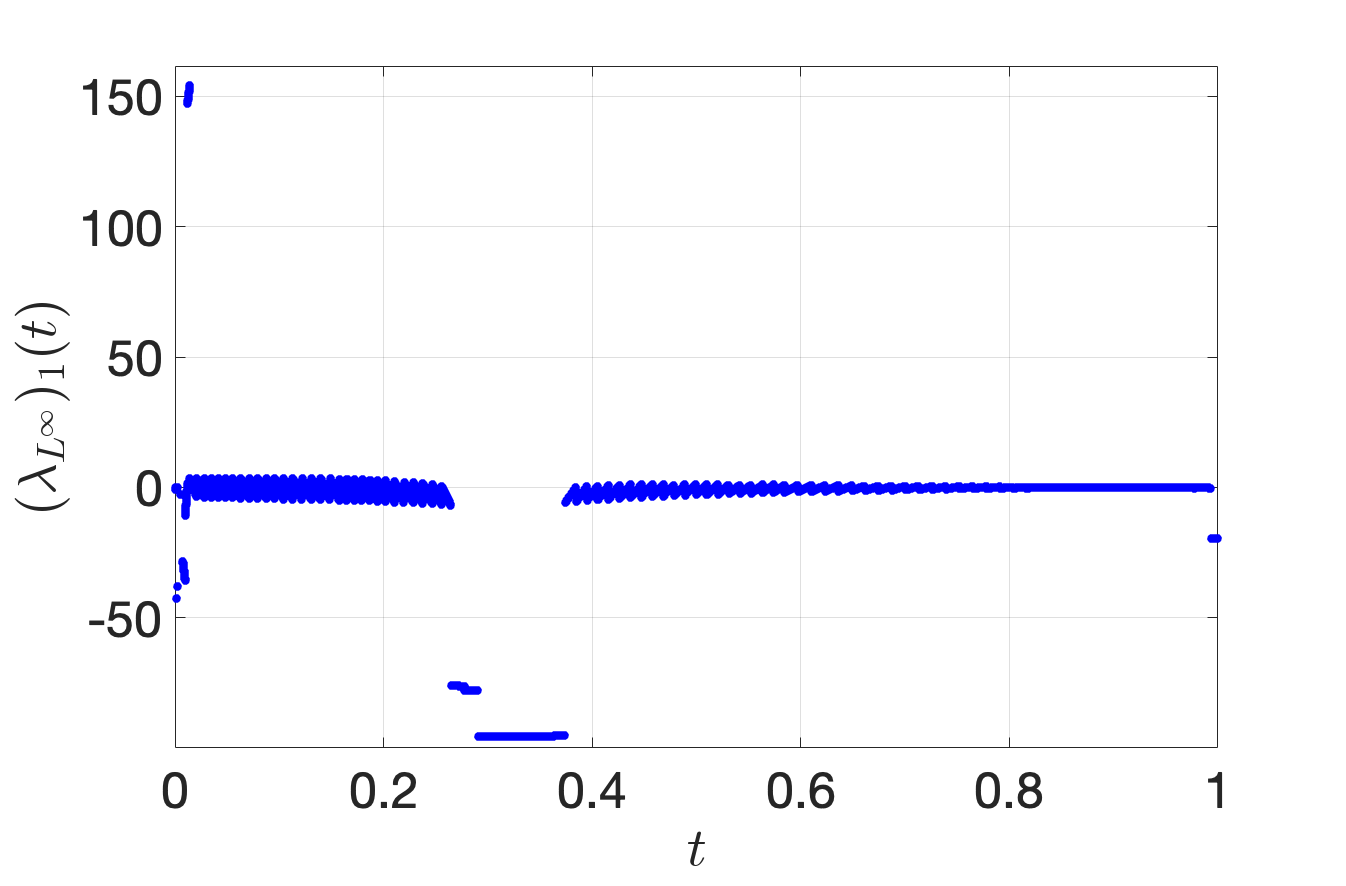} \\[3mm]
(c)
\end{center}
\end{minipage}
\begin{minipage}{80mm}
\begin{center}
\includegraphics[width=80mm]{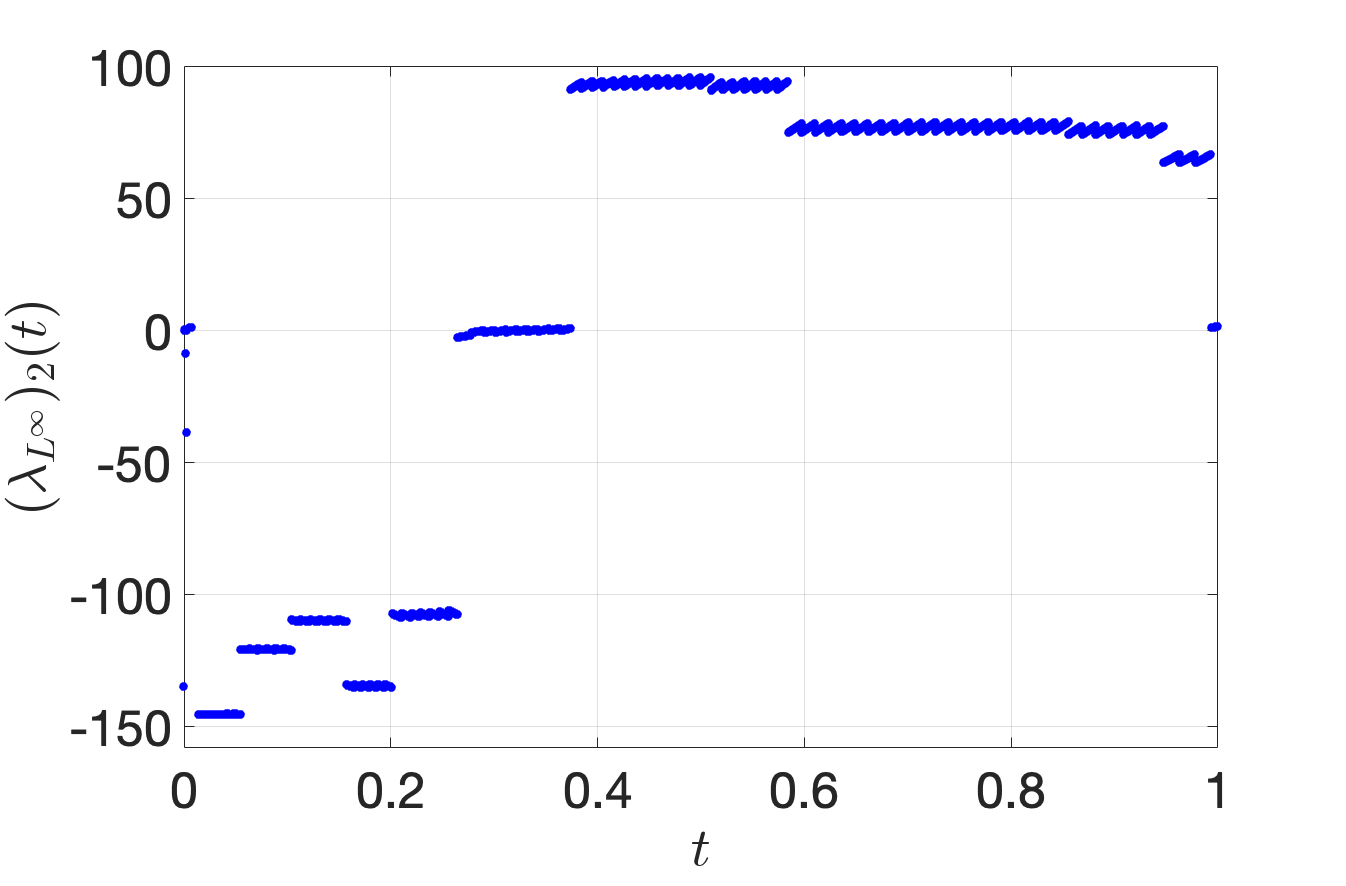} \\[3mm]
(d)
\end{center}
\end{minipage}
\
\caption{\sf Example 3 -- (a)--(b) control variable vector and (c)--(d) costate variable vector components, all for equation~\eqref{ode:vdp}, with $0\le t\le1$, $x(0)=(-1,-3)$, and {\sc Matlab}'s {\tt ode15s}.}
\label{fig:ex3b}
\end{figure}

\afterpage{\clearpage}
\begin{figure}
\begin{minipage}{80mm}
\begin{center}
\includegraphics[width=80mm]{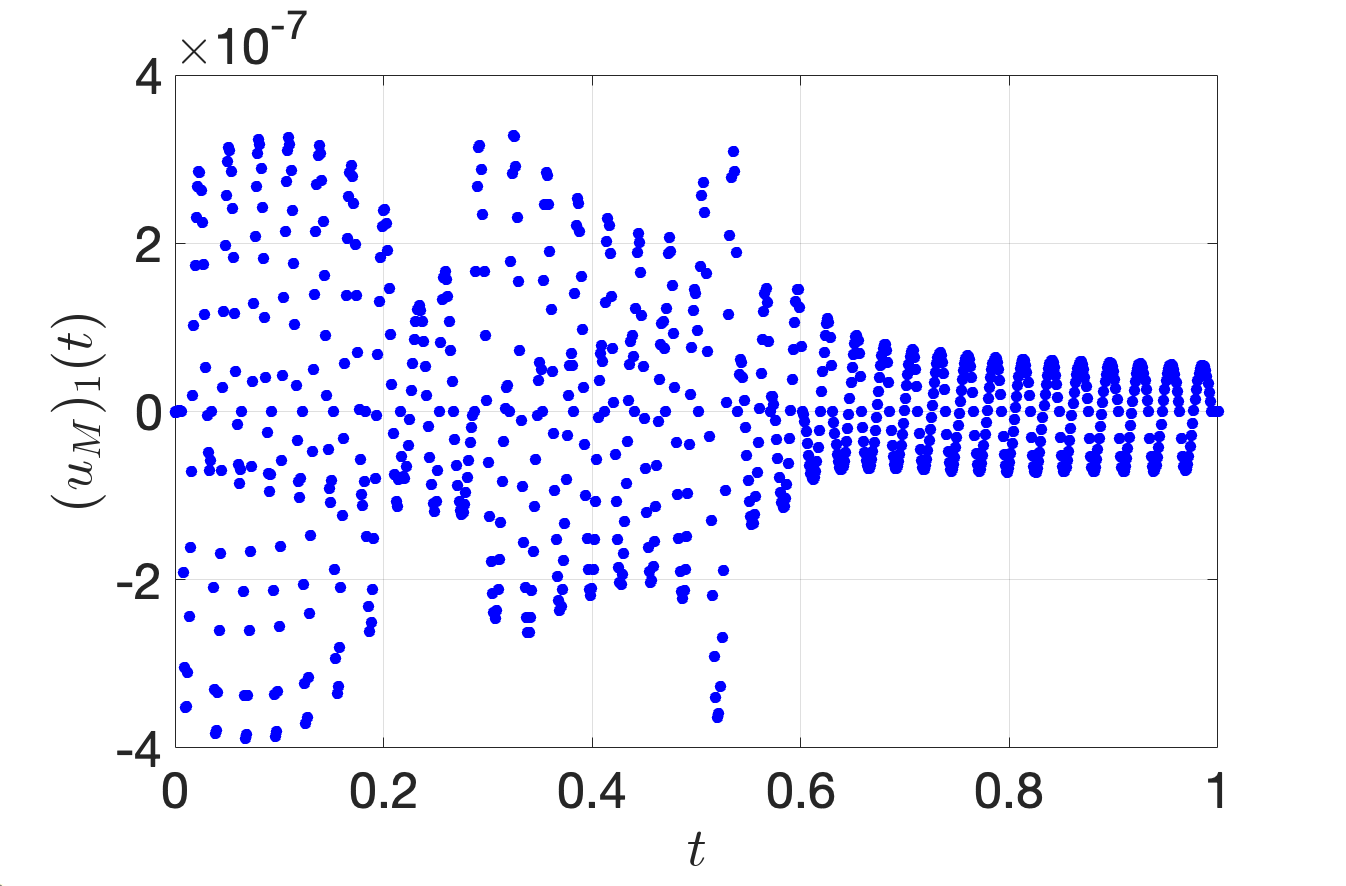} \\[3mm]
(a)
\end{center}
\end{minipage}
\begin{minipage}{80mm}
\begin{center}
\includegraphics[width=80mm]{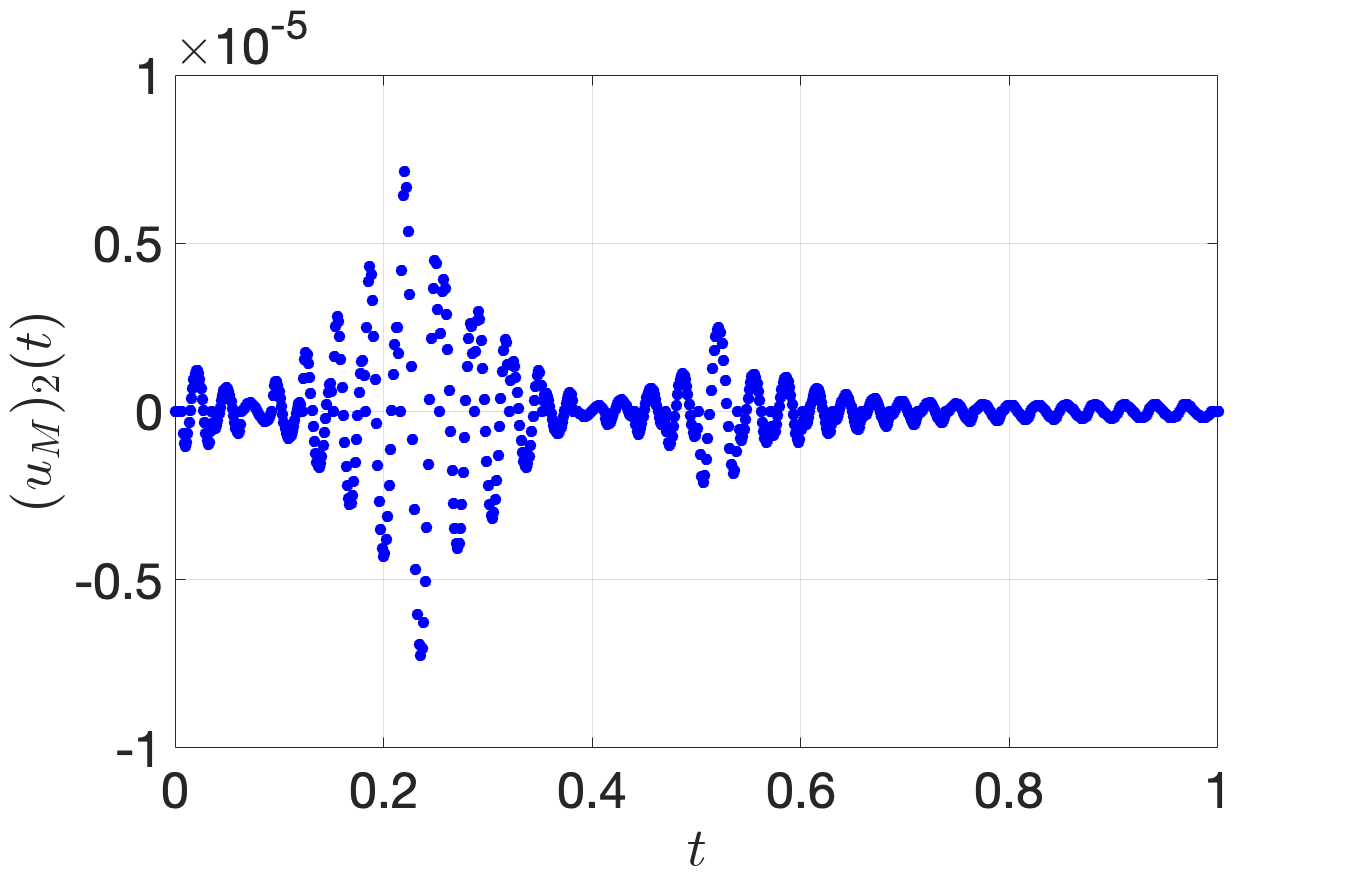} \\[3mm]
(b)
\end{center}
\end{minipage}
\\[10mm]
\begin{minipage}{80mm}
\begin{center}
\includegraphics[width=80mm]{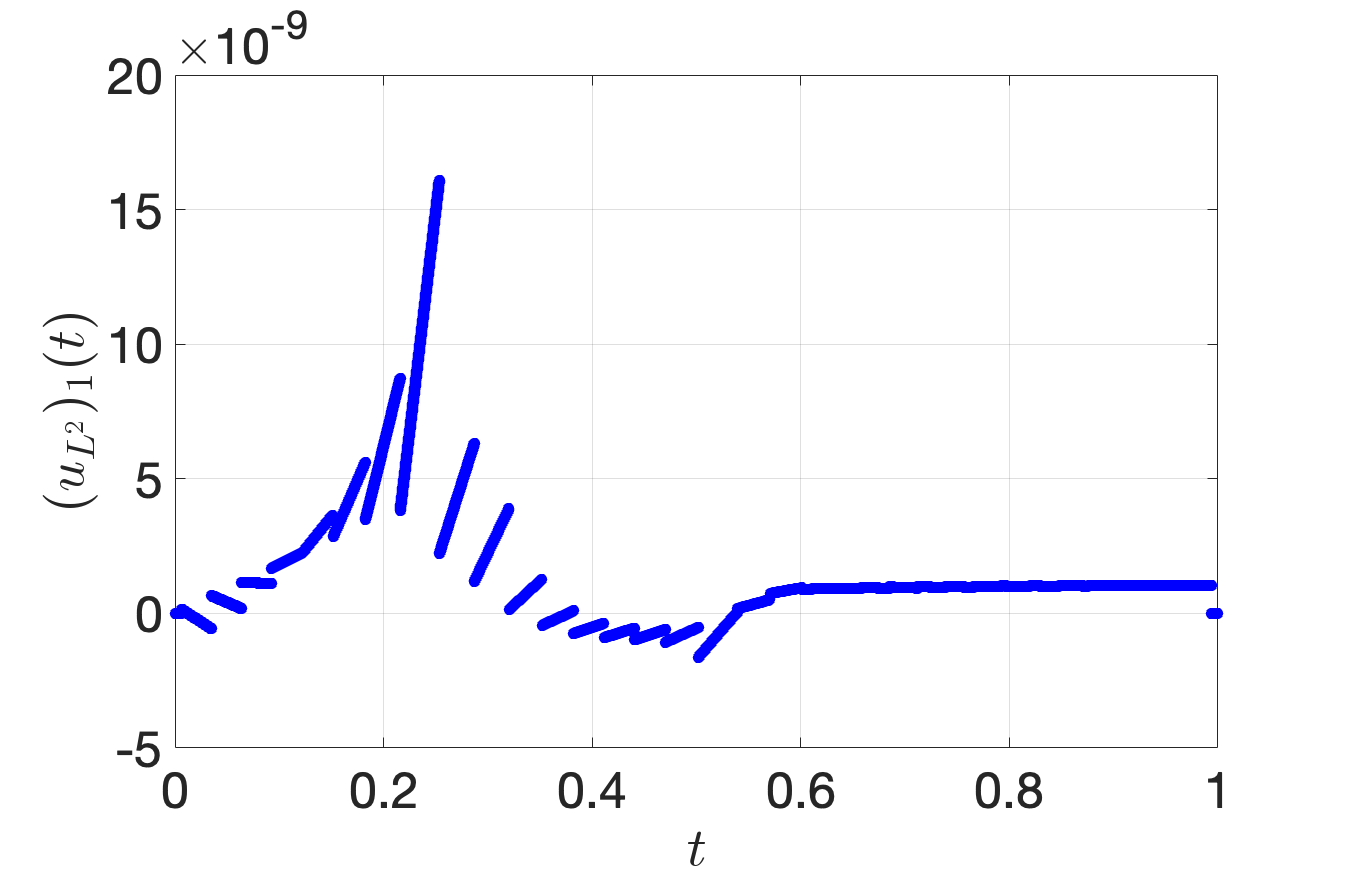} \\[3mm]
(c)
\end{center}
\end{minipage}
\begin{minipage}{80mm}
\begin{center}
\includegraphics[width=80mm]{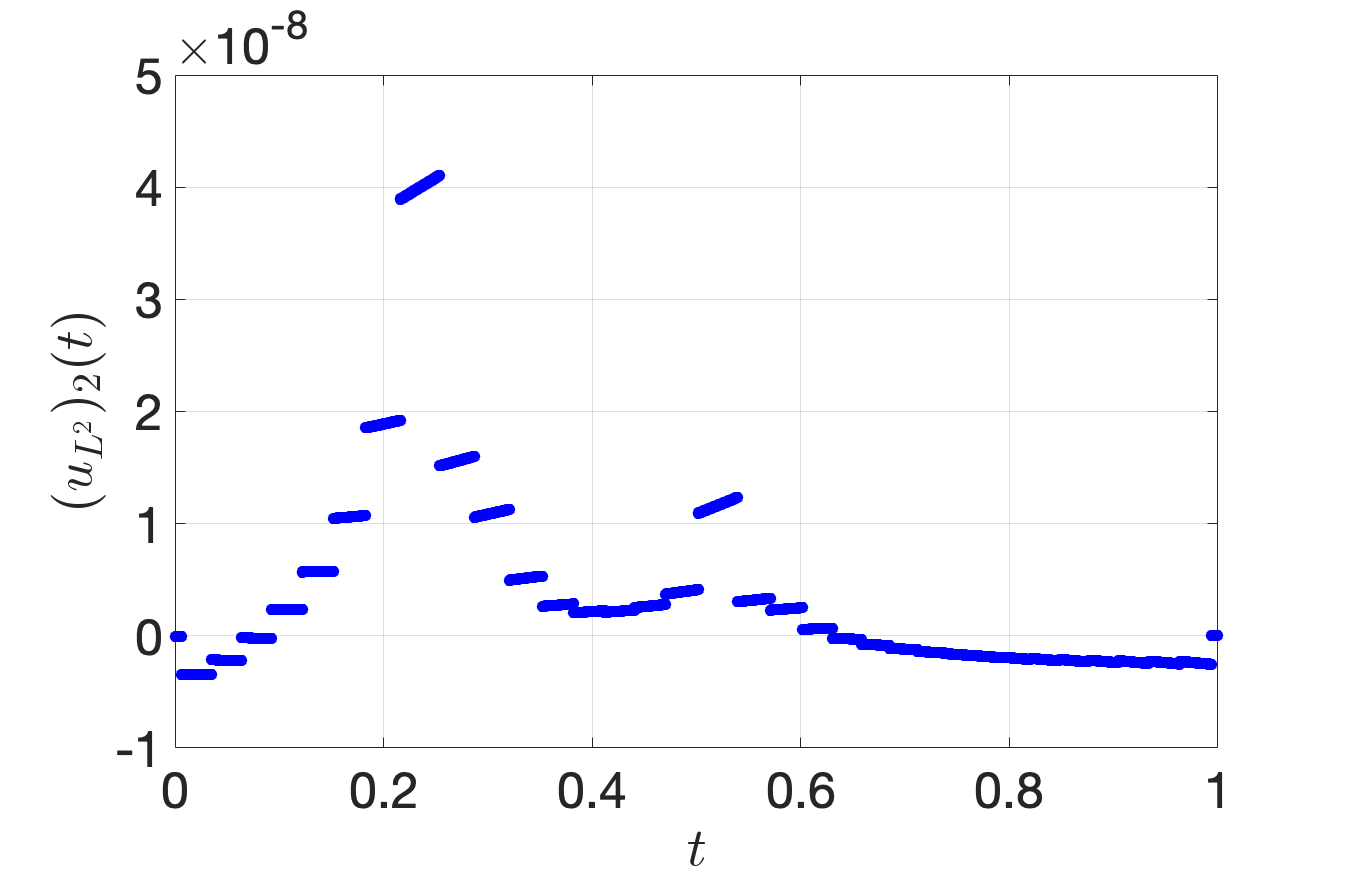} \\[3mm]
(d)
\end{center}
\end{minipage}
\\[10mm]
\begin{minipage}{80mm}
\begin{center}
\includegraphics[width=80mm]{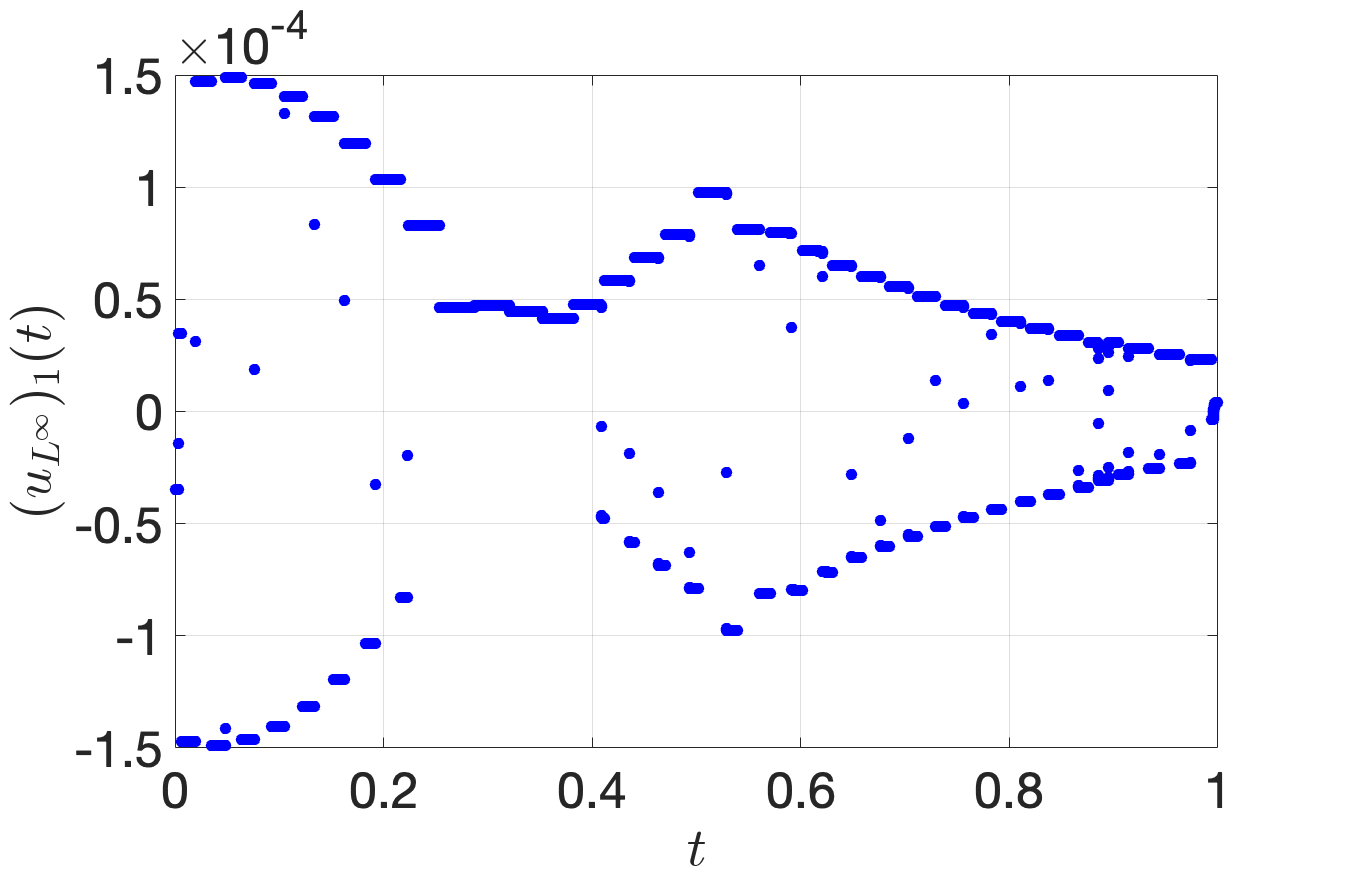} \\[3mm]
(e)
\end{center}
\end{minipage}
\begin{minipage}{80mm}
\begin{center}
\includegraphics[width=80mm]{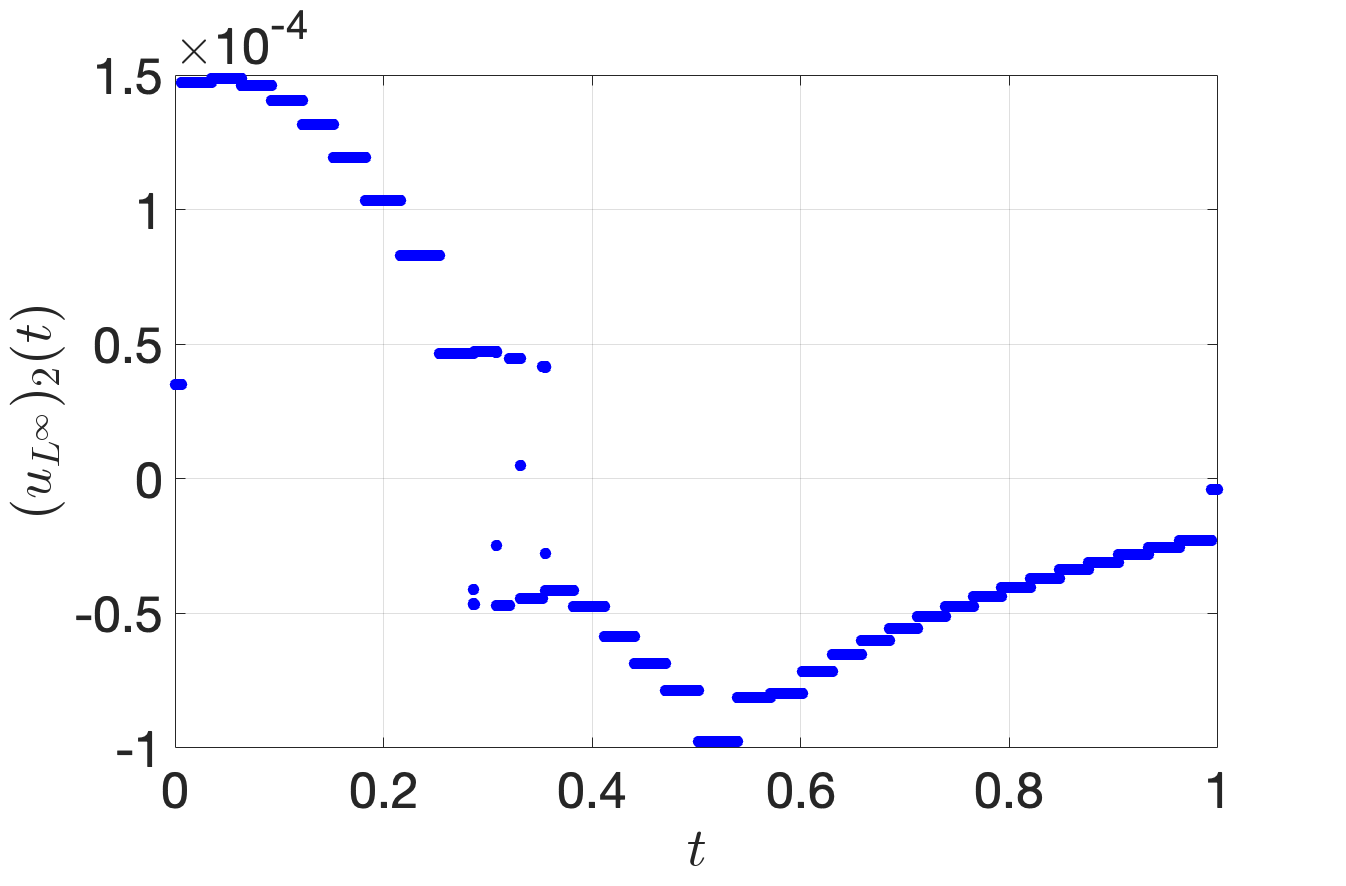} \\[3mm]
(f)
\end{center}
\end{minipage}
\
\caption{\sf Example 3 -- interpolant residual vector components
computed for equation~\eqref{ode:vdp} and {\sc Matlab}'s {\tt ode45}, with $0\le t\le1$ and $x(0)=(-1,-3)$, using (a)--(b)~{\sc Matlab}'s {\tt deval}, (c)--(d) $L^2$-minimization and (e)--(f)~stage $L^\infty$-minimization.}
\label{fig:ex3c}
\end{figure}

\afterpage{\clearpage}
\begin{figure}[t!]
\begin{minipage}{80mm}
\begin{center}
\includegraphics[width=80mm]{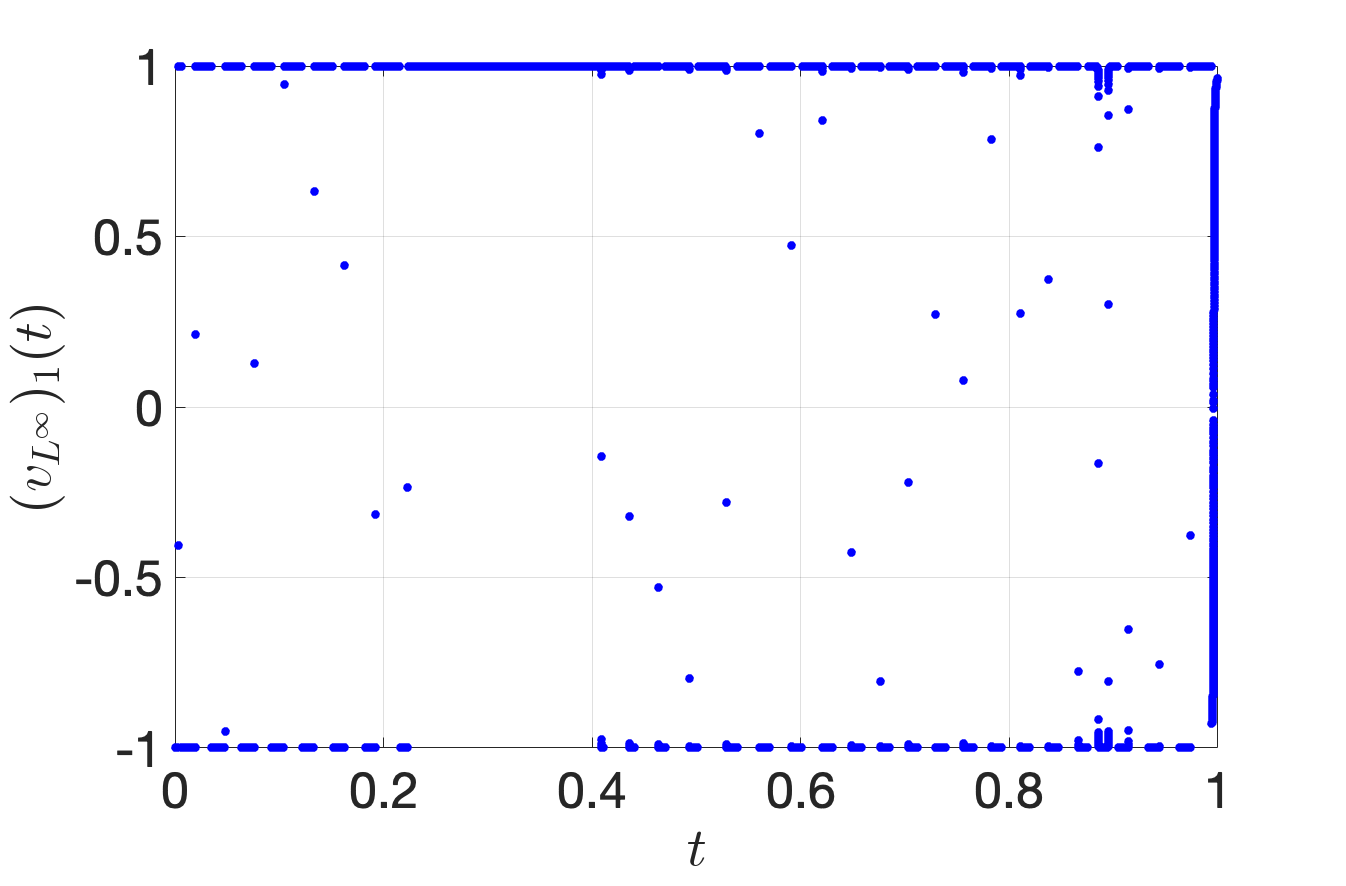} \\[3mm]
(a)
\end{center}
\end{minipage}
\begin{minipage}{80mm}
\begin{center}
\includegraphics[width=80mm]{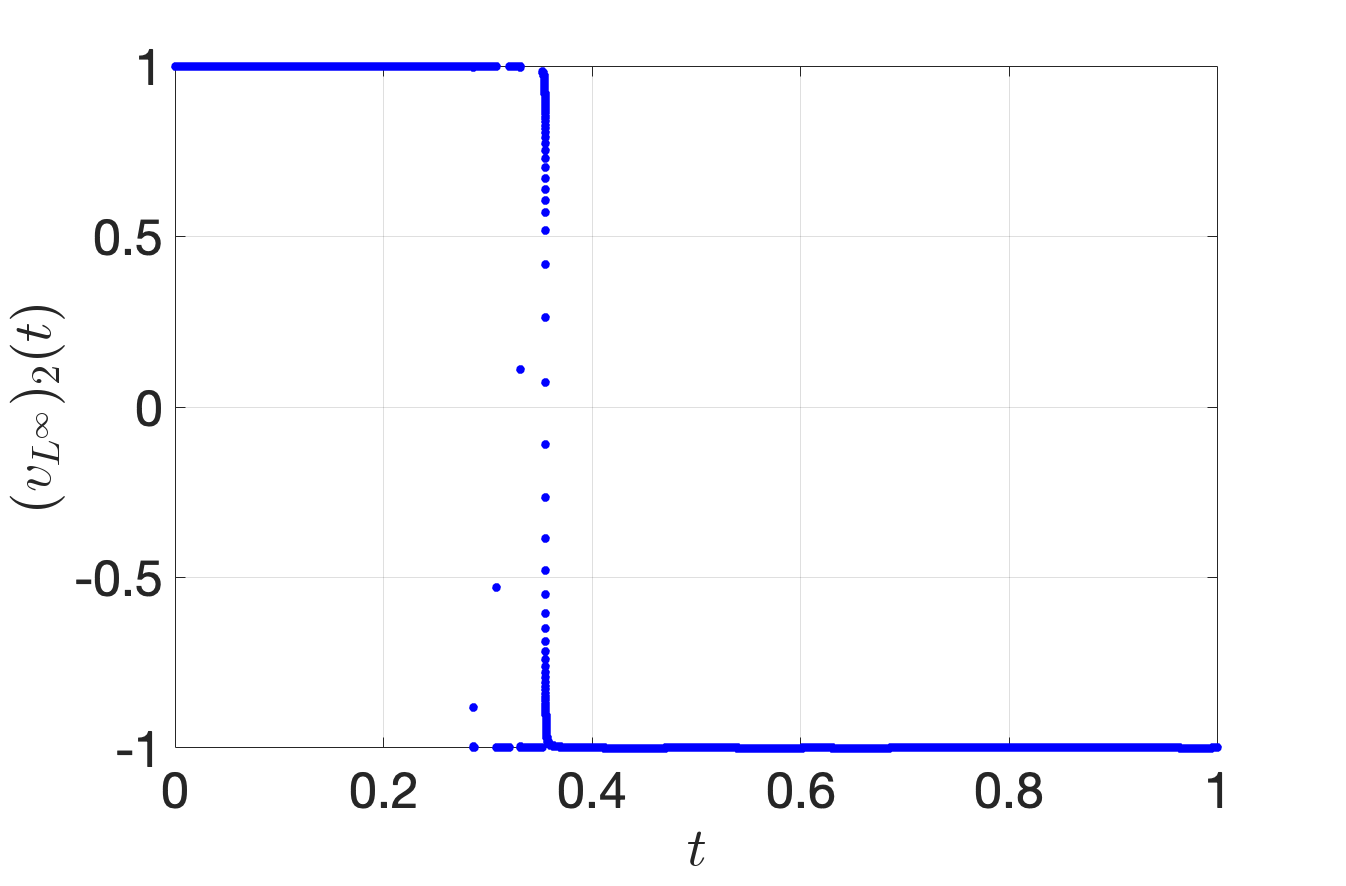} \\[3mm]
(b)
\end{center}
\end{minipage}
\\[10mm]
\begin{minipage}{80mm}
\begin{center}
\includegraphics[width=80mm]{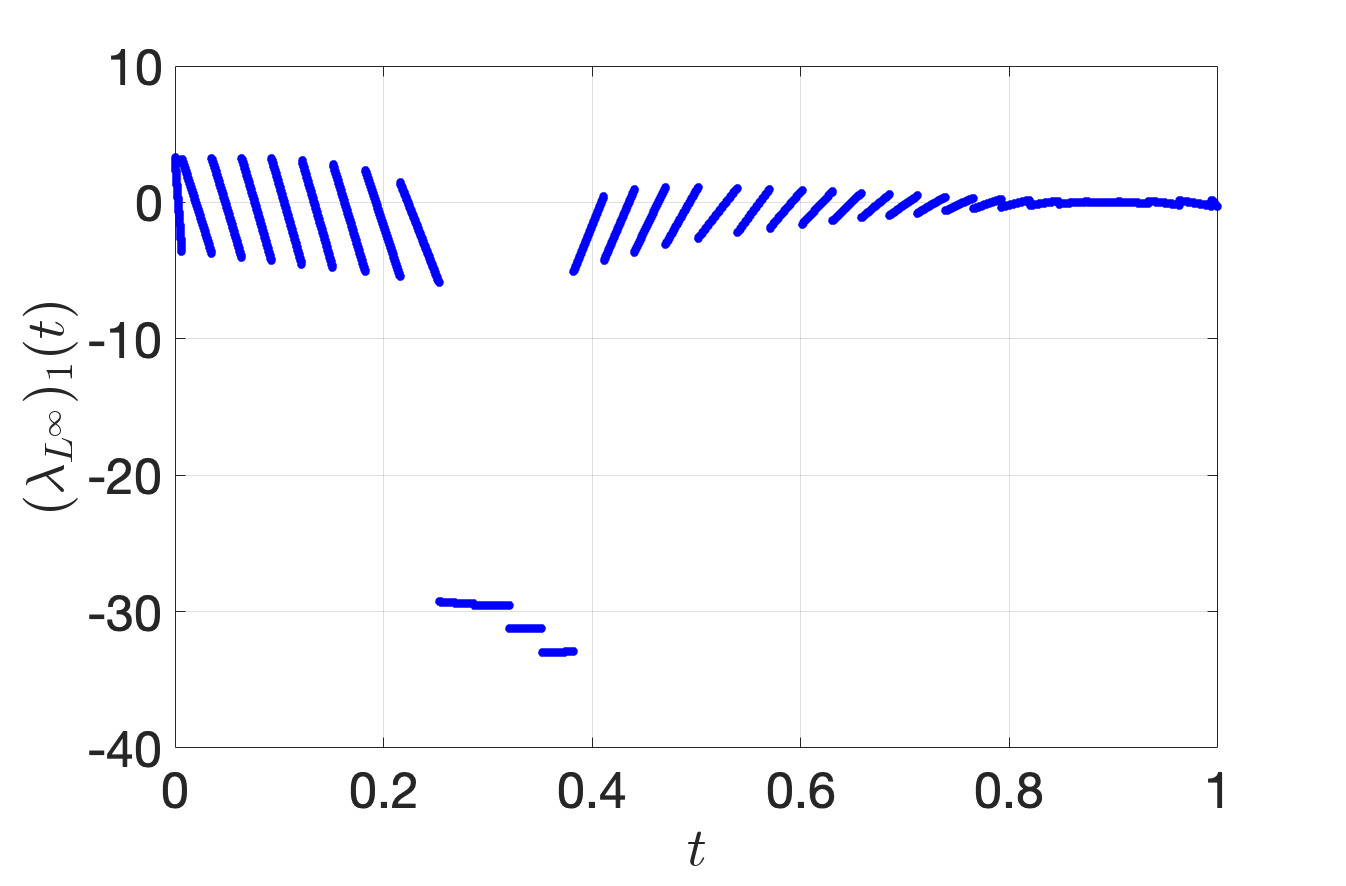} \\[3mm]
(c)
\end{center}
\end{minipage}
\begin{minipage}{80mm}
\begin{center}
\includegraphics[width=80mm]{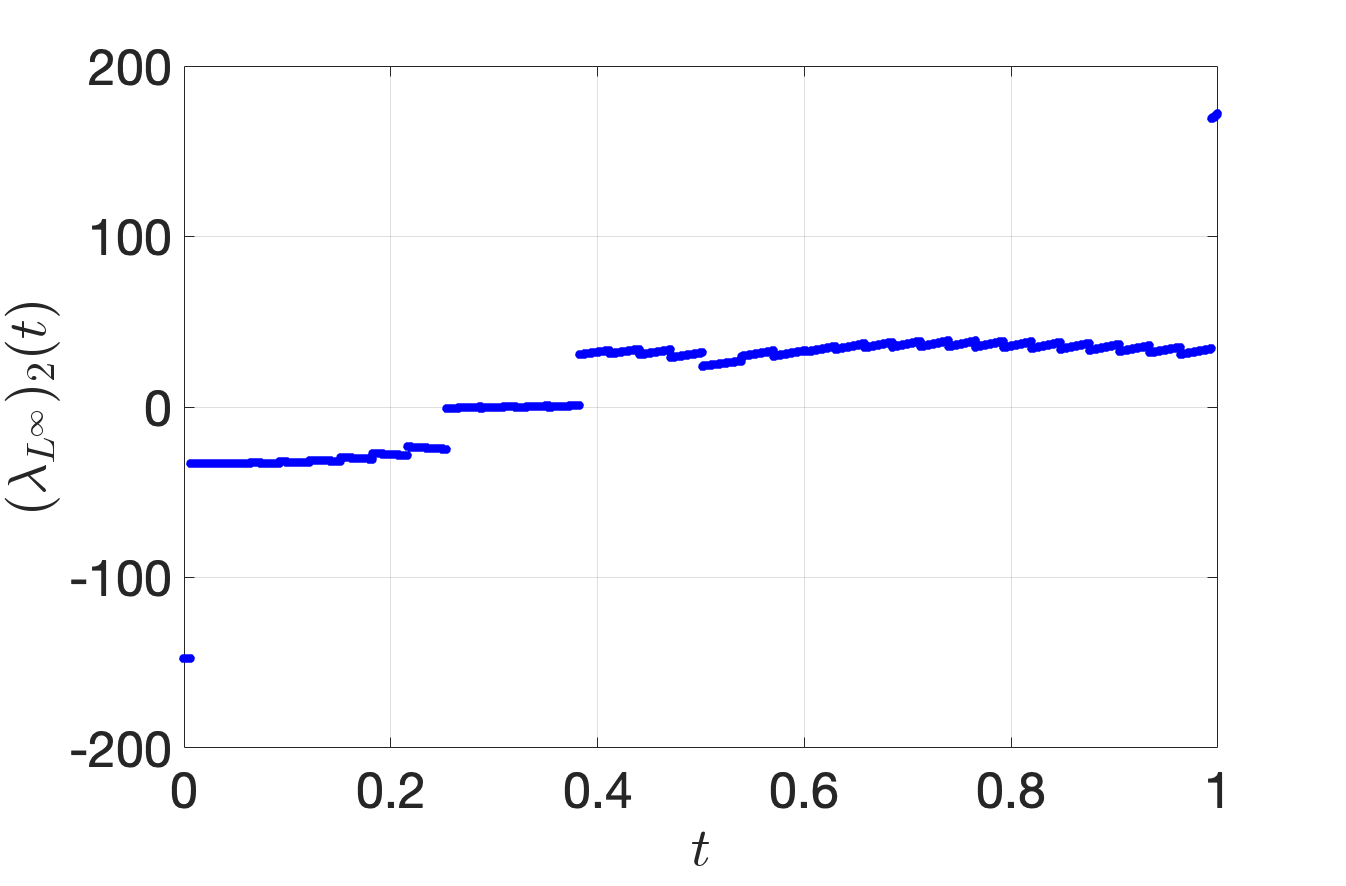} \\[3mm]
(d)
\end{center}
\end{minipage}
\
\caption{\sf Example 3 -- (a)--(b) control variable vector and (c)--(d) costate variable vector components, all for equation~\eqref{ode:vdp}, with $0\le t\le1$, $x(0)=(-1,-3)$, and {\sc Matlab}'s {\tt ode45}.}
\label{fig:ex3d}
\end{figure}

\afterpage{\clearpage}
\begin{figure}
\begin{minipage}{80mm}
\begin{center}
\includegraphics[width=80mm]{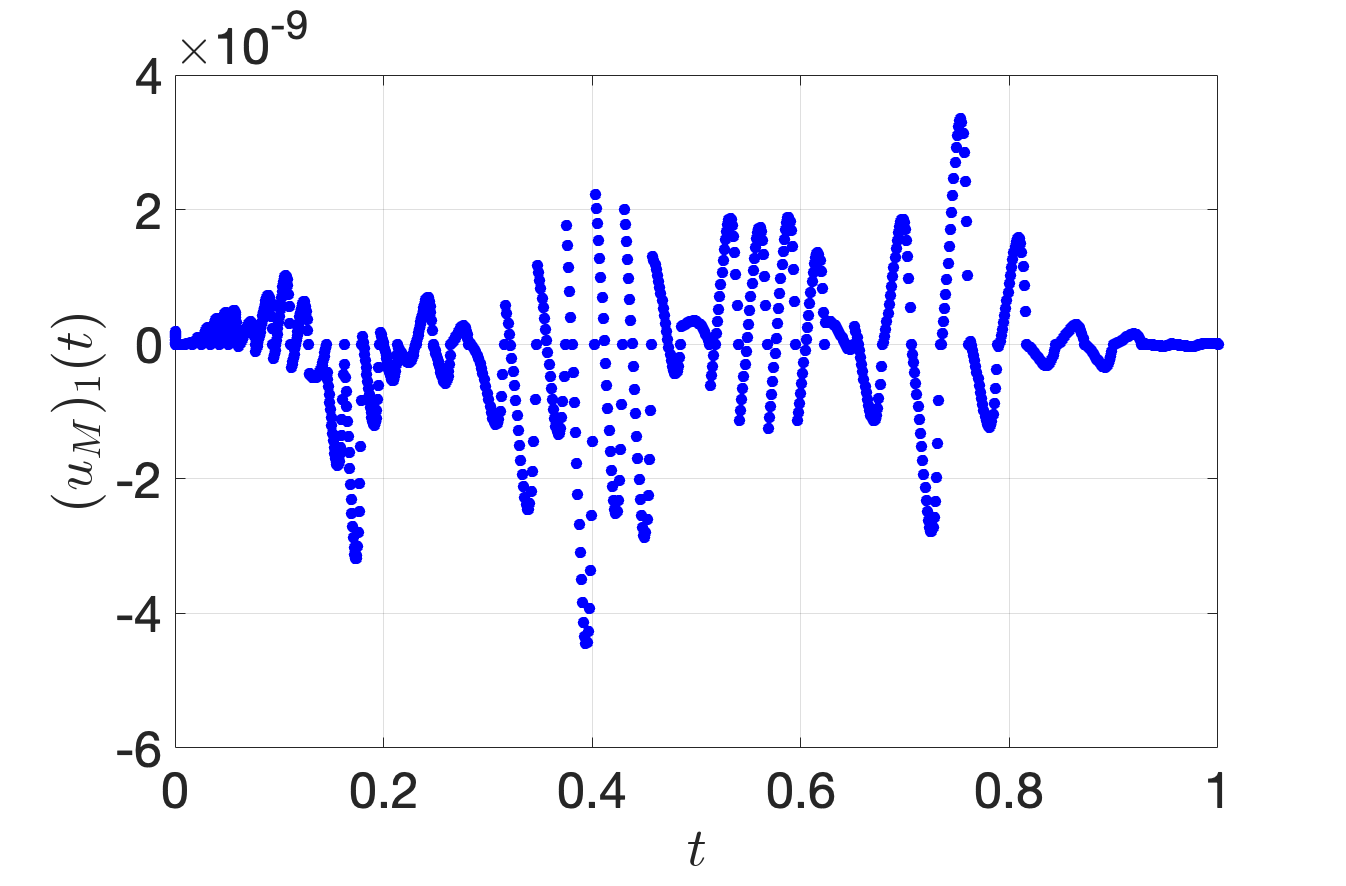} \\[3mm]
(a)
\end{center}
\end{minipage}
\begin{minipage}{80mm}
\begin{center}
\includegraphics[width=80mm]{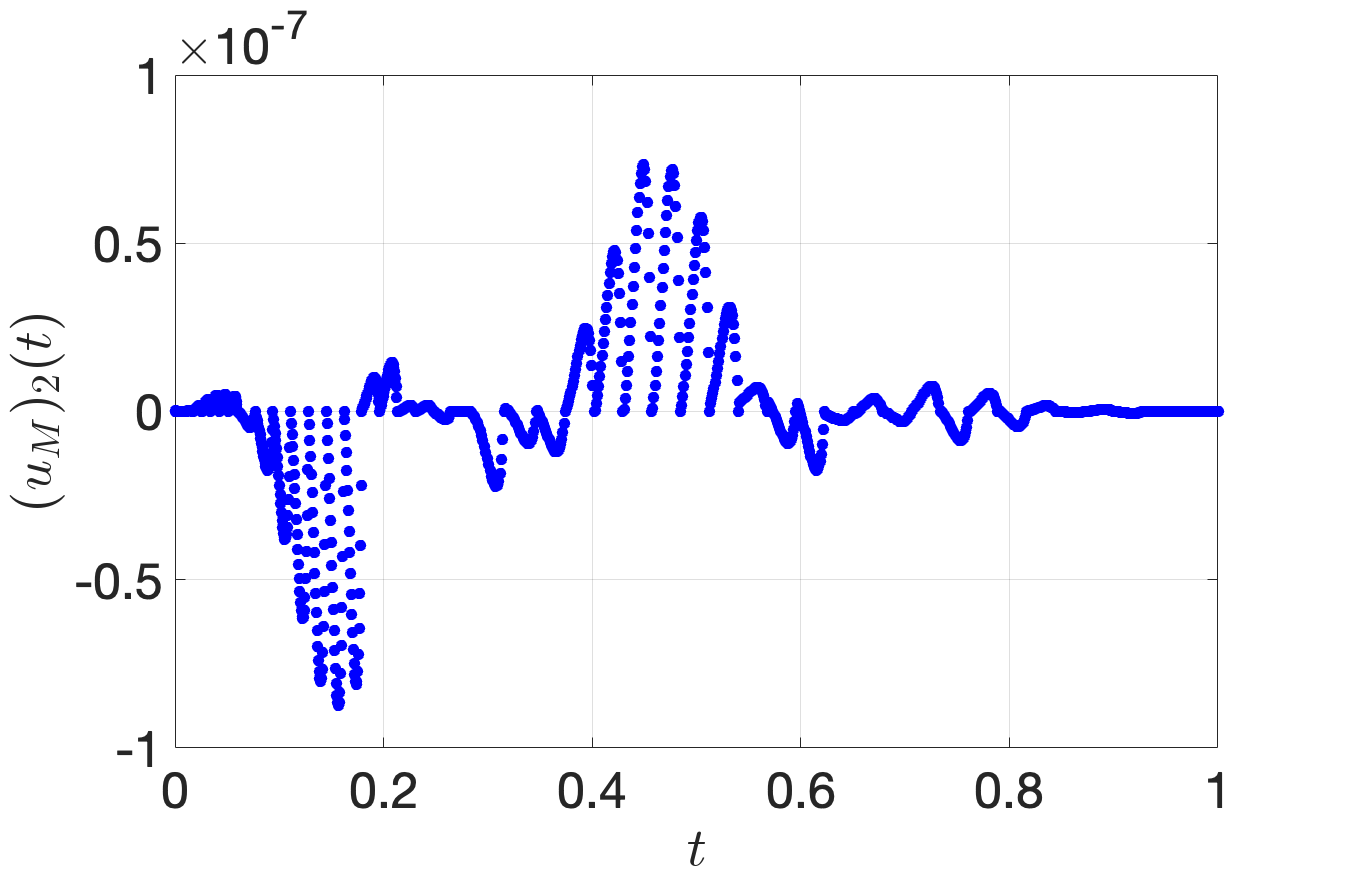} \\[3mm]
(b)
\end{center}
\end{minipage}
\\[10mm]
\begin{minipage}{80mm}
\begin{center}
\includegraphics[width=80mm]{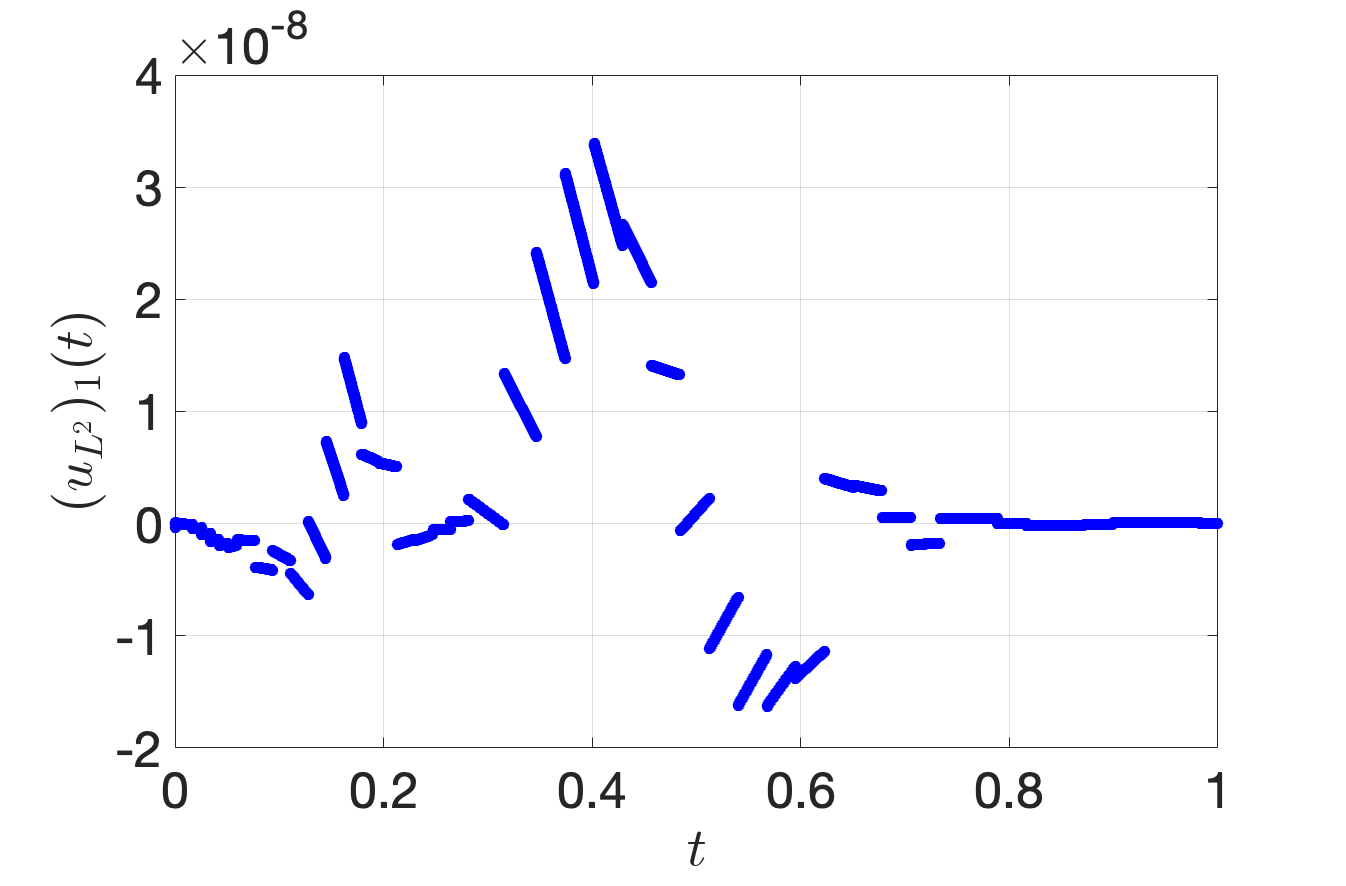} \\[3mm]
(c)
\end{center}
\end{minipage}
\begin{minipage}{80mm}
\begin{center}
\includegraphics[width=80mm]{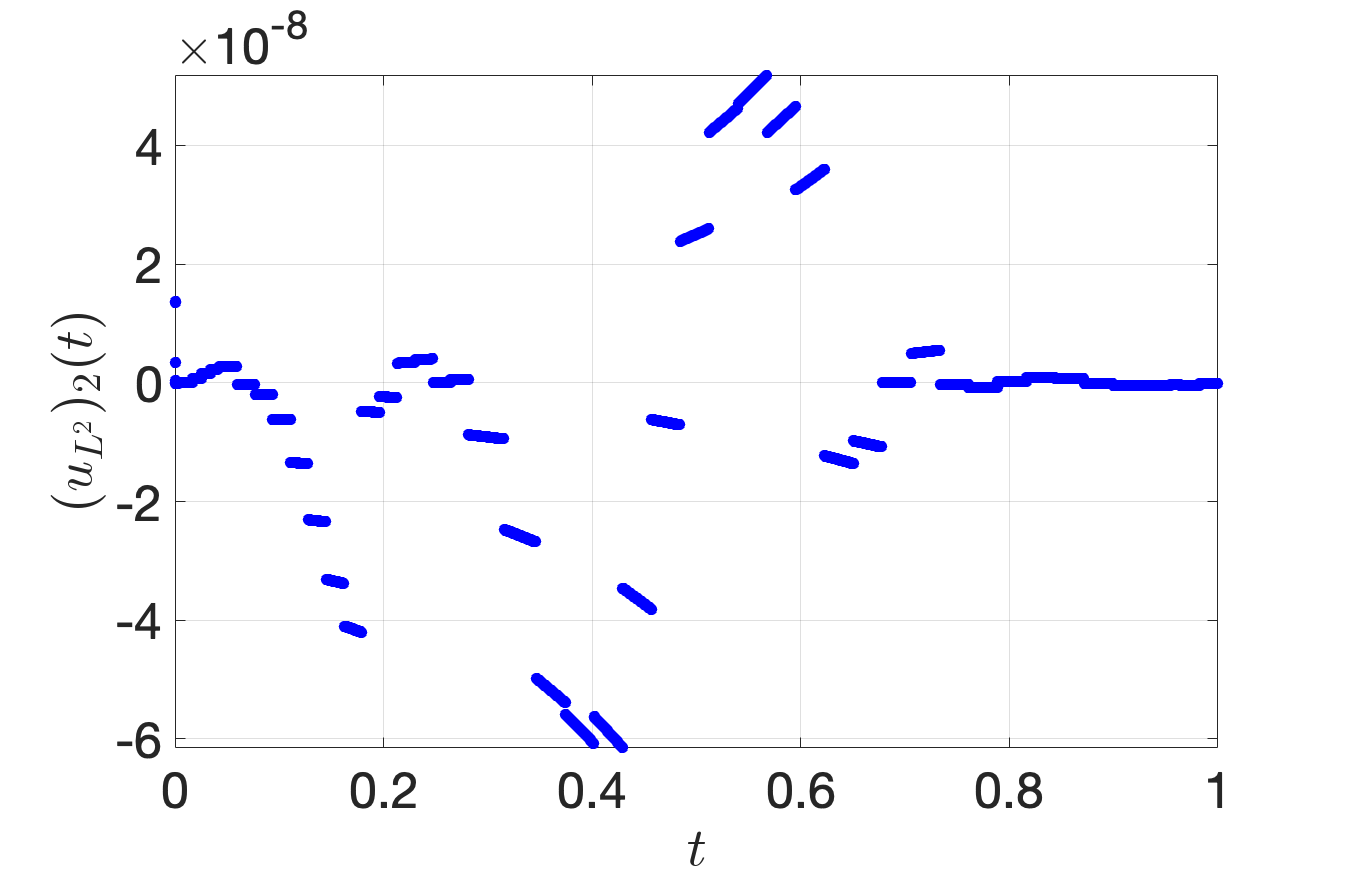} \\[3mm]
(d)
\end{center}
\end{minipage}
\\[10mm]
\begin{minipage}{80mm}
\begin{center}
\includegraphics[width=80mm]{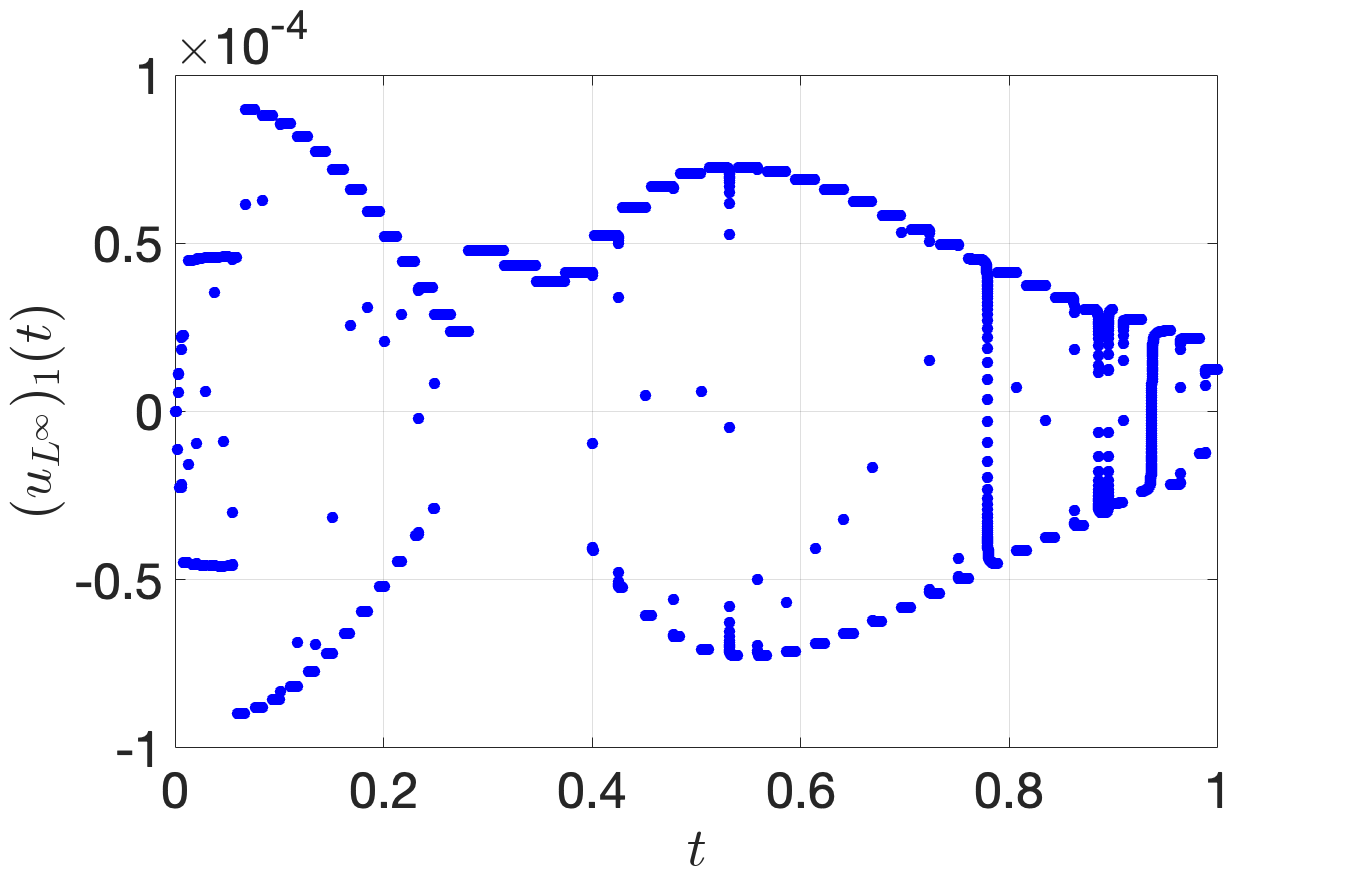} \\[3mm]
(e)
\end{center}
\end{minipage}
\begin{minipage}{80mm}
\begin{center}
\includegraphics[width=80mm]{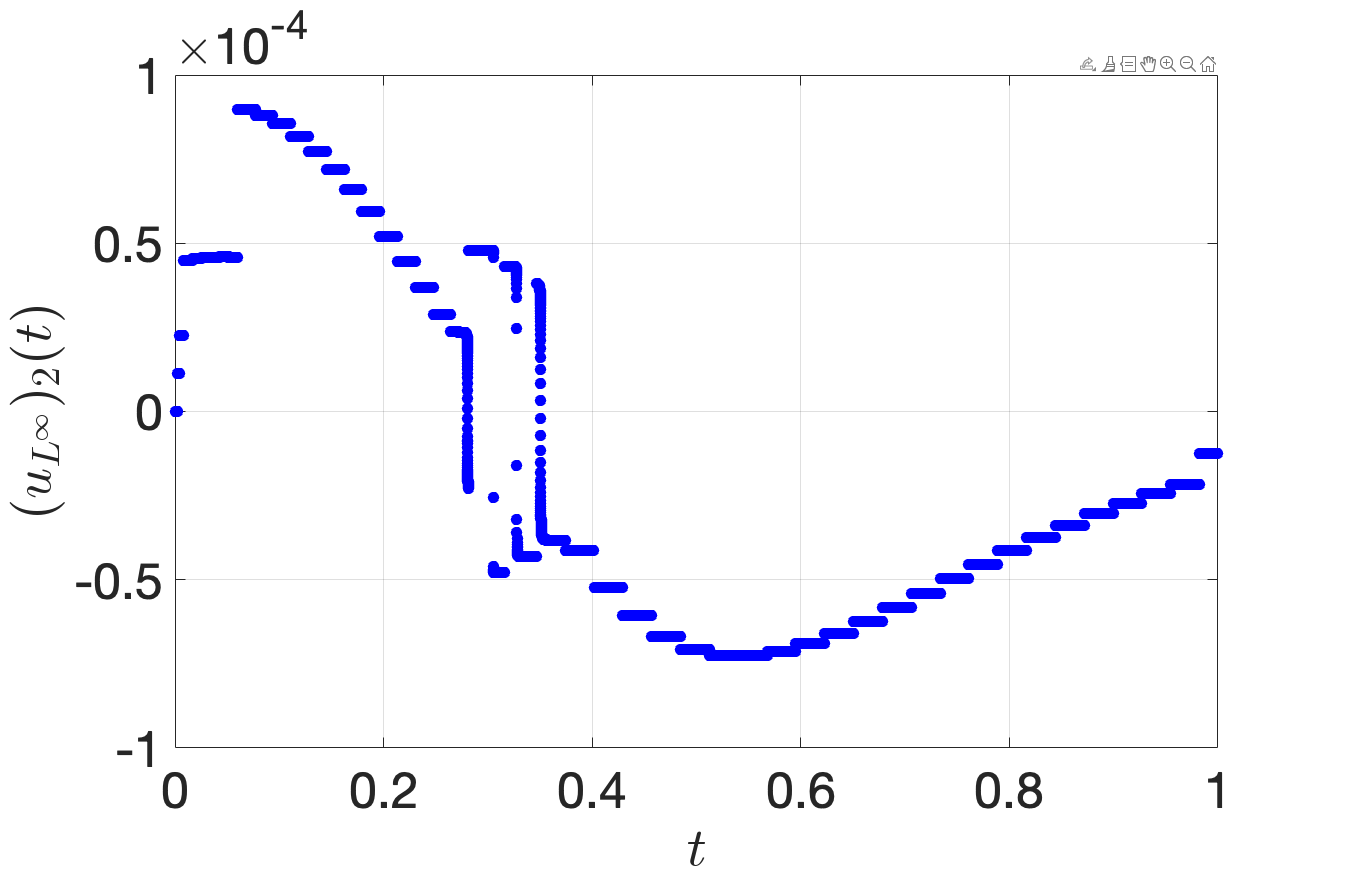} \\[3mm]
(f)
\end{center}
\end{minipage}
\
\caption{\sf Example 3 -- interpolant residual vector components
computed for equation~\eqref{ode:vdp} and {\sc Matlab}'s {\tt ode113}, with $0\le t\le1$ and $x(0)=(-1,-3)$, using (a)--(b)~{\sc Matlab}'s {\tt deval}, (c)--(d) $L^2$-minimization and (e)--(f)~stage $L^\infty$-minimization.}
\label{fig:ex3e}
\end{figure}

\afterpage{\clearpage}
\begin{figure}[t!]
\begin{minipage}{80mm}
\begin{center}
\includegraphics[width=80mm]{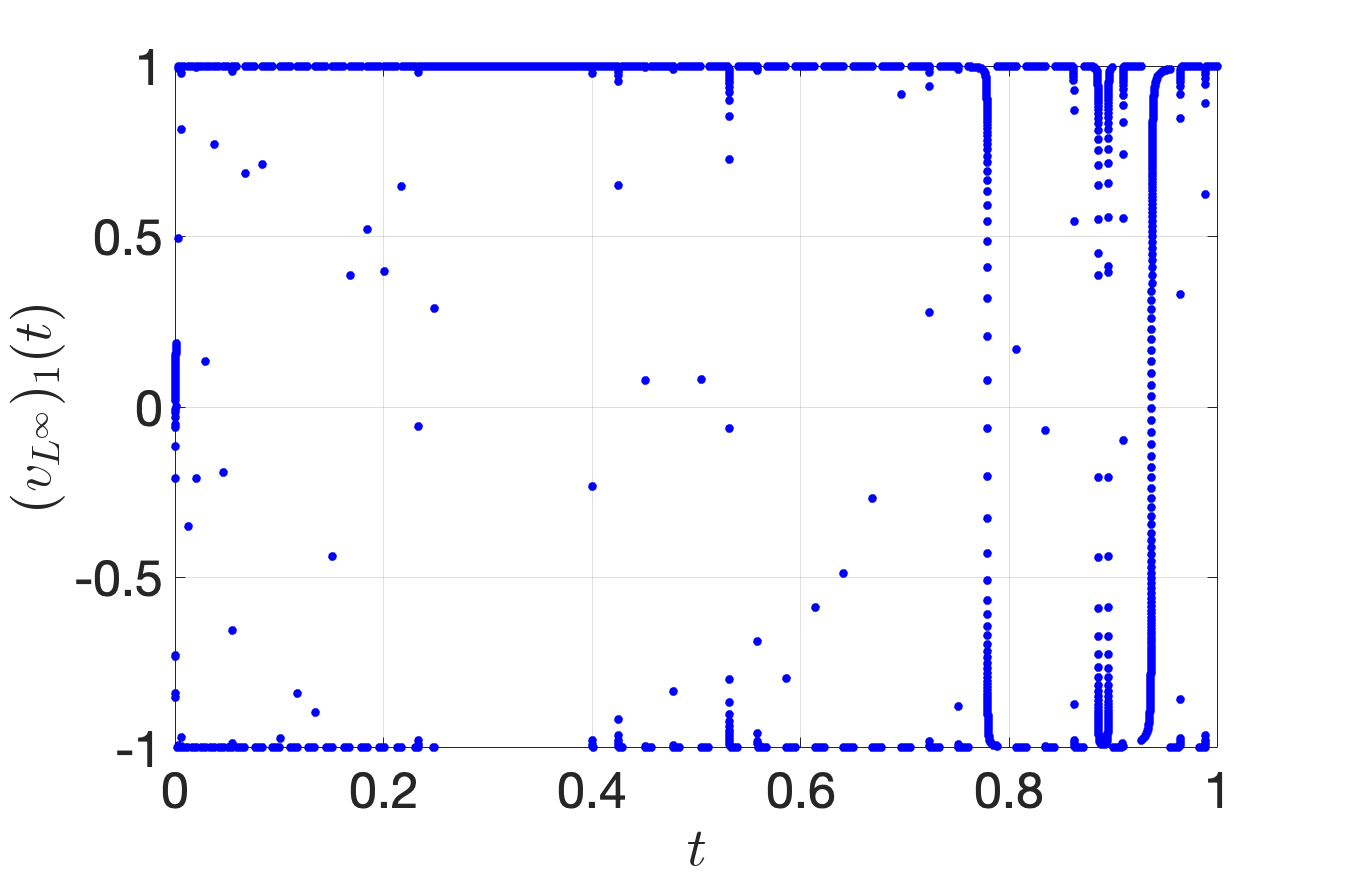} \\[3mm]
(a)
\end{center}
\end{minipage}
\begin{minipage}{80mm}
\begin{center}
\includegraphics[width=80mm]{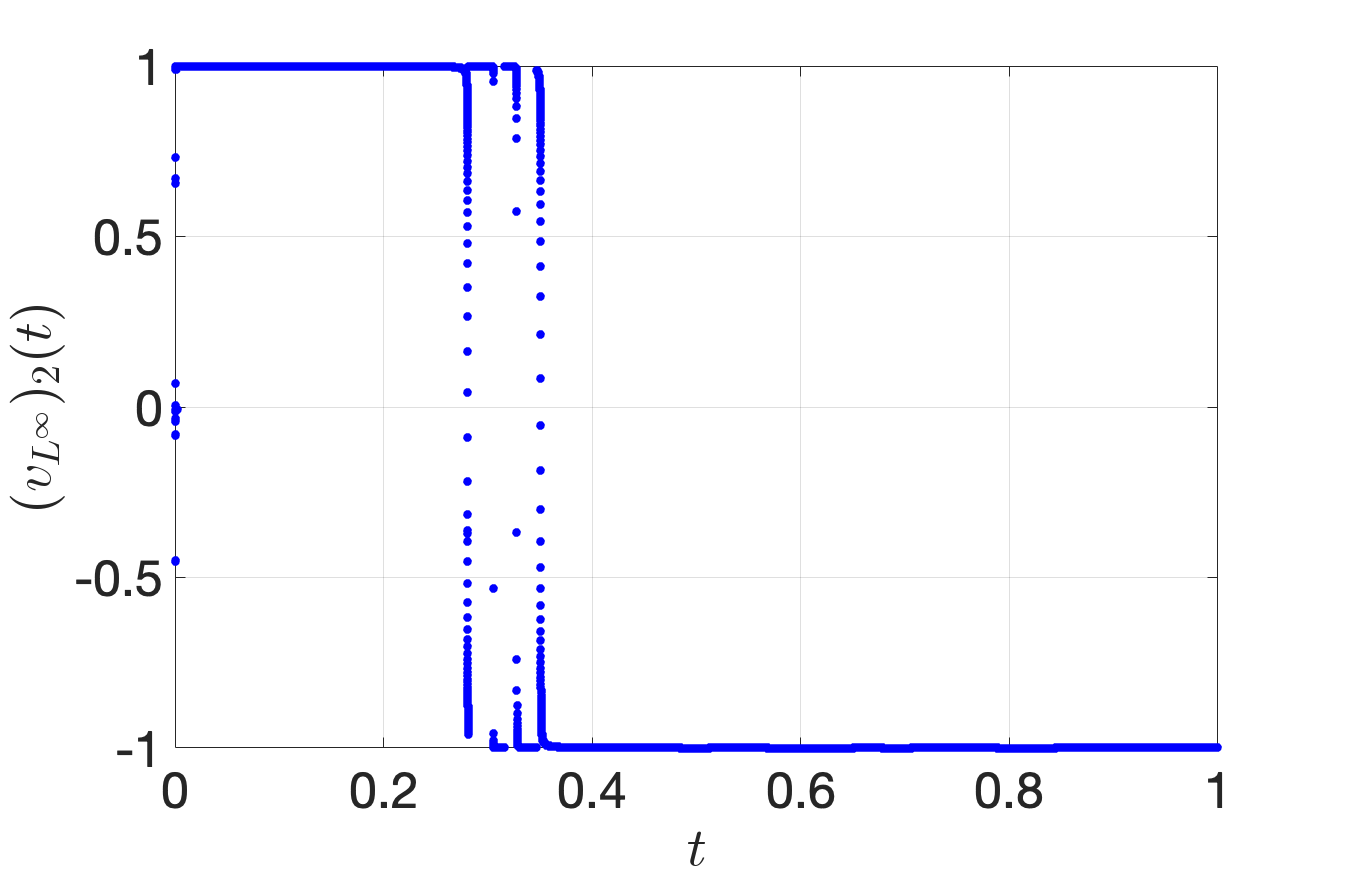} \\[3mm]
(b)
\end{center}
\end{minipage}
\\[10mm]
\begin{minipage}{80mm}
\begin{center}
\includegraphics[width=80mm]{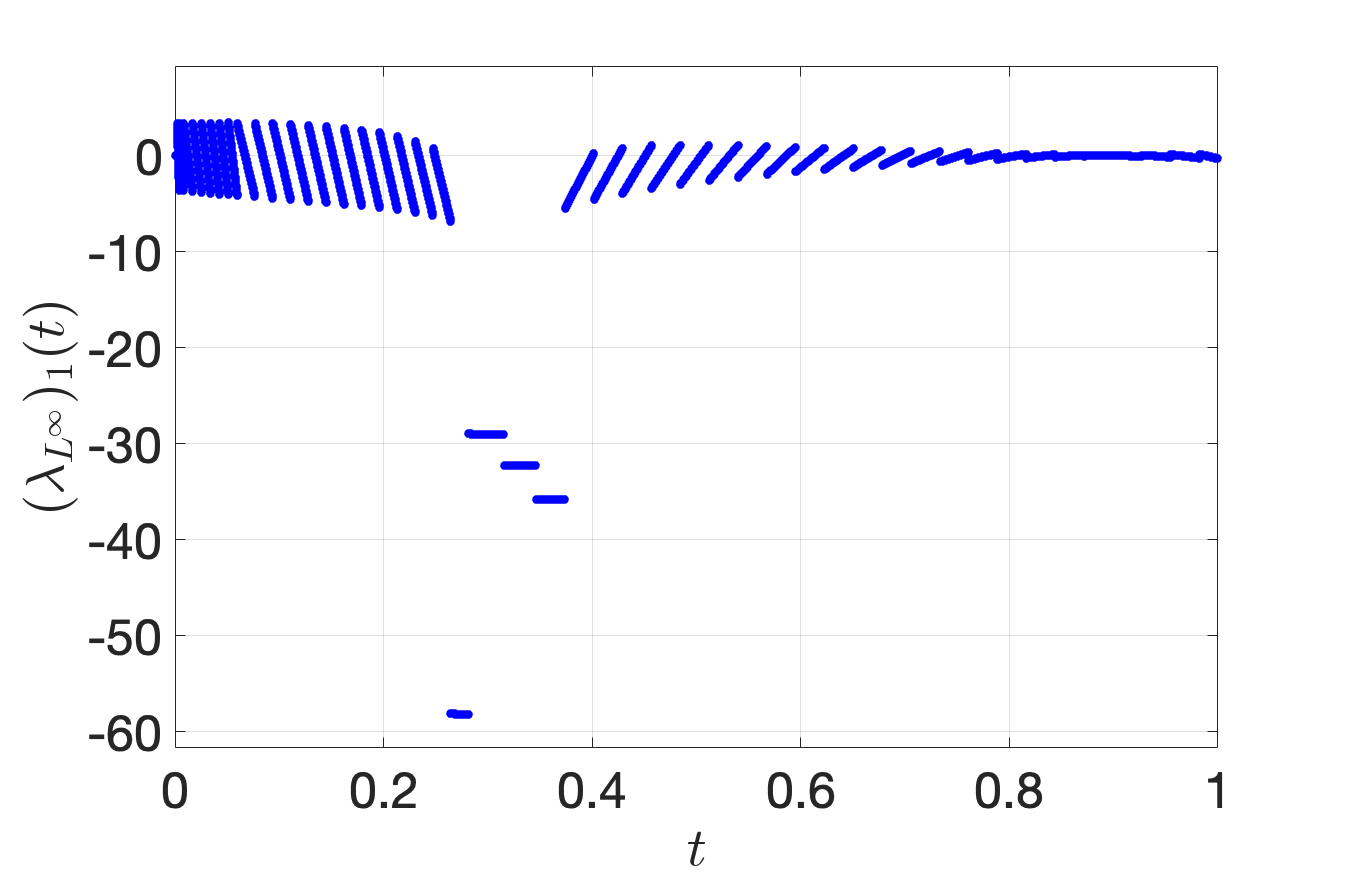} \\[3mm]
(c)
\end{center}
\end{minipage}
\begin{minipage}{80mm}
\begin{center}
\includegraphics[width=80mm]{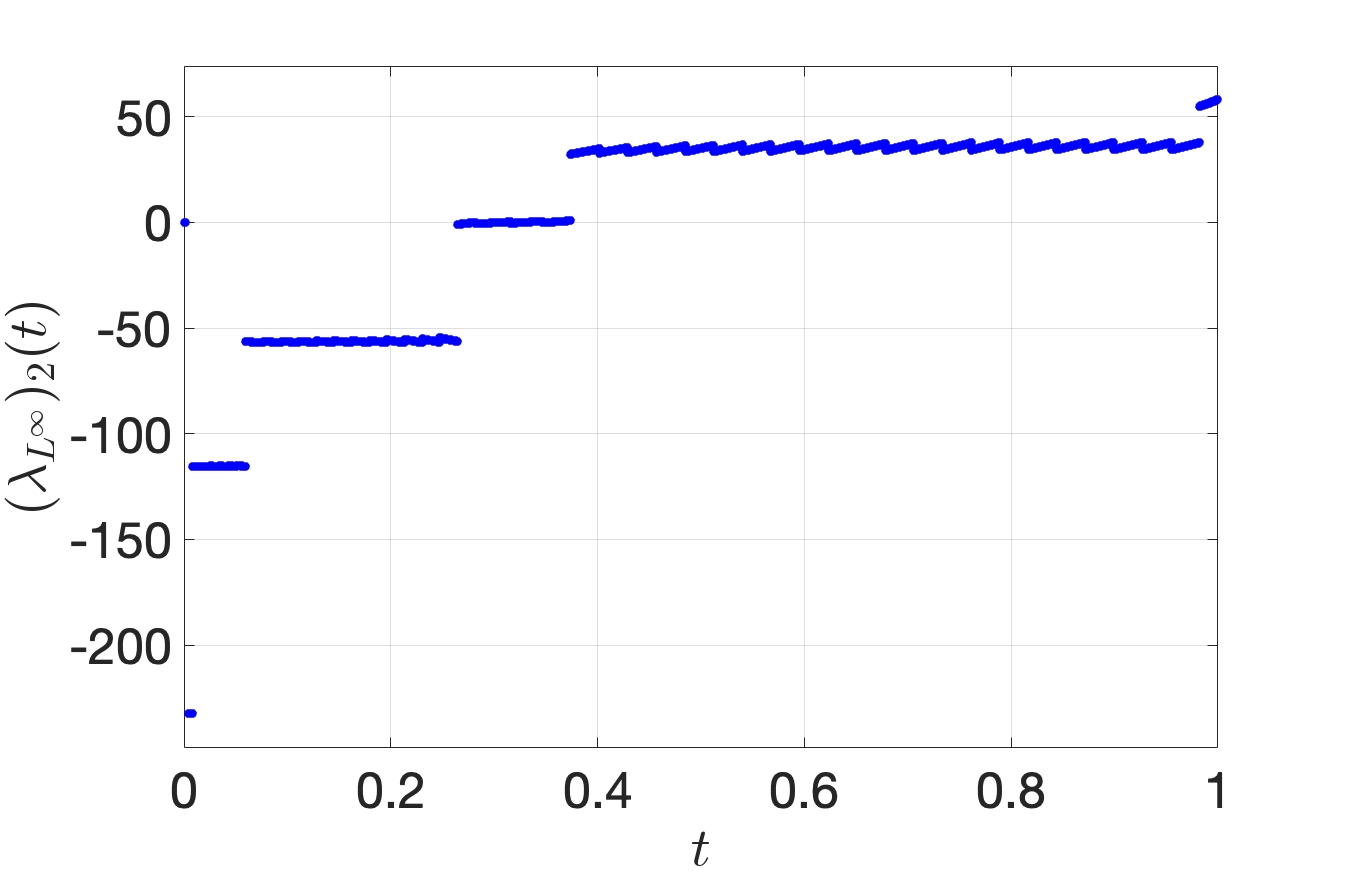} \\[3mm]
(d)
\end{center}
\end{minipage}
\
\caption{\sf Example 3 -- (a)--(b) control variable vector and (c)--(d) costate variable vector components, all for equation~\eqref{ode:vdp}, with $0\le t\le1$, $x(0)=(-1,-3)$, and {\sc Matlab}'s {\tt ode113}.}
\label{fig:ex3f}
\end{figure}

\newpage
\section{Discussion of Results\label{sec:discussion}}
In all the examples using \texttt{ode45}, we see that the  minimizing residuals (either $L^2$ or $L^\infty$) are never larger than those produced by the polynomial interpolants used by \texttt{deval}, and are in some cases a hundred times smaller. This independent confirmation of the quality of the computed skeleton of the solution to the differential equation provides a useful reassurance that the code has done its job. In the case of \texttt{ode113} and \texttt{ode15s}, where the polynomial approximations are not continuous, the residuals computed thereby are not reliable; the examples here show that (except when \texttt{ode113} accidentally got the exact solution) the true residual can be larger than that computed using polynomials.

The fact that the optimal residuals are not continuous is perhaps disconcerting.  They come from a piecewise continuous, but not piecewise continuously differentiable, interpolant.  We point out that requiring the interpolant to be continuously differentiable is usually not possible: if the interpolant is continuously differentiable, then it will almost always be possible to find another interpolant with a slightly smaller residual. That is, the solution to the minimization problem will not actually \textsl{exist}, if we insist on too much continuity. This is a familiar fact, discussed in many textbooks, for example~\cite{young1969lectures}.  Nonetheless, the modeller may not be expecting discontinuous inputs to the right-hand side of the model differential equations, and be happier with the larger residuals produced by smoother interpolants.

A further refinement of this idea is possible, as mentioned in~\cite{CorJan2016}: one might compute a ``modified equation'' to account for the largest part of the residual, and then use the techniques of this present paper to find the best interpolant that fits that modified equation.  If we follow that process, the discontinuous residual will typically be an order of magnitude smaller.

For clarity, here is an example.  Suppose we are solving $\dot y = F_f(y)$ by using the parameterized 2nd order Runge--Kutta method
\begin{align*} \begin{array}{c|cc}
     0\; \\
     \sfrac{1}{2a}\; & \;\sfrac{1}{2a}\; \\
     \hline
         & 1-a & a
   \end{array} \end{align*}
This is a 2nd order Runge--Kutta method for any choice of $a \ne 0$.

Using the \texttt{BSeries.jl} package we mentioned previously, the modified equation is (to 2nd order) computed as follows.
\begin{lstlisting}
using BSeries
using Latexify
import SymPyPythonCall; sp = SymPyPythonCall;
a = sp.symbols("a", real = true)
A = [0 0; 1/(2*a) 0]; b = [1-a, a]; c = [0, 1/(2*a)]
coeffs2 = bseries(A, b, c, 3)
meq = modified_equation(coeffs2)
\end{lstlisting}
The output of the above equation has to be divided by $h$ (because of the conventions of the code normalization) in order to construct the right-hand side of the modified equation below.  Again, we use a common notation for elementary differentials using trees.
\begin{equation}
    \dot y = F_{f}\mathopen{}\left( \rootedtree[.] \right)\mathclose{} - \frac{ 1 }{6} h^{2} F_{f}\mathopen{}\left( \rootedtree[.[.[.]]] \right)\mathclose{} + h^{2} \left( \frac{-1}{6} + \frac{1}{8 a} \right) F_{f}\mathopen{}\left( \rootedtree[.[.][.]] \right)\mathclose{} + O(h^3)
    \label{eq:Ralstonmodified}
\end{equation}
If we take the skeleton and minimize the residual not for $\dot y = F_f$ but rather in the above equation~\eqref{eq:Ralstonmodified}, then we would find that the minimal residual would be $O(h^3)$ in size rather than $O(h^2)$.  That is, the largest part of the residual would be explained by the method of modified equations, and only a minor part would contain discontinuities (and its magnitude, if small, would be a guarantee of fidelity to the model).

One is tempted to choose $a=3/4$ (this is \textsl{Ralston's method}) which removes one term from the modified equation above, because that would make the example simpler.  But the purpose of this paper is not to \textsl{improve} numerical methods but rather to provide tools from control theory that would \textsl{explain their success}. For that, the choice of parameter makes no difference.
\section{The Lorenz Equations Again}
If we use the explicit trapezoidal rule, as Lorenz did, to solve the Lorenz equations on (say) $0 \le t \le T = 27.5$ which is long enough so that the solution is appreciably in the attracting set, we can look at a \textsl{sample} residual, say on a single time-step $[T-h, T]$.  Because the method is only second order accurate, the cubic Hermite interpolant that fits the known values $y_{N-1}$ and $y_N$ at $t_{N-1} = T-h$ and $t_N = T$, and also fits the known derivatives $F_f(y_{N-1})$ and $F_f(y_{N})$ there, is accurate enough to give a decent estimate of the residual.
\begin{figure}
    \centering
    \subfigure[Final timestep $N=10^4$]{\includegraphics[width=0.45\linewidth]{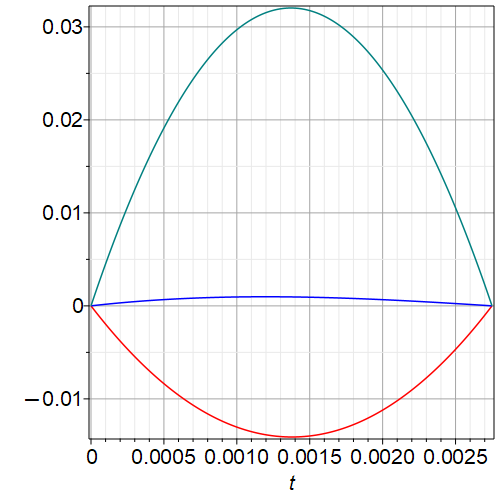}}
    \subfigure[Residual accuracy]{\includegraphics[width=0.45\linewidth]{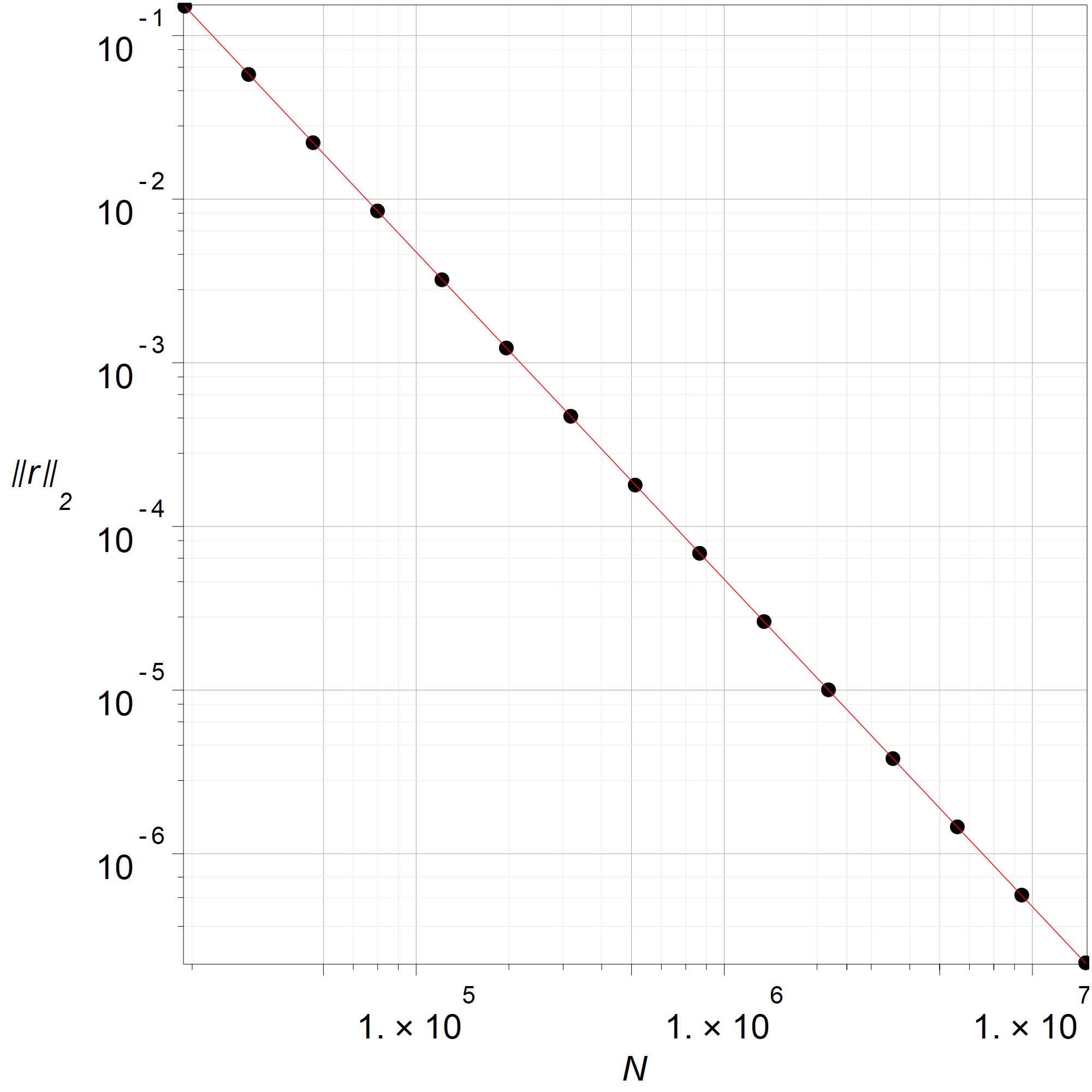}}
    \caption{(left) The residual using cubic Hermite interpolation to the solution of the Lorenz system with $\sigma = 10 $, $\rho = 28$, and $\beta = 8/3$, and a time-step $h=0.00275$ (or $10,000$ steps), on the final subinterval $[T-h,T]$ where $T=27.5$ (the axis is labelled, equivalently, as $[0,h]$). The residual $r_x = \dot{x}- \sigma(y-x)$ for $x$ is in red, the residual $r_y$ is in blue, and the residual $r_z$ is teal-coloured. The maximum value of $\sqrt{r_x^2 + r_y^2 + r_z^2}$ occurs near $t= T-h/2$ and is approximately $0.035$. (right) The maximum $2$-norm $\sqrt{r_x^2 + r_y^2 + r_z^2}$ of the sampled residuals over $0 \le t \le T$ for various $h=T/N$. Samples were taken at $t = (k+1/2)h$ for $0 \le k \le N-1$, and the maximum was recorded for each $N$. The red line is $C/N^2$ for $C = 4.75\times 10^7$, which reflects the size of the second derivatives of the Lorenz equations.}
    \label{fig:Lorenzresidualoneinterval}
\end{figure}
We see in Figure~\ref{fig:Lorenzresidualoneinterval} that the maximum residual with this interpolant occurs very nearly in the middle of the interval.  We can sample that and compute the norm (say the $2$-norm) of all three residuals, and do so for a sequence of decreasing $h$.  When we do that---see Figure~\ref{fig:Lorenzresidualoneinterval}---we find that as expected the size of the residual decays like $C/N^2 = O(h^2)$, although the constant $C$ is quite large, being about $4.75\times 10^7$. This is, however, simply reflective of the size of the second derivative terms in equation~\eqref{eq:modifiedLorenz}.

With the cubic Hermite interpolant, we get an upper bound for the minimal residual.  The tools of this paper offer a way to find a minimal residual. When we do that for the interval shown in the left side of Figure~\ref{fig:Lorenzresidualoneinterval}, we find that the $2$-norm $\sqrt{r_x^2 + r_y^2 + r_z^2}$ of the minimal residual is nearly constant across the interval, and is approximately $0.023$, which is about $66\%$ of the residual shown by the cubic Hermite interpolant.  Because it is a \textsl{minimal} residual, of course it must be smaller.  That the residual from the cubic Hermite interpolant is less than twice as large as the minimum possible shows that, for this method and this problem, cubic Hermite interpolants do quite a good job of assessing the quality of the numerical solution of the Lorenz equations by the trapezoidal rule.
It's a pity that cubic Hermite interpolants have too low an order to be useful for higher-order methods than $2$nd order.  There is a generalization to higher-order piecewise Hermite interpolation, which we call ``blends,'' that may be interesting for the purpose~\cite{Corless2023,corless2023ISSAC,corless2024hermite}, and this could be investigated in future work.

If we took the modified equation~\eqref{eq:modifiedLorenz} into account, we could show that the known $O(h^2)$ terms account for the majority of the residual. We could then bound the remaining $\delta(t)$ in equation~\eqref{eq:modifiedLorenz}. This might be important, depending on the context of the particular problem being solved.  See Figure~\ref{fig:Lorenzresidual} where we plot the three components of $r(t) = \dot z - F_f(z)$ against each other. The very clear correlation of the residual with the solution raises the spectre of feedback and resonance, and, again, depending on the context of the problem, one might wish reassurance that the potential resonance was completely understood.

The Maple worksheet where we computed these can be found at
\url{https://maple.cloud/app/4871518686674944/Solving+the+Lorenz+Equations+in+Maple}.
We computed the derivative of the cubic Hermite interpolant of each component $x(t)$, $y(t)$, and $z(t)$, and put $r_x(t) = \dot{x}(t) - \sigma(y(t)-x(t))$, $r_y(t) = \dot{y}(t) - (x(t)(\rho-z(t))-y(t))$, and $r_z(t) = \dot{z}(t) - (x(t)y(t)-\beta z(t))$.  For the left-hand side of Figure~\ref{fig:Lorenzresidualoneinterval} we plotted each of these across the interval $[T-h,T]$.  For the right hand side of that figure and for Figure~\ref{fig:Lorenzresidual} we only sampled the residual at one point in each timestep $[t_n, t_{n+1}]$, namely the half-way point. That is, we evaluated cubic Hermite interpolants to $x(t)$, $y(t)$, and $z(t)$ at $t = t_n + h/2$ on each step, together with their derivatives $\dot{x}(t)$, $\dot{y}(t)$, and $\dot{z}(t)$, which allowed evaluation of the residual close to where it is of maximum size for the cubic Hermite interpolant. We then plotted all the samples. For the right-hand side of Figure~\ref{fig:Lorenzresidualoneinterval} we computed the maximum 2-norm of the sampled residuals $\sqrt{ r_x^2 + r_y^2 + r_z^2}$ on $0 \le t \le T$ and recorded that, for runs with $N = F_{22}=17,711$, $F_{23} = 28,657$, $F_{24} = 46,368$, and so on up to $N=F_{36} = 14,930,352$.
The numbers $F_k$ are
Fibonacci numbers, which make for nice spacing on a log-log plot. That the residuals decay as $N$ increases indicates that we are solving systems closer and closer to the Lorenz system.  Of course, since the system is chaotic, the individual trajectories do not converge (except in a very short initial time interval).

In the interest of convenience for the reader, the cubic Hermite interpolant to a function taking on values $y_n$ and $y_{n+1}$ at each end of a step of width $h$, with derivative values $f(y_n)$ and $f(y_{n+1})$ at the corresponding ends, is
\begin{equation}
\left(\left(1+2 s \right)y_n +s h f(y_n) \right) \left(1-s \right)^{2}+\left(\left(3-2 s \right)y_{n+1} -\left(1-s \right) f(y_{n+1}) h \right) s^{2}
\end{equation}
where $s=(t-t_n)/h$.  From here one gets the simple formulae for the value of the function and derivative at the half-way point, which one can use to estimate the residual.
\begin{align}
    y_{n+1/2} &= \frac{1}{2}\left( y_n + y_{n+1} \right) - \frac{h}{8}\left( f(y_{n+1}) - f(y_n) \right) \\
    \dot{y}_{n+1/2} &= \frac32\left( \frac{y_{n+1}-y_n}{h}\right) - \frac14\left( f(y_n) + f(y_{n+1}) \right)\>.
\end{align}

What we get here is \textsl{not} the minimal residual advocated in this paper, but rather an upper bound for that minimal residual\footnote{This upper bound is relatively simple to compute, because the method is such a low-order method.  Cubic Hermite interpolation requires no extra function evaluations.  Higher-order methods, and more sophisticated methods with many heuristics and ``magic constants'' would mandate more effort, such as is advocated in this paper.}.  It obviously would require more effort to compute such a minimal residual, and the value of that effort would strongly depend on the context of the problem.  What we wished to convey by this example is that it's believable that for more complicated problems such effort could well be worthwhile.

\begin{figure}
    \centering
    \includegraphics[width=0.5\linewidth]{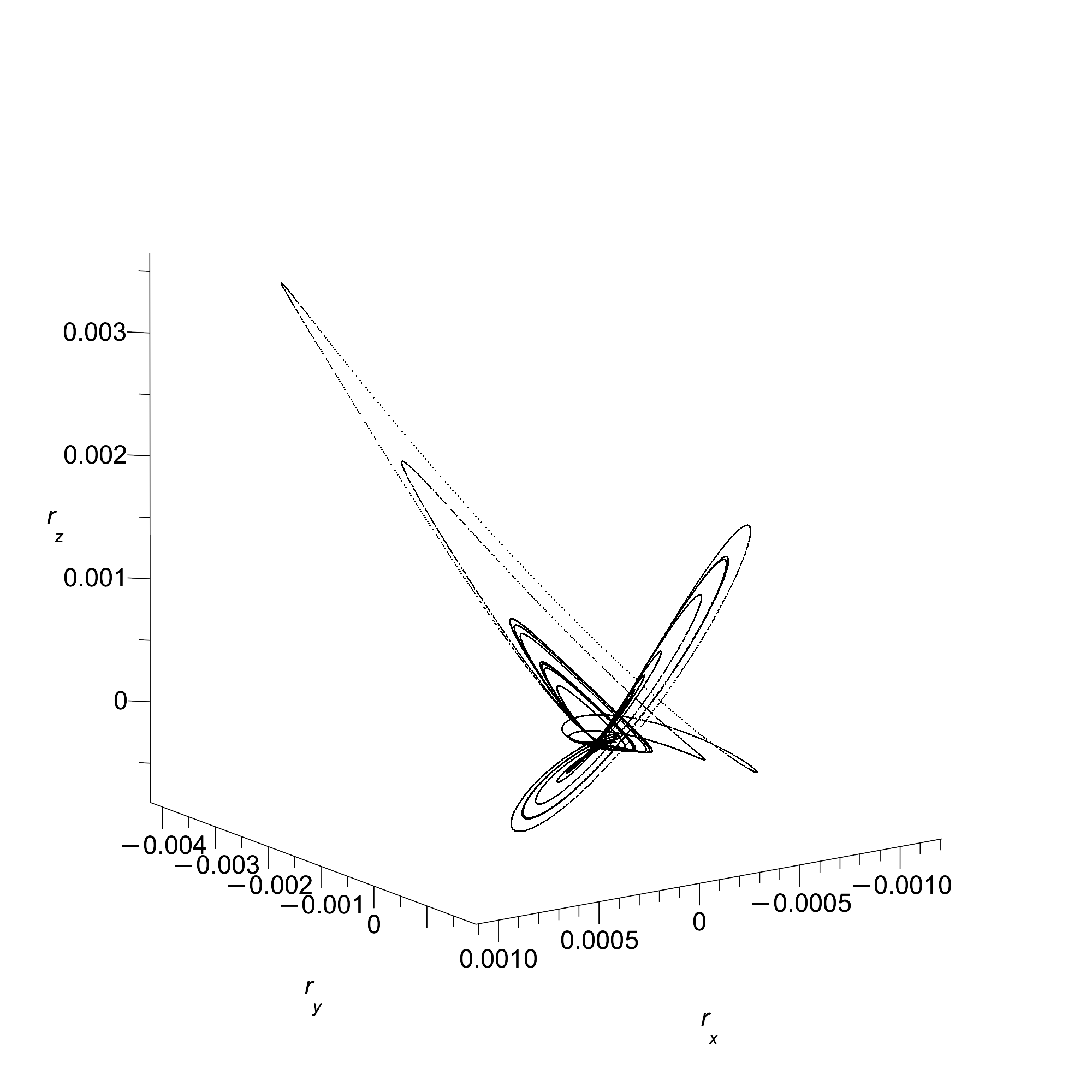}
    \caption{The residuals sampled at the half-step, with $N=10^5$ steps taken on $0 \le t \le 27.5$, of the explicit trapezoidal rule on the Lorenz equations.  We see a clear correlation with the solution of the Lorenz equations.}
    \label{fig:Lorenzresidual}
\end{figure}
\section{Concluding Remarks}
\label{sec:conclusion}

We passed several remarks about so-called ``na{\"\i}ve users'' of modern codes for solving IVPs for ODEs numerically, saying that the users couldn't be expected to understand just what the codes were doing.  This is of course an oversimplification; yes, modern users are busy people and focused on what the results of their simulations mean, rather than on the intricacies of the basic and well-tested codes they are using in their simulations. Yes, they should know how to use modern codes, which are excellent and efficient.  See in particular the Julia package by Rackauckas and Nie~\cite{rackauckas2017differentialequations}, which we find especially impressive. But it seems somehow unfair to force users to learn \textsl{all} the details of how the codes work: the codes should just work as the users expect them to.

Unfortunately, the codes do \textsl{not} always work as users expect them to, even though the documentation is usually clear.  {One issue is, we believe, that \textsl{local error} and its relation to \textsl{forward error} (also known as ``global error'') is not as widely understood as it should be. Another and perhaps more significant issue is that the theory of error is true only asymptotically as the (mean) stepsize $h$ goes to zero, while for efficiency the code tries to make $h$ as large as possible consistent with the tolerances.  This antinomy is managed by heuristics and ``magic constants'' in many codes.}
A further issue is that the language used to discuss the polynomial approximations provide by \texttt{deval} sometimes uses the words ``interpolant'' and sometimes ``continuous extension.''  These words are sometimes incorrect, strictly speaking, and can be misleading.  The words ``dense output'' is sometimes used, and this is more accurate, but its subtleties could easily be overlooked.

In cases where lives or a lot of money are involved, and greater surety of the correctness of the simulation is needed, it is incumbent on the user to be aware of that fact, and to check the results.  There are many tools already in existence for doing so.  Examples include using more reliable (continuously differentiable) interpolants~\cite{EnrJacNorTho1986} and computing residuals using those\footnote{Re-reading that paper, we see that the task we set ourselves in this present paper is at least forty years old.  Quoting from their section 5, Applications of Interpolation: ``Interpolants can also be employed to improve the reliability of a code. One
way this can be done is to use an interpolant to produce an estimate of the defect
$z'(x) - f(x, z(x))$ in the numerical solution, where $z$ interpolates the computed
values $\{y_n\}$. The size of the defect is frequently a good way to characterize the
accuracy of the numerical solution.
Ideally, we would like to choose $z$ to minimize some norm of the defect.
Although none of the interpolants discussed in this paper satisfy this ideal (as
far as we know), the defect associated with the higher-order interpolant $z_n^p$ is
generally not much larger than the minimum.''}. An alternative would be to use a different method, which controlled the residual directly~\cite{Enright2000}.  This paper provides even more tools, which may be familiar to people who know some control theory.

To be explicit, the methods of this paper can
\begin{enumerate}
    \item identify when the dense output supplied by the code gives too \textsl{large} a residual (e.g. for \lstinline{ode45})
    \item identify when the dense output supplied by the code gives too \textsl{small} an estimate for the residual (e.g. for \lstinline{ode113})
    \item Identify when a custom interpolant (e.g. cubic Hermite interpolation for the trapezoidal rule) does a good job
    \item supply reassurance by sampling residuals at strategic places that the code has done a good job of giving the exact solution to a nearby problem.
\end{enumerate}
One reason to want such reassurance is that modern codes
can be quite complex, and the interaction of high-order interpolants, stepsize control methods, and hard-coded limits on such controls, is nonlinear and can be difficult to analyze on a large and complex problem.
We admit that some of the value of the methods of this paper would come from the user's familiarity with the methods of control theory, and indeed carrying out the program may require access to commercial optimization tools such as AMPL.

The method we advocate here will guarantee that the code has solved a ``nearby problem'' and quantify just how near.  Of course, one has to understand the effects of changing the problem to a nearby one, perhaps by using the Gr\"obner--Alexeev nonlinear variation of constants formula, or studying the sensitivity of the model to changes in another way, but this has to be done even if one has the exact analytical solution.  See~\cite{CorFil2013} for more discussion of this point.

We re-iterate that we used the stage $L^\infty$-norm mainly because we were able to obtain analytical expressions for the residual in Examples 1 and 2.  With the true $L^\infty$-norm over the whole interval $[t_0,t_f]$, this advantage is lost, and the residual can only be obtained numerically.   The multi-stage or multi-process optimal control setting we have used in the present paper will already facilitate the means to compute the residual.  The (true) $L^\infty$-minimization should be expected to result in residuals which are better than those obtained by the stage $L^\infty$-minimization.

{\color{black} $L^\infty$-minimization has resulted in numerical experiments where there appears to be at most one switching in each optimal control stage. Therefore, another improvement in stage $L^\infty$-minimization could be achieved by parametrizing the problem with respect to switching times and thus computing the residuals more accurately---see~\cite{KayNoaSch2024} and the references therein.  This would also allow one to use higher order discretization schemes in intervals between consecutive switchings, as in~\cite{KayNoaSch2024}.}

As a final remark, we see in~\cite{Corless1994} (a paper from more than thirty years ago) a similar idea in embryo.  Only $L^2$ minimization was considered in that paper, and only sketched for a simple example.  But in one respect that paper advocates something that we have not done here, namely to include other terms in the equation which are motivated physically in order to interpret some aspects of the numerical solution.  The example used there involved numerical solution of the simple harmonic oscillator $\ddot y + y = 0$ and in the analysis thereof included an artificial damping term $2\zeta \dot y$ with an unknown damping $\zeta$. The paper suggested minimizing the residual with an optimal $\zeta$ in order to interpret the numerics.  We note that this can now be done more generally using the methods of this present paper.

\textbf{Acknowledgements}
We are grateful to David Ketcheson and Hendrik Ranocha for help with the B-series package.  We thank the author of~\cite{Stepanov2022} for putting the \LaTeX\ source for that paper on arXiv, which allowed us to avoid much of the effort of constructing the Butcher tableau for the Dormand--Prince RK (5,4) pair for ourselves. This, in addition for the work of the paper itself which we found helpful. We also thank a sharp-eyed referee who caught an omission (and a duplicate row) in our printing of that tableau, even so.  This work was partially supported by NSERC grant RGPIN-2020-06438.


\end{document}